\documentclass[11pt,reqno]{article}   %style-phys/cs 
\usepackage{fullpage}

\usepackage{microtype}
\usepackage[unicode=true]{hyperref}
\usepackage[british]{babel}
\usepackage{amsmath}
\usepackage{cite}
\usepackage{amsfonts}
\usepackage{amssymb}
\usepackage{amsthm}
\usepackage{cleveref}  
\usepackage{authblk}
\usepackage{mypack} %custom commands
\usepackage[pdftex]{color,graphicx}
\usepackage{adjustbox}
\usepackage{soul}
\usepackage{comment}
\usepackage{bbold}
\usepackage{tikz} %for diagrams
\usetikzlibrary{decorations.pathreplacing,angles,quotes}
\usetikzlibrary{matrix,arrows,decorations.pathmorphing}
\usepackage{tikz-cd}
\usetikzlibrary{arrows}
\usetikzlibrary{decorations.markings}

\usetikzlibrary{arrows.meta,decorations.markings}

\usepackage{bbold}
\usepackage{diagbox}
\usepackage{caption}
\usepackage{subcaption}
\usepackage{pdfpages} 
\usepackage{blkarray}
\usepackage{centernot}
\usepackage{mathtools}
\usepackage{stmaryrd}
\usepackage{soul}  %strikethrough text: \st{}
\usepackage{enumitem}

\def\on{\operatorname}

\definecolor{bluegray}{rgb}{0.4, 0.6, 0.8}
\definecolor{turquoise}{rgb}{0.2, 0.7, 0.6}
\usepackage{todonotes}
\newcounter{commcount}

\newcommand{\catCat}{\mathbf{Cat}}
\newcommand{\catPSh}{\mathbf{PSh}}
\newcommand{\catRel}{\mathbf{Rel}}

\title{{Possibilistic collapse and extremality of simplicial distributions}} 
 
\author{Aziz Kharoof\footnote{aziz.kharoof@bilkent.edu.tr} }
\author{Cihan Okay\footnote{cihan.okay@bilkent.edu.tr}} 
\affil{{\small{Department of Mathematics, Bilkent University, Ankara, Turkey}}}

\date{\today}

\begin{document}

\maketitle

\begin{abstract} 
{
Consistent families of locally defined probability distributions that do not admit a joint global distribution are known as contextual, with primary examples arising in quantum theory. In this paper, we study such families of distributions using the theory of simplicial distributions, and further develop the theory for possibilistic distributions defined over the Boolean semiring. We characterize possibilistic collapses of simplicial distributions geometrically using bundle scenarios. Using this characterization together with a new connectivity condition on the total space of a bundle scenario, we provide a criterion for detecting extremal simplicial distributions. In parallel, we develop an analogous theory for presheaves on simplicial complexes, describe possibilistic collapses of empirical models on them using event scenarios together with a categorical extremality condition, and relate the two frameworks via a comparison result. We provide examples of contextual simplicial distributions that arise from our criteria on scenarios of interest in quantum foundations, such as Bell scenarios and boundaries of standard simplices, the latter connecting to Vorob'ev's classical theorem on acyclic complexes.
}
\end{abstract}

\tableofcontents

\section{Introduction}

{ 
Contextuality is a fundamental feature of probabilistic models arising in quantum theory, capturing the impossibility of extending a consistent family of locally defined probability distributions to a global probability distribution \cite{bell64,KS67,brunner2014bell,budroni2022kochen}. Over the past decade, several mathematical frameworks have been developed to formalize and study contextuality, most notably the sheaf-theoretic approach of Abramsky and Brandenburger~\cite{abramsky2011sheaf} and, more recently, the simplicial set approach of the current authors~\cite{okay2022simplicial}. In this paper, we develop the theory in the possibilistic case and provide topological and categorical conditions for detecting extremal distributions.
}

{ 
In the sheaf-theoretic approach of \cite{abramsky2011sheaf}, contextuality is described using empirical models defined on measurement scenarios. In this framework, a measurement scenario is described by a simplicial complex $\Sigma$ together with a family of outcome sets $O_x$ for each $x$ in the vertex set $V$. Here, the vertex set $V$ represents the set of all measurements that can be performed, and the simplices $\sigma\in\Sigma$ represent the subsets of measurements that can be jointly performed. Each measurement $x\in V$ can take an outcome in a set $O_x$. An empirical model then consists of probability distributions $p_\sigma$ on the products $\prod_{x\in \sigma} O_x$ that are compatible under inclusions of simplices. In this work, we work with arbitrary presheaves on a simplicial complex, i.e., functors of the form
\[
F\colon \catC_{\Sigma}^{\op} \to \catSet .
\]
Here, the domain category is the opposite of the simplex poset of the complex. We define empirical models, denoted by $\Emp(F)$, as compatible families of distributions on $F(\sigma)$ under inclusions of simplices. An empirical model is called contextual if it does not arise as the marginals of a distribution on the limit of $F$. In the case where outcomes are given as products of the $O_x$, such a distribution would be defined on the set $\prod_{x\in V} O_x$ of all possible outcomes.
}

{
More recently, a more general perspective on contextuality was introduced in \cite{okay2022simplicial,barbosa2023bundle}, based on the theory of simplicial distributions. In this approach, a measurement scenario is represented by a simplicial map
\[
f\colon E \to X.
\]
Here, $X$ represents the simplicial set of measurements, and for an $n$-simplex $x\in X_n$, the fiber $f_n^{-1}(x)$ represents the possible outcomes for this measurement. A simplicial distribution $p$ consists of a family of probability distributions $p_x$, one for each simplex \(x \in X_n\), on the finite fiber \(f_n^{-1}(x)\), subject to compatibility conditions with respect to the face and degeneracy maps of the simplicial sets \(E\) and \(X\). 
We denote by $\sDist(f)$ the set of all simplicial distributions on $f$. A simplicial distribution $p$ is called contextual if it does not arise as a probabilistic mixture of the sections of $f$. 
}

{
The sheaf-theoretic perspective based on empirical models embeds into the theory of simplicial distributions. We construct a realization functor that sends a presheaf $F$ over a simplicial complex $\Sigma$ to a simplicial map $f_F:E(F)\to S(\Sigma)$ and, moreover, show that there is an affine isomorphism 
\[
\on{Emp}(F) \to \on{sDist}(f_F)
\]
from the convex set of empirical models on $F$ to the convex set of simplicial distributions on $f_F$. This construction is very useful for importing results from simplicial distributions to empirical models. Moreover, it is natural with respect to morphisms that amount to changing the measurements and outcomes. Such morphisms are best expressed using a relative version of the Grothendieck construction, which captures the $2$-categorical data leading to categories of scenarios. Loosely speaking, a morphism between two scenarios, expressed as simplicial maps, amounts to changing the measurements and outcomes, that is, the base and the fibers. Simplicial distributions, and similarly empirical models, can be regarded as functors on these categories of scenarios. Our embedding gives a natural isomorphism between these functors of empirical models and simplicial distributions. These categories of scenarios were introduced in \cite{barbosa2023bundle}, together with a different realization construction, and the relative Grothendieck point of view was introduced in \cite{kharoof2025simplicial}.
}

{
Under mild finiteness conditions on $X$ and the fibers of $f$, the set of simplicial distributions forms a convex polytope, that is, a convex set with finitely many extremal points, or vertices. Describing the extremal points of probability polytopes that arise in quantum foundations is an important problem \cite{pitowsky1989quantum,pr94,barrett2005nonlocal,jones2005interconversion}.
 Recently, simplicial distributions have proved useful in addressing this problem \cite{kharoof2023homotopical,
kharoof2024extremal,kharoof2026geometry,
kharoof2026vertex}. A well-known approach, going back to \cite{abramsky2012logical,abramsky2015paradox,abramsky2016possibilities}, is to use the semiring homomorphism $\RR_{\geq 0}\to \BB$ from the nonnegative reals to the Booleans. Using special types of simplicial maps, we characterize the possibilistic collapse of simplicial distributions:
\[
\kappa:\on{sDist}(f) \to \on{sDist}_\BB(f).
\]
Note that simplicial distributions can be defined over any semiring. 
The key definition is the notion of a bundle scenario: a surjective simplicial map $f:E\to X$ is called a bundle scenario if it satisfies the left lifting condition with respect to every simplicial map $\theta:\Delta[n]\to \Delta[m]$ between standard simplices. Writing $\on{Sub}(f)$ for the simplicial subsets $E'\subset E$ such that the restriction $f|_{E'}$ is a bundle scenario with finite bundles, we show that there is a natural bijection
\[
\on{sDist}_\BB(f) \cong  \on{Sub}(f).
\]
Again, this construction is natural with respect to morphisms of scenarios and lifts to a natural isomorphism between the relative Grothendieck constructions (Theorem~\ref{thm:possibilistic sdist as sub}). There is a parallel construction for possibilistic collapses of empirical models. In this case, special types of presheaves, which we call event scenarios, are used (Theorem~\ref{cor:EmpB=Esub}).
}

{
Using our characterization of possibilistic collapses of simplicial distributions in terms of bundle scenarios, we prove a topological criterion for a simplicial distribution to be a vertex (Theorem~\ref{thm:strongisvertx}). This extremality criterion is called the strong connectivity condition on the total space $E$.  

\begin{thm*}
Let $p$ be a simplicial distribution on a simplicial map $f\colon E\to X$, and let $g$ denote the bundle scenario corresponding to $\kappa(p)$, the possibilistic collapse of $p$. If $g$ is strongly connected, then $p$ is a vertex of $\sDist(f)$.
\end{thm*}

We also introduce an analogous categorical condition for presheaves on simplicial complexes (Theorem~\ref{thm:1vert}), which remains applicable in some cases where the topological condition fails. We show that the topological condition implies the categorical one (Proposition~\ref{pro:strongconcimplieszigzag}), so the latter detects a strictly larger class of extremal distributions, though neither is a complete characterization of extremality.

Using our extremality criterion, we exhibit new contextual simplicial distributions $p:X\to \Delta_{\ZZ_2}$ when $X$ is given by a specific triangulation of the $n$-dimensional sphere $S^n$. Our first result concerns the case $X=\partial\Delta^n$, the boundary of the standard $n$-simplex (Proposition~\ref{pro:vertboundrysimpled}), and the second result concerns the case $X=B(n,2)$, the $n$-fold join of $S^0$ (Proposition~\ref{pro:vertex(n,2,2)}).  
The first scenario plays a key role in the famous theorem of Vorob'ev \cite{vorob1962consistent}, which characterizes acyclic simplicial complexes as those that do not admit a contextual empirical model.
 The second result models Bell scenarios in quantum foundations \cite{brunner2014bell}, which are fundamental examples of where quantum contextuality arises.
}

{
Our paper is organized as follows. In Section~\ref{sec:Preliminaries}, we introduce the necessary preliminaries on simplicial distributions and empirical models, and show how the former theory embeds into the latter. We also discuss the relative Grothendieck construction in relation to the naturality properties of these theories. Section~\ref{sec:Possibilistic} focuses on the identification of possibilistic collapses of simplicial distributions as sub-bundle scenarios, together with their empirical-model counterpart. In Section~\ref{sec:Sufficient conditions for extremality}, we introduce our topological and categorical criteria, which imply extremality of simplicial distributions and empirical models, respectively. Finally, in Section~\ref{sec:Examples}, we discuss two triangulations of $n$-spheres as base spaces, together with other important examples from the quantum foundations literature. We also provide examples comparing our topological and categorical criteria.
}

\paragraph{Acknowledgments.}
This work is supported by the Air Force Office of Scientific Research (AFOSR) under award number  FA9550-24-1-0257.

\section{Preliminaries}
\label{sec:Preliminaries}

{In this section, we review two main approaches to the study of joint probabilities and contextuality:
\begin{enumerate}
\item A simplicial-set approach based on simplicial distributions on simplicial maps~\cite{okay2022simplicial,barbosa2023bundle}.
\item A sheaf-theoretic approach based on empirical models on presheaves~\cite{abramsky2011sheaf,kharoof2025simplicial}.
\end{enumerate}
We show how the second approach embeds into the first. In addition, we describe the functoriality of these constructions.}

\subsection{Simplicial distributions}

Simplicial distributions, introduced in \cite{okay2022simplicial,kharoof2022simplicial}, combine simplicial sets from algebraic topology with probability distributions. We begin by introducing distributions \cite{jacobs2010convexity}.

\begin{defn}\label{def:DDDDD}
Let $R$ be a semiring.
The \emph{distribution functor} 
%is {the} functor 
$D_{R}\colon  \catSet \to \catSet$ is defined as follows:
\begin{itemize}
    \item For a set $X$, the set 
    %$D_{R}(X)$ 
    of distributions on $X$ is defined by
    \[
   D_{R}(X)=\set{P\colon X \to R \mid~ |\set{x\in X\mid~P(x)\neq 0}|<\infty \;\; \text{and}\;\; \sum_{x\in X}P(x)=1}.
    \]
    \item For a map $f\colon X \to Y$, the map $D_{R}(f)\colon D_{R}(X) \to D_{R}(Y)$ is given by
    \[
    D_{R}(f)(P)(y) = \sum_{x \in f^{-1}(y)} P(x).
    \]
\end{itemize}
\end{defn}

%\coc{$\delta^x$ conflicts with the $P^x$}

The distribution functor is in fact a monad.
The unit of this monad, $\delta_X\colon X \xhookrightarrow{} D_{R}(X)$, sends each $x \in X$ 
to the \emph{delta distribution} $\delta_x$, defined by
\begin{equation}\label{eq:deltax}
\delta_x(x') =
\begin{cases}
1, & x' = x,\\
0, & \text{otherwise.}
\end{cases}
\end{equation}
The algebras over the distribution monad, called \emph{$R$-convex sets}, generalize the usual theory of convexity to arbitrary semirings {(see \cite[Theorem 4]{jacobs2010convexity}).}

When {$R$ is the semiring $\RR_{\geq 0}$ of non-negative real numbers}, we write $D$ instead of $D_{\RR_{\geq 0}}$ and refer to it as the \emph{probability distribution monad}. The other semiring of interest is the Boolean semiring $\BB={0,1}$, which gives rise to the \emph{possibilistic distribution monad}.

There is a map
\begin{equation}\label{eq:piX}
\kappa_X\colon D(X) \to D_{\BB}(X)
\end{equation}
induced by the semiring homomorphism $\RR_{\geq 0} \to \BB$ that sends $0$ to $0$ and every non-zero real number to $1$.

Next, we introduce simplicial sets \cite{friedman2008elementary,goerss2009simplicial}.

\begin{defn}
Let $\Delta$ denote the \emph{simplex category}, whose objects are finite ordered sets
\[
[n] = \{0 < 1 < \dots < n\}, \quad n \ge 0,
\]
and whose morphisms are order-preserving maps, called \emph{ordinal maps}.
A \emph{simplicial set} is a functor
\[
X \colon \Delta^{\op} \to \catSet,
\]
from the opposite of the simplex category.

For each $n \ge 0$, the set $X_n := X([n])$ is called the set of \emph{$n$-simplices} of $X$.  
The images of the coface maps $d^i$ and codegeneracy maps $s^j$ in $\Delta$ under $X$ are called the \emph{face maps} and \emph{degeneracy maps}, respectively:
\[
d_i {:=X(d^{i})} \colon X_n \to X_{n-1}, \qquad s_j {:=X(s^{j})} \colon X_n \to X_{n+1},
\]
and satisfy the standard simplicial identities.
A simplex $\sigma \in X_n$ is called a \emph{generator} if it is not in the image of any face or degeneracy map.  

A \emph{simplicial set map} (or \emph{simplicial map}) between simplicial sets
$$
f \colon X \to Y
$$
is a natural transformation between the corresponding functors $X, Y \colon \Delta^{\op} \to \catSet$.  
Equivalently, it consists of a family of functions
$
\set{f_n \colon X_n \to Y_n}_{n \geq 0}
$
that commute with all face and degeneracy maps of $X$ and $Y$. 
For $x\in X_n$, we will usually write $f_x$ to denote the value of $f_n(x)$.

Simplicial sets with simplicial maps can be assembled into a category denoted by $\catsSet$.
\end{defn}

The distribution monad $D_R$ lifts to a monad on the category of simplicial sets \cite{kharoof2022simplicial}. For a simplicial set $X$, we write $D_R(X)$ for the functor obtained by composing $X\colon \catDelta^\op \to \catSet$ with the distribution monad $D_R\colon \catSet \to \catSet$.

\begin{defn}\label{def:simp-dist}
Let $f \colon E \to X$ be a simplicial map.
An \emph{$R$-simplicial distribution} on $f$ 
is a simplicial set map $p \colon X \to D_{R}(E)$ that makes the following diagram commute:
\begin{equation}\label{dia:simp-dist-revisited}
\begin{tikzcd}[column sep=huge,row sep=large]
& D_{R}(E) \arrow[d,"{D_{R}(f)}"] \\
X \arrow[ur,"p"] \arrow[r,"\delta_X"'] & D_{R}(X)
\end{tikzcd}
\end{equation}
We will write $\on{sDist}_R(f)$ for the set of simplicial distributions on $f$.
\end{defn}

%\coc{introduce $\on{sDist}$}

%
%

We will usually assume that $f$ is surjective to guarantee that a simplicial distribution exists on it. For simplicity of notation, we write $\on{sDist}(f)$ for $\on{sDist}_{\RR_{\geq 0}}(f)$. We refer to the elements of $\on{sDist}_{\RR_{\geq 0}}(f)$ as \emph{probabilistic simplicial distributions}, and to the elements of $\on{sDist}_{\BB}(f)$ as \emph{possibilistic simplicial distributions}.
A useful characterization of simplicial distributions is the following basic observation, see \cite[Proposition 4.9]{barbosa2023bundle}.

\begin{lem} 
\label{lem:charcofsimpdist}
A simplicial map $p \colon X \to D_R(E)$ is an $R$-simplicial distribution if and only if for every $x \in X_n$ we have:
$$
\set{e\in E_n \mid ~ p_x(e)\neq 0} \subset f_n^{-1}(x) .
$$
\end{lem}

\begin{cor}\label{cor:gentogen}
Let $p$ be a simplicial distribution on 
$f\colon E \to X$.  
If $\sigma \in X_n$ is a generator simplex  
of $X$,
then
$$
p_{\sigma}(d_i(e_1)) = 0
\quad \text{and} \quad
p_{\sigma}(s_j(e_2)) = 0
$$
for all $e_1 \in E_{n+1}$ and $e_2 \in E_{n-1}$.
\end{cor}
\begin{proof}
Suppose $p_{\sigma}(d_i(e_1)) \neq 0$ for some $e_1 \in E_{n+1}$.  
By Lemma~\ref{lem:charcofsimpdist}, this implies
$f_n(d_i(e_1)) = \sigma$, hence 
$
d_i(f_{n+1}(e_1)) = \sigma
$. This contradicts the assumption that $\sigma$ is a generator simplex. Therefore,
$p_{\sigma}(d_i(e_1)) = 0$. A similar argument applies to degeneracies.
\end{proof}

Given  
$f \colon E \to X$, let $\on{sSect}(f)$ denote the set of sections of $f$. To a section $s\colon X\to E$, we associate a simplicial distribution $\delta_s \colon X\to D(E)$ defined by $(\delta_s)_x=\delta_{s_x}$, the delta function concentrated at the simplex $s_x$. Simplicial distributions of the form $\delta_s$, for a section $s$, are called \emph{deterministic distributions}.
We have a commutative diagram
\begin{equation}\label{eq:Theta_f}
\begin{tikzcd}[column sep=huge,row sep=large]
D_R(\on{sSect}(f)) \arrow[r,"\Theta_f"] & \on{sDist}_R(f) \\
\on{sSect}(f) \arrow[u,hook,"\delta_{\on{sSect}(f)}"] \arrow[ru,hook,"s\mapsto \delta_s"']
\end{tikzcd}
\end{equation} 
In other words, $\Theta_f$ is the unique $R$-convex extension of the assignment of deterministic distributions to sections.
For details, see 
\cite{barbosa2023bundle}.

\begin{defn}\label{def:conSimp} 
An $R$-simplicial distribution $p$ on $f\colon E\to X$ is called \emph{$R$-contextual} if it does not lie in the image of $\Theta_f$.  
Otherwise, it is called \emph{$R$-noncontextual}. 
When $R = \RR_{\geq 0}$, we  
simply say \emph{(non-)contextual}. 
\end{defn}

\begin{ex}\label{ex:first examples}
The canonical examples considered in this paper are simplicial distributions on a projection map
\[
X\times Y \to X.
\]
Unraveling the definition shows that such simplicial distributions are equivalently described by simplicial maps of the form
\[
p\colon X\to D_R(Y).
\]
{We will denote the set of simplicial distributions in this case by $\sDist_R(X,Y)$.}

Many important examples arise when $Y$ is a specific simplicial set. Let $\Delta_{\ZZ_m}$ denote the simplicial set whose $n$-simplices are given by
\[
(\Delta_{\ZZ_m})_n = \{ (a_0,a_1,\dots,a_n) \in \ZZ_m^{n+1} \}.
\]
The face maps are defined by omitting a coordinate,
\[
d_i(a_0,a_1,\dots,a_n) = (a_0,a_1,\dots,\widehat{a_i}, \dots,a_n),
\]
and the degeneracy maps are defined by duplicating a coordinate,
\[
s_i(a_0,a_1,\dots,a_n) = (a_0,\dots,a_i,a_i,\dots, a_n).
\]
We denote the corresponding projection map by
\[
f_{X,m}\colon X\times \Delta_{\ZZ_m} \to X.
\]
In this setting, there are many interesting contextual simplicial distributions. One fundamental class of examples is the following:

Let ${C^{(n)}}$ denote the simplicial set with vertices $\{0,1,\dots,n-1\}$ and generator simplices given by the edges $i\to i+1$ for $i=0,\dots,n-2$ and $0\to n-1$. This is a one-dimensional simplicial set corresponding to a cycle graph with $n$ edges.  
All extremal contextual simplicial distributions
\[
p\colon {C^{(n)}}\to D(\Delta_{\ZZ_m})
\]
on this class of examples are classified in \cite{kharoof2024extremal}. {For details, see Example \ref{ex:circle}}. 
\end{ex}

\subsection{Distributions on presheaves}

In the quantum foundations literature, simplicial distributions often arise from presheaves of distributions. More specifically, one considers distributions indexed by the simplices of a simplicial complex. Such distributions can be realized as simplicial distributions.

Let $\catsComp$ denote the \emph{category of simplicial complexes}:
\begin{itemize}
    \item Objects are finite abstract simplicial complexes $(V,\Sigma)$, where $V$ is a finite vertex set and $\Sigma \subset \mathcal{P}(V)$ is a family of nonempty subsets closed under inclusion; that is, if 
    $\sigma \in \Sigma$ and $\tau \subset \sigma$, then $\tau \in \Sigma$. Usually, we write simply $\Sigma$ instead of $(V,\Sigma)$, and denote its vertex set by $V(\Sigma)$. 
    
    \item Morphisms are simplicial maps $f\colon \Sigma_1 \to \Sigma_2$, that is, functions $f\colon V(\Sigma_1) \to V(\Sigma_2)$ such that, for every simplex $\sigma \in \Sigma_1$, we have
    \[
        f(\sigma) = \{ f(v) \mid ~ v \in \sigma \} \in \Sigma_2.
    \]
\end{itemize}
We will write $\catC_{\Sigma}$ for the \emph{poset category} of the simplicial complex $\Sigma$, whose objects are the simplices of $\Sigma$ and whose morphisms are simplex inclusions. 

\begin{defn}\label{def:empirical model on F}
{Let $\Sigma$ be a simplicial complex.} We write $\on{PSh}(\Sigma)$ for the set of presheaves $F\colon \catC_\Sigma^\op \to \catSet$ on {the} poset category {$\catC_{\Sigma}$} {(or simply we say \emph{presheaves on $\Sigma$})}.
For a presheaf $F\in \on{PSh}(\Sigma)$, we write 
\[
{\on{Emp}_R(F)}= \lim\, {(}D_R\circ F{)}
\]
and refer to an element of this inverse limit as an \emph{{$R$-}empirical model} on $F$. 
\end{defn}

We say that \(F\) is \emph{non-trivial} if, for every \(\sigma\in \Sigma\), the set \(F(\sigma)\) is non-empty. We will usually assume that $F$ is non-trivial, so that $\on{Emp}_R(F)\neq \emptyset$. 
For simplicity of notation, we write $\on{Emp}(F)$ for $\on{Emp}_{\RR_{\geq 0}}(F)$. We call the elements of $\on{Emp}(F)$ \emph{probabilistic empirical models}, and the elements of $\on{Emp}_{\BB}(F)$ \emph{possibilistic empirical models}.
 
\begin{rem}\label{rem:restofdist}
For a 
%pre-event scenario 
presheaf
$F \colon \catC^\op_{\Sigma} \to \catSet$, an element of
%$p \in \Emp_{R}(F)$ 
$\on{Emp}_R(F)$
is a 
%collection
family
of 
distributions $\set{p_{\sigma}}_{\sigma \in \Sigma}$ such that 
\begin{itemize}
    \item For every $\sigma \in \Sigma$, we have $p_{\sigma}\in D_{R}(F(\sigma))$.
    \item For every 
    %morphism 
    inclusion $i \colon \sigma \to \tau$ of simplices, we have $D_{R}(F(i))(p_{\tau})=p_{\sigma}$.
\end{itemize}
We usually write $p_{\tau}|_{\sigma}=D_{R}(F(i))(p_{\tau})$.   
\end{rem}

\begin{defn}\label{def:conEmp}
For $F\in \on{PSh}(\Sigma)$, we define  
$$
\Theta_{F}\colon D_R(\lim F) \to \lim {(D_R \circ F)}=\on{Emp}_R(F) 
$$
{to be the canonical map induced by the universal property of the limit.}
An $R$-{empirical model}
%empirical {model} 
$p$ on $F$ is called \emph{$R$-contextual} if it does not lie in the image of $\Theta_F$. Otherwise, it is called \emph{$R$-noncontextual}. %
%In the case 
When $R=\mathbb{R}_{\ge 0}$, we 
%will 
simply say \emph{(non-)contextual}
% and \emph{noncontextual} 
instead of 
$\mathbb{R}_{\ge 0}$-(non-)contextual.
% and $\mathbb{R}_{\ge 0}$-noncontextual.
\end{defn}

\begin{ex}\label{ex:eventpreshef}
Our definition generalizes a well-known construction from quantum foundations.
Let $T=(\Sigma,O)$, where $\Sigma$ is a simplicial complex and 
$O=\{O_x\}_{x\in V(\Sigma)}$ is a family of sets. {We write $T=(\Sigma,\ZZ_m)$ when $O_x=\ZZ_m$ for every $x\in V(\Sigma)$.}
The \emph{presheaf of events}
\[
\eE_T \colon \catC_\Sigma^{\op} \to \catSet
\]
is defined by sending each simplex $\sigma$ to
$\prod_{x\in \sigma} O_x$.
The notion of contextuality defined above generalizes the one introduced in \cite{abramsky2011sheaf} to study quantum contextuality.
\end{ex}

\begin{ex}
An interesting class of examples of contextual {empirical models} is obtained from the boundary $\partial \Delta^n$ of the standard $n$-simplex 
{(Definition \ref{def:standardsimpl})}. Let 
${\eE}_T\in \on{PSh}(\partial \Delta^n)$, where {$T=(\partial \Delta^n, \ZZ_m)$.}
The existence of contextual empirical models {in $\on{Emp}(\eE_T)$} plays a key role in Vorob'ev's theorem \cite{vorob1962consistent}, which characterizes acyclic simplicial complexes as precisely those simplicial complexes that do not admit contextual empirical models 
{(see Proposition \ref{pro:vertboundrysimpled})}. {We will return to this example in Section~\ref{sec:Combinatorial sphere}.}
\end{ex}

\subsubsection{Singular realization}

{In this section, we explain how empirical models can be realized as simplicial distributions.}

\begin{defn}\label{def:singular realization}
Let $\catComp_{\geq}$ denote the category of finite simplicial complexes equipped with a total order on their vertex sets, with order-preserving simplicial maps as morphisms. We define the \emph{singular realization functor}
\[
S\colon \catComp_{\geq}\to \catsSet
\]
as follows.
For an ordered simplicial complex $\Sigma$, let $S(\Sigma)$ be the simplicial set whose set of $n$-simplices is
\[
S(\Sigma)_n
:=
\left\{
(v_0,\dots,v_n)\in V(\Sigma)^{n+1}
\mid ~
v_0\leq \cdots \leq v_n
\text{ and }
\{v_0,\dots,v_n\}\in \Sigma
\right\}.
\]
Thus, an $n$-simplex of $S(\Sigma)$ is an ordered simplex of $\Sigma$, with repetitions of vertices allowed.
The face maps are defined by deleting vertices:
\[
d_i(v_0,\dots,v_n)
=
(v_0,\dots,\widehat{v_i},\dots,v_n),
\qquad 0\leq i\leq n,
\]
%where $\widehat{v_i}$ indicates that $v_i$ is omitted.
and the degeneracy maps are defined by repeating vertices:
\[
s_i(v_0,\dots,v_n)
=
(v_0,\dots,v_i,v_i,\dots,v_n),
\qquad 0\leq i\leq n.
\]
These maps satisfy the simplicial identities, so $S(\Sigma)$ is a simplicial set.
For an order-preserving simplicial map $\varphi\colon \Sigma\to\Gamma$, the simplicial map
\[
S(\varphi)\colon S(\Sigma)\to S(\Gamma)
\]
is defined degreewise by
\[
S(\varphi)_n(v_0,\dots,v_n)
=
(\varphi(v_0),\dots,\varphi(v_n)).
\] 
\end{defn}

\begin{defn}\label{def:realization of F}
Let $\Sigma$ be an ordered simplicial complex and let 
$F\in \on{PSh}(\Sigma)$. We define a simplicial set $E(F)$ over $S(\Sigma)$ as follows.
For each $n\geq 0$, set
\[
E(F)_n
:=
\bigsqcup_{(v_0,\dots,v_n)\in S(\Sigma)_n}
F(\{v_0,\dots,v_n\}).
\]
Thus, an $n$-simplex of $E(F)$ is a pair
\[
((v_0,\dots,v_n),a),
\qquad
a\in F(\{v_0,\dots,v_n\}).
\]
The face maps are defined as follows. For $(v_0,\dots,v_n)\in S(\Sigma)_n$ and $0\leq i\leq n$, the inclusion
\begin{equation}\label{eq:iotainclusion}
\iota_i\colon \{v_0,\dots,\widehat{v_i},\dots,v_n\}
\hookrightarrow
\{v_0,\dots,v_n\}
\end{equation}
induces, by contravariance of $F$, a map
\[
F(\{v_0,\dots,v_n\})
\to
F(\{v_0,\dots,\widehat{v_i},\dots,v_n\}).
\]
We define
\[
d_i^{E(F)}((v_0,\dots,v_n),a)
=
\bigl((v_0,\dots,\widehat{v_i},\dots,v_n),
F(\iota_i)(a)\bigr).
\]
The degeneracy maps are defined by
\[
s_i^{E(F)}((v_0,\dots,v_n),a)
=
((v_0,\dots,v_i,v_i,\dots,v_n),a).
\] 
Finally, {the} projection onto the indexing simplex defines a simplicial map
\[
f_F\colon E(F)\to S(\Sigma),
\qquad
(f_F)_n((v_0,\dots,v_n),a)=(v_0,\dots,v_n).
\]
\end{defn}

%Since $F$ is nontrivial, we get that $f_F$ is surjective.  

\begin{ex}
Let $F=\eE_T$ denote the presheaf of events for $T=(\Sigma,\ZZ_m)$ (see Example~\ref{ex:eventpreshef}). 
Note that
$$
\begin{aligned}
E(F)_n 
&= \bigsqcup_{(v_0,\dots,v_n)\in S(\Sigma)_n} F(\{v_0,\dots,v_n\}) \\
&= \bigsqcup_{(v_0,\dots,v_n)\in S(\Sigma)_n} \;\prod_{x\in \{v_0,\dots,v_n\}} \ZZ_m \\
&\cong \bigsqcup_{(v_0,\dots,v_n)\in S(\Sigma)_n} \;\prod_{x\in \{v_0,\dots,v_n\}} \bigl(\{x\} \times \ZZ_m\bigr).
\end{aligned}
$$
On the other hand, we have
$$
\begin{aligned}
(S(\Sigma) \times \Delta_{\ZZ_m})_n 
= S(\Sigma)_n \times \prod_{i=0}^n \ZZ_m 
\cong \bigsqcup_{(v_0,\dots,v_n)\in S(\Sigma)_n} \;\prod_{i=0}^n \bigl(\{v_i\} \times \ZZ_m\bigr).
\end{aligned}
$$
Therefore, there is an inclusion of simplicial sets over $S(\Sigma)$:
$$
\begin{tikzcd}[column sep=huge,row sep=large]
E(F) \arrow[dr,"f_F"'] \arrow[rr,hook,"i"] && S(\Sigma) \times \Delta_{\ZZ_m} \arrow[dl,"f_{S(\Sigma),m}"] \\
& S(\Sigma) &
\end{tikzcd}
$$
(see Example~\ref{ex:first examples}).
Note that $i_n$ behaves like the identity on the fibers over non-degenerate simplices of $S(\Sigma)$. Consequently, it induces an affine isomorphism
$$
i_{\ast} \colon \sDist(f_F) \to \sDist(f_{S(\Sigma),m}).
$$
\end{ex}
{The construction in Definition \ref{def:realization of F} provides a functor
\begin{equation}\label{eq:functorXi}
\Xi_{\Sigma} \colon  \catPSh_{/\Sigma} \to \catsSet_{/S(\Sigma)},
\end{equation}
where for a presheaf {$F$}, we set $\Xi_{\Sigma} (F)=f_F$ and for a natural transformation $\alpha \colon F \to G$, 
the simplicial map $\Xi_{\Sigma} (\alpha)\colon E(F) \to E(G)$ {is defined} as follows: 
$$
\Xi_{\Sigma} (\alpha)_n\left((v_0,\dots,v_n),a) \right)=\left((v_0,\dots,v_n),\alpha_{\set{v_0,\dots,v_n}}(a)\right).
$$
}

\begin{pro}\label{pro:isomorEventBundle}
Let $F\in \on{PSh}(\Sigma)$, and let 
$f_F \colon E(F) \to S(\Sigma)$ be the associated simplicial map. 
Define a map
$$
\Phi_F {:=\Phi_{R,F}}\colon \on{Emp}_R(F) \to \on{sDist}_R(f_F)
$$
by sending an empirical model $p = \{p_\sigma\}_{\sigma \in \Sigma}$ on $F$ to the simplicial distribution $\Phi_F(p)$ 
given as follows: for every $x=(v_0,\dots,v_n) \in S(\Sigma)_n$ and every 
$e \in E(F)_n$, set
\begin{equation}\label{eq:PhiFFF}
\Phi_F(p)_{x}(e) = 
\begin{cases}
p_{\{v_0,\dots,v_n\}}(a)  & \text{if } e=(x,a) \text{ with } a \in F(\{v_0,\dots,v_n\}), \\
0  & \text{otherwise.}
\end{cases}
\end{equation}
Then $\Phi_F$ is an affine isomorphism.
\end{pro}
\begin{proof}
{First, we prove that $\Phi_F(p) \colon S(\Sigma) \to D_R(E(F))$ is a simplicial map. Let $x=(v_0,\dots,v_n) \in S(\Sigma)_n$ and 
let $e=(y,a) \in E(F)_{n-1}$. Then
$$
\phi_F(p)_{d_i(x)}(e)= \begin{cases}
p_{\set{v_0, \dots, \widehat{v_i}, \dots, v_n}}(a)  & \text{if } y=d_i(x) \text{ and } a \in F(\set{v_0, \dots, \widehat{v_i}, \dots, v_n})  \\
0  & \text{otherwise.}
\end{cases}
$$
On the other hand, consider the inclusion $\iota_i$ in (\ref{eq:iotainclusion}). 
If $d_i(x)=y$, 
then  
$$
\begin{aligned}
D_R(d_i)\bigl(\Phi_F(p)_{x}\bigr)(e)
&=
\sum_{e' :\, d_i(e')=e}
\Phi_F(p)_{(v_0,\dots,v_n)}(e') \\
&=
\sum_{\substack{
e'=(x,a') \\
a' \in F(\{v_0,\dots,v_n\}) \\
d_i(e')=e
%F(\iota_i)(e') = e
}}
p_{\{v_0,\dots,v_n\}}(a') \\
&=
\sum_{\substack{
a' \in F(\{v_0,\dots,v_n\}) \\
F(\iota_i)(a')=a
}}
p_{\{v_0,\dots,v_n\}}(a')  \\
&=
p_{\{v_0,\dots,v_n\}}
|_{\{v_0,\dots,\widehat{v_i},\dots,v_n\}}(a) \\
&=\begin{cases} 
p_{\{v_0,\dots,\widehat{v_i},\dots,v_n\}}(a) & \text{if }  a \in F(\set{v_0, \dots, \widehat{v_i}, \dots, v_n})  \\
0  & \text{otherwise.}
\end{cases}
\end{aligned}
$$
If 
$y \neq d_i(x)$,
then there is no 
$e' \in E(F)_n$
such that
$d_i(e')=e$,
and hence the above sum is zero. Therefore,
$$
D_R(d_i)\bigl(\Phi_F(p)_{x}\bigr)
=
\Phi_F(p)_{d_i(x)}.
$$

Now, let $x=(v_0,\dots,v_n) \in S(\Sigma)_n$ and 
let 
$e=(y,a) \in E(F)_{n+1}$. Then
$$
\Phi_F(p)_{s_i(x)}(e)
=
\Phi_F(p)_{(v_0,\dots,v_i,v_i,\dots,v_n)}(e)
=
\begin{cases}
p_{\{v_0,\dots,v_n\}}(a)
& \text{if } y=s_i(x) \text{ and } a \in F(\{v_0,\dots,v_n\}), \\
0 & \text{otherwise.}
\end{cases}
$$
On the other hand,
$$
\begin{aligned}
D_R(s_i)\bigl(\Phi_F(p)_{x}\bigr)(e)
&=
\sum_{e' :\, s_i(e')=e}
\Phi_F(p)_{(v_0,\dots,v_n)}(e') \\
&=
\begin{cases}
p_{\{v_0,\dots,v_n\}}(a)
& \text{if } y=s_i(x) \text{ and } a \in F(\{v_0,\dots,v_n\}), \\
0 & \text{otherwise.}
\end{cases}
\end{aligned}
$$
Hence,
$$
D_R(s_i)\bigl(\Phi_F(p)_{x}\bigr)
=
\Phi_F(p)_{s_i(x)}.
$$
Therefore, $\Phi_F(p)$ is a simplicial map. 

Moreover, by definition,
$$
\{e \in E(F)_n \mid~
\Phi_F(p)_{(v_0,\dots,v_n)}(e)\neq 0\}
\subseteq
\set{(v_0,\dots,v_n)}\times F(\{v_0,\dots,v_n\})
=
(f_{F})_n^{-1}(v_0,\dots,v_n).
$$
Hence, by Lemma~\ref{lem:charcofsimpdist},
$\Phi_F(p) \in \sDist_R(f_F)$.

The map $\Phi_F$ is clearly affine.
Now let $q \in \sDist_R(f_F)$.
For $x \in S(\Sigma)_n$ and $e \in E(F)_n$, we have
$$
q_{s_i(x)}(s_i(e))
=
D_R(s_i)(q_x)(s_i(e))
=
\sum_{e':\; s_i(e')=s_i(e)} q_x(e')
=
q_x(e),
$$
since $s_i$ is injective. 
Therefore, we can define a map
$$
\Psi_F \colon \sDist_R(f_F) \longrightarrow \Emp_R(F)
$$
by
$$
\Psi_F(q)_{\{v_0,\dots,v_n\}}
:=
q_{(v_0,\dots,v_n)},
$$
where $v_0 \leq \dots \leq v_n$.

One verifies directly that
$$
\Psi_F\circ \Phi_F=\Id_{\Emp_R(F)}
\qquad\text{and}\qquad
\Phi_F\circ \Psi_F=\Id_{\sDist_R(f_F)}.
$$

Therefore, $\Phi_F$ is an affine isomorphism.
}
\end{proof}

\subsection{Naturality}\label{subsec:nattt}
{In this section, we show that both simplicial distributions and empirical models assemble into functors on suitable slice categories, and that the embedding of the theory of empirical models into the theory of simplicial distributions is compatible with these functorial structures.}
\begin{defn}\label{def:catpsipBund}
For a simplicial set $X$, 
%we define the category $\sBund(X)$ to be:
consider the over category $\catsSet_{/X}$:
\begin{itemize}  
    \item objects are simplicial 
    %pre-bundle scenarios (see Definition \ref{def:simplicial-DV-LS-BS}). 
    maps $f\colon E\to X$,
    \item a morphism from $f\colon E \to X$ to $g\colon F \to X$ is a simplicial map $\alpha\colon E \to F$ 
    that makes the following diagram commute:
\begin{equation}\label{eq:alphaftog}
\begin{tikzcd}[column sep=huge,row sep=large]
E \arrow[dr,"f"'] \arrow[rr,"\alpha"] && F\arrow[dl,"g"] \\
& X &
\end{tikzcd}
\end{equation}
\end{itemize}
%
%In other words, $\sBund(X)$ is the subcategory of the slice category $\catsSet/X$, where the objects are the 
%simplicial pre-bundle scenarios.
\end{defn}

\begin{defn}\label{def:piast}
For a simplicial map $\pi\colon Y\to X$, we define the functor 
\[
\pi^\ast \colon \catsSet_{/X}  \to \catsSet_{/Y}
\] 
by 
\begin{itemize}
\item sending an object $f\colon E \to X$ 
%in $\sBund(X)$ 
to the object $\pi^\ast(f)$ 
%that 
defined by the pullback square
$$
      \begin{tikzcd}[column sep=huge,row sep=large]
E \arrow[d,"f"] &\arrow[l,""]  E \times_{X} Y 
\arrow[d,"\pi^{\ast}(f)"]
 \\
X  \arrow[ru, phantom, "\llcorner", very near end]  &\arrow[l,"\pi"]  Y 
\end{tikzcd}
$$
Here, $\pi^{\ast}(f)$ 
is the canonical projection from the pullback.
% to $Y$.
%

\item sending a morphism $\alpha$ as in Diagram~(\ref{eq:alphaftog}) to the natural morphism
\[
\pi^{\ast}(\alpha)=\alpha \times \Id_Y\colon E \times_X Y \to F\times_X Y.
\] 
\end{itemize}
This construction gives a functor
\[
\catsSet_{/-}\colon \catsSet^\op \to \catCat,
\]
{where $\catCat$ is the category of small categories.}
\end{defn}

For a simplicial set $X$, we define a functor 
\[
\sDist_{R,X}\colon \catsSet_{/X} \to \catSet
\]
that sends a {map}
%pre-bundle scenario 
$f\colon  E\to X$ to the set of simplicial distributions on $f$.
%, which we denote by 
%$\sDist_{R}(X)(f)$. 
Note that we carry the base simplicial set in the notation and write $\on{sDist}_{R,X}(f)$ for this set.
Given a morphism $\alpha \colon f \to g$ in $\catsSet_{/X}$, the functor $\sDist_{R,X}$ acts by
\[
\alpha_\ast= \sDist_{R,X}(\alpha)\colon \sDist_{R,X}(f) \to \sDist_{R,X}(g), \qquad
p \mapsto D_{R}(\alpha) \circ p.
\]
Furthermore, for a simplicial map $\pi\colon Y \to X$, we obtain a natural transformation {such that the following diagram is a morphism} in the 
{thick} slice category 
$\catCat \slice \catSet$ {(see Definition \ref{def:Sli2cat})}:
\[
\begin{tikzcd}[column sep=huge,row sep=large]
\catsSet_{/X} 
\arrow[rr,"\pi^{\ast}"]
\arrow[ddr,"{\sDist_{R,X}}"',""{name=A,right}] && \catsSet_{/Y}
\arrow[ddl,"{\sDist_{R,Y}}",""{name=B,left}] \\
 &\arrow[Rightarrow, from=A, to=B, "{\pi^*_{-}}"]& \\
&  \catSet&  
\end{tikzcd}.
\]
{Here,} for 
%each 
$f \in \catsSet_{/X}$, the natural map 
$\pi^*_{f}\colon \sDist_{R,X}(f) \to \sDist_{R,Y}(\pi^\ast(f))$ 
sends $p$ to $\pi_{\ast}(p)$ defined by {the formula 
\begin{equation}\label{eq:piastp}
 \pi^{\ast}_f(p)(y)(e,y)=p_n(\pi_n(y))(e)   
\end{equation}
for every $(e,y)\in E_n \times_{X_n} Y_n$.} 
%\ak{I rewrote the explicit formula, since it is needed later. I also removed the formula involving
%\[
%m\colon D_R(E)\times D_R(Y)\to D_R(E\times Y),
%\]
%because the map $m$ was not explained and we don't use it in any other place.}

Thus, we have a functor
\[
\sDist_{R,-}\colon \catsSet^{\op} \to \catCat \slice \catSet.
\]

\begin{defn}\label{def:simplicial distribution functor}
The \emph{simplicial distribution functor} is defined as the relative Grothendieck construction  
\[
\on{sDist}_{R}=\int_{\catsSet^{\op}} \on{sDist}_{R,-} \colon  \int_{\catsSet^\op} \catsSet_{/-} \longrightarrow \catSet.
\]
{See Definition \ref{def:GenGroth}.}  
\end{defn}

We have analogous constructions for the presheaf point of view. For a simplicial complex $\Sigma$, the category $\catPSh_{/\Sigma}$ of presheaves on $\Sigma$ consists of objects given by functors $F\colon \catC^\op_\Sigma\to \catSet$ and morphisms by natural transformations. 
An ordered simplicial complex map $\pi\colon \Gamma\to \Sigma$ induces a functor
$$
\pi^*\colon \catPSh_{/\Sigma} \to \catPSh_{/\Gamma}
$$
by sending $F$ to the composite {$F \circ \bar \pi$}, where $\bar \pi\colon \catC_\Gamma \to \catC_\Sigma$ is the induced map between the poset categories. This gives us a functor
\[
\catPSh_{/-}\colon \catComp_{\geq}^\op \to \catCat.
\] 

{Moreover, we have} a natural transformation from $\catPSh_{/-}$ to $\catsSet_{/{S^{\op}(-)}}$.

The empirical model functor can be assembled into a functor of the form 
\[
\on{Emp}_{R,-}\colon \catComp_{\geq}^{\op} \to \catCat \slice \catSet.
\]
For {simplicial complex $\Sigma$,} the maps $\Phi_F$ of Proposition \ref{pro:isomorEventBundle} assemble into a natural isomorphism
{
\begin{equation}\label{eq:isoemptoxisdis}
\Phi_{\Sigma}\colon \on{Emp}_{R,\Sigma}  \to \on{sDist}_{R,S(\Sigma)}\circ \Xi_{\Sigma},
\end{equation}
see (\ref{eq:functorXi}). This induces the natural isomorphism}
\[
\on{Emp}_{R,-}  \to \on{sDist}_{R,-}\circ S^{\op}.
\]
Then {we obtain the following diagram between the relative Grothendieck functors:}  
$$
\begin{tikzcd}[column sep=huge,row sep=large]
\int_{ \catComp_{\geq}^{\op}}  \catPSh_{/-} 
\arrow[rr,hook,""]
\arrow[ddr,"{\int_{ \catComp_{\geq}^{\op}} \on{Emp}_{R,-} }"',""{name=A,right}] && \int_{ \catsSet^{\op}}  \catsSet_{/-} 
\arrow[ddl,"{\sDist_{R}}",""{name=B,left}] \\
 &\arrow[Rightarrow, from=A, to=B, "\cong"]& \\
&  \catSet&  
\end{tikzcd}.
$$
This way we see that the embedding of the theory of empirical models into the theory of simplicial distributions is natural. 
%\end{rem}

%\begin{rem}\label{rem:relation version}

There are variants of this construction in the simplicial complex case. We can work over the category $\catComp$ of unoriented simplicial complexes. The \emph{unoriented singular realization functor} 
\[
S\colon \catComp \to \catsSet
\] 
is defined analogously to Definition \ref{def:singular realization} with the same face and degeneracy maps except that the ordering assumption is dropped in the definition of the simplices:
\[
S(\Sigma)_n = \{ (v_0,\dots,v_n) \in V(\Sigma)^{n+1} \mid~ \{v_0,\dots,v_n\} \in \Sigma  \}.
\]
Unfortunately with this realization Proposition \ref{pro:isomorEventBundle}  fails. It is still injective but not a bijection. 

In applications, the unoriented category is extended to $\catRel$, the category of simplicial complexes and simplicial relations, see \cite{barbosa2023closing,barbosa2023bundle}. A categorical way to achieve this extension is to consider the \emph{nerve complex} monad
\[
\hat N\colon \catComp \to \catComp
\]
that sends a simplicial complex $\Sigma$ to the simplicial complex $\hat N \Sigma$ whose vertex set is given by $\Sigma${, and} a subset $\{\sigma_1,\sigma_2,\dots,\sigma_k\}$ is a simplex of this complex if $\cup_{i=1}^k\sigma_i\in \Sigma$. The category $\catRel$ is then defined as the Kleisli category $\catComp_{\hat N}$. In this case the embedding of the theory over simplicial simplicial relation to simplicial set setting can be done by considering the \emph{nerve functor}
\[
N\colon \catComp \to \catsSet
\]
sending $\Sigma$ to the simplicial set $N(\Sigma)$ defined by the unoriented singular realization $S(\hat N(\Sigma))$ of the nerve complex. 
The functor $\catPSh_{/-}$ extends to {$\catRel$} since a simplicial relation $\Gamma\to \Sigma$ represented as a simplicial complex map $\pi \colon \Gamma \to \hat N \Sigma$ induces a functor $\bar \pi \colon \catC_\Gamma \to \catC_\Sigma$, 
{which can be used to} pull back presheaves on $\Sigma$ to presheaves on $\Gamma$.
Details of these constructions are given in \cite{barbosa2023bundle,kharoof2025simplicial}.

\begin{defn}\label{def:empirical model functor}
The \emph{empirical model functor} is defined as the relative Grothendieck construction  
\[
\on{Emp}_{R}=\int_{\catRel^{\op}} \on{Emp}_{R,-} \colon  \int_{\catRel^{\op}} \catPSh_{/-} \longrightarrow \catSet.
\] 
\end{defn}

%\end{rem}

\section{Possibilistic simplicial distributions}\label{sec:Possibilistic} 
{In this section, we introduce the functor of sub-bundle scenarios and show that it is equivalent to the functor of possibilistic simplicial distributions. We then establish the analogous result on the sheaf-theoretic side by showing that the functor of possibilistic empirical models is equivalent to the functor of sub-event scenarios.}

\subsection{Bundle scenarios}

The support of a simplicial distribution on a simplicial map can be described by special types of simplicial maps first introduced in \cite[Definition~4.1]{barbosa2023bundle}.

\begin{defn}\label{def:bundle scenario}
Let $f\colon E\to X$ be a simplicial map. We say that $f$ is \emph{flasque} if, for every ordinal map $\theta\colon [n]\to [m]$, every commutative square
\[
\begin{tikzcd}[column sep=huge,row sep=large]
\Delta[n] \arrow[r] \arrow[d,"\theta"'] & E \arrow[d,"f"] \\
\Delta[m] \arrow[r] \arrow[ru,dashed] & X 
\end{tikzcd}
\]
admits a lift, i.e., a diagonal arrow making both triangles commute.
We call a surjective flasque simplicial map a \emph{bundle scenario}.
\end{defn}

\begin{rem}\label{rem:more on the lifting}
The terminology {above} reflects a sheaf-like property. Note that, by Yoneda's lemma, all simplicial set maps $\Delta[n]\to \Delta[m]$ are given precisely by ordinal maps $[n]\to [m]$. Given such a map, the existence of a lift for every square is equivalent to the surjectivity of the natural map
\begin{equation}\label{eq:Emtopullback}
E_m \to  X_m\times_{X_n} E_n.
\end{equation}
To our knowledge, such maps have not been studied in the literature before and first appeared in~\cite{barbosa2023bundle}. There, face and degeneracy lifting are named separately:
\begin{enumerate}
    \item  \(f\) is called \emph{locally surjective} if it has the right-lifting property with respect to
    all {co}face maps \({d^i}\colon \Delta[n-1] \to \Delta[n]\).
    \item \(f\) is called \emph{discrete over vertices} if it has the right-lifting property with respect to
    all {co}degeneracy maps \({s^i}\colon \Delta[n] \to \Delta[n-1]\).
\end{enumerate}
In words, this means that whenever the image under $f$ of a simplex of $E$ is a face or a degeneracy of a simplex in $X$, then the simplex itself is also a face or a degeneracy{, respectively, of a simplex in $E$ whose image under $f$ is the mentioned simplex in $X$.}

\end{rem}

{We say that a simplicial map $f\colon E \to X$ \emph{has finite fibers} if $|f_n^{-1}(x)|<\infty$ for every $x\in X_n$, $n\geq 0$.}
\begin{lem}\label{lem:alphasurj}
Given a morphism $\alpha\colon f\to g$ as in Diagram~\eqref{eq:alphaftog}, suppose that $\alpha$ is surjective and that $f$ is a bundle scenario {with finite fibers}. Then $g$ is also a bundle scenario {with finite fibers}.
\end{lem}
%bundle scenario.  
%
\begin{proof} 
Given a commutative diagram:
\begin{equation}\label{eq:liftg}
\begin{tikzcd}[column sep=huge,row sep=large]
\Delta[n] \arrow[d,"\theta"'] \arrow[r,"e"] & F\arrow[d,"g"] \\
\Delta[m]  \arrow[r,"x"] & X
\end{tikzcd}
\end{equation}
Since $\alpha_n$ is surjective there exist $e' \in E_n$ such that $\alpha_n(e')=e$. Then, we have 
$$
x \circ \theta=g_n(e)=g_n(\alpha_n(e'))=f_n\circ e'
$$
$f$ is a bundle scenario, so there exists $\tilde{e} \in E_m$ such that the following diagram commutes:
\begin{equation}\label{eq:liftf}
\begin{tikzcd}[column sep=huge,row sep=large]
\Delta[n] \arrow[d,"\theta"'] \arrow[r,"e'"] & E\arrow[d,"f"] \\
\Delta[m] \arrow[ru,"\tilde{e}"] \arrow[r,"x"] & X
\end{tikzcd}
\end{equation}
Composing Diagrams (\ref{eq:liftf}) and (\ref{eq:alphaftog}) gives us a lifting for Diagram (\ref{eq:liftg}). 
{Finally, for every $n\geq 0$ and $x \in X_n$, since $\alpha_n$ is surjective, we have $g_n^{-1}(x)=\alpha_n(f^{-1}_n(x))$. Therefore, $g$ also has finite fibers.}
\end{proof}
%
%

%\coc{An alternative notation for $\subb(f)$ is $\on{Supp}(f)$.}

\begin{defn}\label{def:subXXXX}
Given a simplicial set $X$, we define the \emph{functor of sub-bundle scenarios over $X$}
\[
\on{Sub}_X\colon \catsSet_{/X} \to \catSet
\] 
as follows:
\begin{itemize}
    \item It sends a simplicial map $f\colon E \to X$ to the set
\[
\on{Sub}_X(f)=\set{E' \in \catsSet\mid~ E'\subset E \; \text{and}\; f|_{E'} \; \text{is a bundle scenario 
{with finite fibers}}}.
\]
    \item It sends a morphism $\alpha\colon f \to g$ to the set map
\[
\on{Sub}_X(\alpha)\colon \on{Sub}_X(f)\to \on{Sub}_X(g)
\]
defined by sending $E' \in \on{Sub}_X(f)$ to $\alpha(E')$.
\end{itemize}
\end{defn}
Note that $\alpha(E') \in  {\on{Sub}_X}(g)$ by Lemma \ref{lem:alphasurj}. Now, for the next definition, recall Definition~\ref{def:piast}. 
Note that if $f$ is a bundle scenario then $\pi^\ast(f)$ is also a bundle scenario, since liftings are preserved under pullbacks and 
for every $n \geq 0$ and $y \in Y_n$, we have
$$
(\pi^\ast (f))_n^{-1}(y) = f_n^{-1}(\pi_n(y)) \times \{y\},
$$
so $\pi^\ast (f)$ has finite fibers.
 
\begin{defn} 
For a simplicial set map $\pi\colon Y\to X$, we define a morphism in $\catCat\slice \catSet$ from $\on{Sub}_X$ to $\on{Sub}_Y$. It consists of the functor $\pi^{\ast}\colon \catsSet_{/X} \to  \catsSet_{/Y}$ and the natural transformation determined by

\[
\on{Sub}(\pi)_f\colon \on{Sub}_X(f) \to \on{Sub}_Y(\pi^\ast(f)).
\] 
This map sends $E' \in \on{Sub}_X(f)$ to $E' \times_X Y \in \on{Sub}_Y(\pi^\ast(f))$. See Diagram~(\ref{eq:piastf}) and note that $\pi^\ast(f|_{E'})=\pi^{\ast}(f)|_{E'\times_{X} Y}$:
\begin{equation}\label{eq:piastf}
\begin{tikzcd}[column sep=huge,row sep=large]
E' \arrow[d,hook] & E'\times_{X} Y \arrow[dd,"\pi^\ast(f|_{E'})"]  \arrow[l]  \\
E \arrow[d,"f"] &    \\
X & \arrow[l,"\pi"]  Y 
\end{tikzcd}
\end{equation}
This defines the morphism $(\pi^\ast,\on{Sub}(\pi)_{-})$ as follows:
$$
\begin{tikzcd}[column sep=huge,row sep=large]
\catsSet_{/X}
\arrow[rr,"\pi^{\ast}"]
\arrow[ddr,"\on{Sub}_X"',""{name=A,right}] && \catsSet_{/Y}
\arrow[ddl,"\on{Sub}_Y",""{name=B,left}] \\
 &\arrow[Rightarrow, from=A, to=B, "{\on{Sub}(\pi)_{-}}"]& \\
&  \catSet&  
\end{tikzcd}
$$
Thus, we have a functor 
$$
\on{Sub}_{-}\colon \catsSet^{\op} \to \catCat \slice \catSet 
$$
and, by its relative Grothendieck construction, the \emph{functor of sub-bundle scenarios}
% of $\on{Sub}_{-}$:
\[
\on{Sub} =\int_{\catsSet^{\op}} \on{Sub}_{-}\colon  \int_{\catsSet^{\op}} \catsSet_{/-}\to \catSet. 
\]
\end{defn} 

\subsubsection{Possibilistic simplicial distributions}\label{subsec:Possibsimpdist}

In this section, we show that the functor of 
sub-bundle scenarios  
is equivalent to the functor of possibilistic simplicial distributions.
% defined in Section \ref{subsec:Simdistonbun}.
%
\begin{defn}\label{def:zetaXf}
For a 
%pre-bundle scenario 
simplicial map
$f \colon E \to X$, we define the map 
\begin{equation}\label{eq:zetaXfff}
\zeta_{X,f}\colon \on{sDist}_{\BB}(f) \to \on{Sub}_X(f)   
\end{equation}
by sending $p \in \on{sDist}_{\BB}(f)$ to the simplicial set $\zeta_{X,f}(p)$ 
given by
\begin{equation}\label{eq:zetaxf}
\zeta_{X,f}(p)_n=\set{e \in E_n\mid ~ p_n(f_n(e))(e)\neq 0}.
\end{equation}
\end{defn}

\begin{pro}\label{pro:etaXf}
The maps $\zeta_{X,f}$ from Definition~\ref{def:zetaXf} assemble into a natural isomorphism 
%$\zeta_X$ from $\sDist_{\BB}(X)$ to $\subb(X)$.    
\[
\zeta_X \colon \sDist_{\BB,X} \to \on{Sub}_X.
\]
\end{pro}
\begin{proof}
First, we prove that $\zeta_{X,f}(p)$ is a simplicial subset of $E$.
Let \(e\in \zeta_{X,f}(p)_n\) and \(0\le i\le n\). Then
$$
\begin{aligned}
p_{n-1}(f_{n-1}(d_i(e)))(d_i(e))
&=D_{\mathbb{B}}(d_i)\bigl(p_n(f_n(e))\bigr)(d_i(e)) \\
&=\sum_{e'\,:\, d_i(e')=d_i(e)} p_n(f_n(e))(e') \\
&=p_n(f_n(e))(e)+\dotsb \neq 0,
\end{aligned}
$$
so $d_i(e)\in \zeta_{X,f}(p)_{n-1}$. Similarly, we obtain that $s_i(e)\in \zeta_{X,f}(p)_{n+1}$. 
Now we prove that $\zeta_{X,f}(p)$ lies in $\subb(X)(f)$: 

Surjectivity: Let \(x\in X_n\). There exists $e \in E_n$ such that $p_n(x)(e)\neq 0$. 
By Lemma~\ref{lem:charcofsimpdist} $f_n(e)=x$, so $e \in \zeta_{X,f}(p)_n$. 

Local surjectivity: Let $e \in \zeta_{X,f}(p)_{n-1}$, $x\in X_n$, and $0 \leq i \leq n$ such that $f_n(e)=d_i(x)$. Then 
$$
\sum_{e'\,:\, d_i(e')=e} p_n(x)(e')
= D_{\mathbb{B}}(d_i)(p_n(x))(e)
= p_{n-1}(d_i(x))(e)
= p_n(f_n(e))(e)\neq 0.
$$
Thus there exists \(e'\in E_n\) with \(d_i(e')=e\) and \(p_n(x)(e')\neq 0\).  
By Lemma~\ref{lem:charcofsimpdist}, \(x=f_n(e')\), so \(e'\in \zeta_{X,f}(p)_n\).  
  
Discreteness over vertices follows by a similar argument. {In addition, for every $x \in X_n$ the fiber
$$
(f|_{\zeta_{X,f}(p)})_n^{-1}(x)=\set{e\in \zeta_{X,f}(p)_n\mid ~ f_n(e)=x }=\set{e\in E_n\mid ~ f_n(e)=x \text{ and } p_n(x)(e)\neq 0}
$$
is finite since $p_n(x)$ is a distribution.}

To prove the naturality of $\zeta_X$, given {a morphism} $\alpha\colon f \to g$ as in (\ref{eq:alphaftog}), we prove that the following diagram commutes:
$$
\begin{tikzcd}[column sep=huge,row sep=large]
{\sDist_{\BB,X}}(f) \arrow[rr,"\zeta_{X,f}"] \arrow[dd,"\sDist_{\BB,X}(\alpha)"'] && {\subb_X}(f)
\arrow[dd,"\subb_X(\alpha)"] \\
&& \\
\sDist_{\BB,X}(g) \arrow[rr,"\zeta_{X,g}"'] && \subb_X(g)
\end{tikzcd}
$$
Given $p \in \sDist_{\BB,X}(f)$, we have 
$$
\begin{aligned}
\zeta_{X,g}(D_{\BB}(\alpha)\circ p)_n &=\set{e \in F_n \mid~ (D_{\BB}(\alpha_n) \circ p_n)(g_n(e))(e)\neq 0}  \\
&=\set{e \in F_n \mid~ \sum_{\alpha_n(e')=e} p_n(g_n(e))(e')\neq 0}  \\
&=\set{e \in F_n \mid~ \exists e' \in E_n \; \text{s.t} \; \alpha_n(e')=e \; \text{and} \; p_n(g_n(e))(e')\neq 0} \\
&=\set{e \in F_n \mid~ \exists e' \in E_n \; \text{s.t} \; \alpha_n(e')=e \; \text{and} \; p_n(g_n(\alpha_n(e')))(e')\neq 0} \\
&=\set{e \in F_n \mid~ \exists e' \in E_n \; \text{s.t} \; \alpha_n(e')=e \; \text{and} \; p_n(f_n(e'))(e')\neq 0} \\
&=\alpha_n\left(\set{e' \in E_n \mid~ p_n(f_n(e'))(e')\neq 0}\right)  \\
&=\alpha_n(\zeta_{X,f}(p)_n). \\
\end{aligned}
$$
Finally, we prove that $\zeta_{X,f}$ is an isomorphism.  
Given \(E'\in\subb_X(f)\), define \({M_f}(E')\in \sDist_{\BB,X}(f)\) by
$$
M_f(E')_n(x)(e)=
\begin{cases}
1 &\text{if } e\in E'_n \text{ and } f_n(e)=x,\\
0 &\text{otherwise}.
\end{cases}
$$
We prove that \(M_f(E') \in \sDist_{\BB,X}(f)\). Given \(x\in X_n\). Since 
\(f_n|_{E'_n}\colon E'_n\to X_n\) is surjective, there is \(e\in E'_n\) with \(f_n(e)=x\), hence \(M_f(E')_n(x)(e)=1\).  
Recall also that, by Definition \ref{def:subXXXX}, that {$|f_n^{-1}(x)|<\infty$ for every $x \in X$ }. Thus \(M_f(E')_n(x)\in D_{\mathbb{B}}(E_n)\).
Next, we prove the simpliciality of $M_f(E')$. For $x \in X_n$, $e \in E_{n-1}$, by definition we have $M_f(E')_{n-1}(d^X_i(x))(e)=1$ if and only if $e \in E'_{n-1}$ and $f_{n-1}(e)=d_i^X(x)$. This is equivalent to commutativity of
$$
\begin{tikzcd}[column sep=huge,row sep=large]
\Delta[n-1] \arrow[d,"d^i"'] \arrow[r,"e"] & E'
\arrow[d,"f|_{E'}"] \\
\Delta[n] \arrow[r,"x"] & X
\end{tikzcd}
$$
On the other hand,
$
D_\BB(d_i^E)(M_f(E')_n(x))(e)=\sum_{e':\, d_i^E(e')=e}M_f(E')_n(x)(e')
$
is equal to $1$ if and only if there exists $e' \in E_n$ that makes the following two triangles commute
$$
\begin{tikzcd}[column sep=huge,row sep=large]
\Delta[n-1] \arrow[d,"d^i"'] \arrow[r,"e"] & E'
\arrow[d,"f|_{E'}"] \\
\Delta[n] \arrow[ru,"e'"] \arrow[r,"x"] & X
\end{tikzcd}
$$
Since $f|_{E'}$ is locally surjective, we obtain 
$$
M_f(E')_{n-1}(d^X_i(x))=
D_\BB(d_i^E)\bigl(M_f(E')_n(x)\bigr).
$$
Similarly, using the fact that $f|_{E'}$ is discrete over vertices, we get 
$$
M_f(E')_{n+1}(s^X_i(x))=D_\BB(s_i^E)(M_f(E')_n(x)).
$$
Finally, we prove that $D_{\BB}(f) \circ M_f(E')=\delta_X$. Given $x,x' \in X_n$, then 
$$
D_{\BB}(f_n)(M_f(E')_n(x))(x')=\sum_{e:\, f_n(e)=x'}M_f(E')_n(x)(e)=\begin{cases} 
1 & \text{if} \,\, x=x'\\
0 & \text{otherwise} 
\end{cases}
$$
Now let $p \in \sDist_{\BB,X}(f)$, $x \in X_n$, and $e\in f_n^{-1}(x)$. Then 
$
M_f(\zeta_{X,f}(p))_n(x)(e)=1 
$
if and only if $e \in \zeta_{X,f}(p)_n$, which holds if and only if $p_n(x)(e)=1$. Hence, 
$$
M_f(\zeta_{X,f}(p))=p
$$ 
Conversely, if $E' \in \subb(X)(f)$, then
$e \in \zeta_{X,f}(M_f(E'))_n$ if and only if $M_f(E')(f_n(e))(e)=1$, which holds if and only if $e \in E_n'$. 
We also proved that 
$$
\zeta_{X,f}(M_f(E'))=E'
$$
We therefore conclude that $\zeta_{X,f}$ is an isomorphism.
\end{proof}
\begin{pro}\label{pro:sdis=sub}
{There is a natural isomorphism} 
\[
\zeta\colon \sDist_{\BB,-} \to \on{Sub}_{-},
\]
defined as $\zeta_X$ at a simplicial set $X$.
%$\zeta$ is a natural isomorphism from $\sDist_{\BB,-}$ to $\subb_{-}$.
\end{pro}
\begin{proof}
By Proposition \ref{pro:etaXf}, for every object $X \in \catsSet$, we have the isomorphism $(\Id_{{\catsSet_{/X}}},\zeta_X)$ in $\catCat \slice \catsSet$:
$$
\begin{tikzcd}[column sep=huge,row sep=large]
{\catsSet_{/X}} 
\arrow[rr,equal]
\arrow[ddr,"\sDist_{\BB,X}"',""{name=A,right}] && {\catsSet_{/X}}
\arrow[ddl,"\subb_X",""{name=B,left}] \\
 &\arrow[Rightarrow, from=A, to=B, "{\zeta_X}"]& \\
&  \catSet&  
\end{tikzcd}
$$
We should prove that the following diagram commutes in the category $\catCat \slice \catsSet$: 
$$
\begin{tikzcd}[column sep=huge,row sep=large]
\sDist_{\BB,X} \arrow[rr,"{(\pi^\ast{,}\pi^\ast_{-})}"] \arrow[dd,"(\Id_{\catsSet_{/X}}{,}\zeta_X)"'] &&
\sDist_{\BB,Y} 
\arrow[dd,"(\Id_{\catsSet_{/Y}}{,}\zeta_Y)"] \\
&& \\
\subb_X \arrow[rr,"(\pi^\ast{,} \subb(\pi)_{-})"'] && \subb_Y
\end{tikzcd}
$$
That means, to prove that for every {simplicial map} $f\colon E \to X$ the following diagram commutes:
$$
\begin{tikzcd}[column sep=huge,row sep=large]
\sDist_{\BB,X}(f) \arrow[rr,"\pi^{\ast}_f"] \arrow[dd,"\zeta_{X,f}"'] && \sDist_{\BB,Y}(\pi^\ast(f)) 
\arrow[dd,"\zeta_{Y{,}\pi^{\ast}(f)}"] \\
&& \\
\subb_X(f) \arrow[rr,"\subb(\pi)_{f}"'] && \subb_Y(\pi^{\ast}(f))
\end{tikzcd}
$$
Given a simplicial distribution $p \in \sDist_{\BB,X}(f)$, we have
$$
\begin{aligned}
\zeta_{Y{,}\pi^{\ast}(f)}\left(\pi^{\ast}_f(p)\right)_n
&=\set{(e,y)\in E_n \times_{X_n} Y_n\mid ~  \pi^{\ast}_f(p)(y)(e,y)\neq 0} \\
&=\set{(e,y)\in E_n \times_{X_n} Y_n\mid ~  p_n(\pi_n(y))(e)\neq 0}  \\
&=\set{(e,y) \in E_n \times Y_n \mid~ f_n(e)=\pi_n(y) \;\text{and}\;   p_n(f_n(e))(e) \neq 0} \\
&=\set{(e,y)  \in E_n \times Y_n \mid~ f_n(e)=\pi_n(y) \;\text{and}\;  e\in \zeta_{X,f}(p)_n} \\
&=\zeta_{X,f}(p)_n \times_{X_n} Y_n \\
& =\subb(\pi)_f\left(\zeta_{X,f}(p)\right)_n
\end{aligned}
$$
see Equation (\ref{eq:piastp}). 
\end{proof}

\begin{thm}\label{thm:possibilistic sdist as sub}
The functors $\on{sDist}_{\BB}\colon \int \catsSet_{/-} \to \catSet$ and $\on{Sub}\colon\int \catsSet_{/-}\to \catSet$ are naturally isomorphic.
\end{thm}
\begin{proof}
Follows directly 
%by 
from
Propositions \ref{pro:sdis=sub} and \ref{pro:NatNat}.
\end{proof}

\subsection{Event scenarios}

The supports of empirical models can be studied using special types of presheaves introduced in \cite{kharoof2025simplicial}, {whose fundamental properties were first investigated in \cite{abramsky2015contextuality}.}

\begin{defn}\label{def:nontrivlocalsurj}
Let \(F \colon \catC^{\op} \to \catSet\) be a functor.
We say that \(F\) is \emph{flasque} if, for every morphism \(a \to b\) in \(\catC\), the induced map
$
F(b) \to F(a)
$
is surjective. A presheaf on a simplicial complex is called flasque if the corresponding functor is flasque. A presheaf 
{$F\colon \catC^{\op}_{\Sigma} \to \catSet$} that is both non-trivial, i.e., $F(\sigma)\neq \emptyset$ for every {$\sigma \in \Sigma$}, and flasque, is called an \emph{event scenario}. {The event scenario is called \emph{finite} if $|F(\sigma)|<\infty$ for every $\sigma \in \Sigma$.}
\end{defn}
\begin{pro}\label{pro:realization of event is bundle}
A presheaf $F$ on $\Sigma$ is an event scenario if and only if its realization $f_F$, as defined in Definition~\ref{def:realization of F}, is a bundle scenario over $S(\Sigma)$.
\end{pro}
\begin{proof}
{First, we show that $f_F$ is always discrete over vertices (see Remark~\ref{rem:more on the lifting}).
Given a codegeneracy map $[n]\to[n-1]$ in the simplex category $\Delta$, we prove that the induced map
\[
E(F)_{n-1}
\longrightarrow
S(\Sigma)_{n-1}\times_{S(\Sigma)_n} E(F)_n
\]
is surjective (cf.~(\ref{eq:Emtopullback})).
An element in
$S(\Sigma)_{n-1}\times_{S(\Sigma)_n}E(F)_n$ is in the form 
$$
\left((v_0,\dots,v_i,v_{i+2},\dots,v_n),(v_0,\dots,v_i,v_i,v_{i+2},\dots,v_n),a\right).
$$ 
Such an element is clearly the image of
$$
\bigl((v_0,\dots,v_i,v_{i+2},\dots,v_n),a\bigr)\in E(F)_{n-1}.
$$
Hence the map is surjective, and therefore $f_F$ is discrete over vertices.

Now consider a coface map $[n-1]\to[n]$ in $\Delta$. One can see that the induced map
\[
E(F)_{n}
\longrightarrow
S(\Sigma)_{n}\times_{S(\Sigma)_{n-1}} E(F)_{n-1}
\]
is surjective if and only if, for every inclusion $\iota_i$ as in (\ref{eq:iotainclusion}), the map 
$F(\iota_i)$ is surjective. Thus, $f_F$ is flasque if and only if $F$ is flasque.
Finally, it is obvious that $f_F$ is surjective if and only if $F$ is non-trivial.}
%\coc{Aziz, could you include the proof? This will be a good justification for flasque simplicial map terminology.}
\end{proof}

%In this section, we define the functor of sub-event scenarios using the relative Grothendieck construction. 
We begin by defining a partial order on 
%$\Func(\catC^{\op},\catSet)$
the functor category $[\catC^\op,\catSet]$.
% for a given category $\catC$.
%

\begin{defn}\label{def:Funcrelations}
Given functors $F',F\colon \catC^{\op}\to \catSet$, we write $F'\leq F$ if $F'$ is a subfunctor of $F$; that is, if $F'(a)\subset F(a)$ for every object $a$ of $\catC$, and for every morphism $s\colon a\to b$ in $\catC$ we have
\[
{F(s)|_{F'(b))} =F'(s).}
\]
Equivalently, the inclusions $F'(a)\to F(a)$ assemble into a natural transformation $F'\to F$. For presheaves on simplicial complexes, we use the same notion to define an order among them.
\end{defn}

%\begin{defn}\label{def:Funcrelations}
%Given functors $F', F\colon \catC^{\op} \to \catSet$, we write $F' \leq F$ if:
%\begin{itemize}
%    \item {$F'$ is a subfunctor of $F$, i.e.,} $F'(a) \subset F(a)$ for every object $a$ of $\catC$;  
%    \item the inclusions $\{F'(a) \to F(a)\}_{a \in \catC}$ 
%form a natural transformation from $F'$ to $F$.     
%\end{itemize}
%\end{defn}
%
%The second condition in Definition \ref{def:Funcrelations} is equivalent to requiring that for every morphism $s\colon a \to b$ in $\catC$, we have $F'(s)(x) = F(s)(x)$ for every $x \in F'(b)$.

%
%
%

\begin{defn}\label{def:alphaastF}
Given functors $F,G\colon \catC^{\op} \to \catSet$ and a natural transformation $\alpha \colon F \to G$, we define the functor 
$\alpha_\ast(F)\colon \catC^{\op} \to \catSet$ as follows:
\begin{itemize}
    \item For an object $a$, set $\alpha_\ast(F)(a) = \alpha_{a}(F(a))$.
    \item For a morphism $s\colon a \to b$, define $\alpha_\ast(F)(s)$ to be the restriction
    \begin{equation}\label{eq:Gsalpha}
    G(s)|_{\alpha_{b}(F(b))} \colon \alpha_{b}(F(b)) \to \alpha_{a}(F(a)).
    \end{equation}
\end{itemize}      
\end{defn}

The map in Eq.~\eqref{eq:Gsalpha} is well-defined by the commutative diagram
\begin{equation}\label{eq:GFalpha}
\begin{tikzcd}[column sep=huge,row sep=large]
F(b) \arrow[r,"\alpha_{b}"] \arrow[d,"F(s)"'] & G(b) 
\arrow[d,"G(s)"] \\
F(a) \arrow[r,"\alpha_{a}"'] & G(a)
\end{tikzcd}
\end{equation}
Note that $\alpha_{\ast}(F)\leq G$.

%\coc{finiteness omitted}
 
\begin{lem}\label{lem:alphaastF}
Let $F$ and $G$ be presheaves on $\Sigma$, and let $\alpha \colon F \to G$ be a natural transformation. If $F$ is {a finite} event scenario, then $\alpha_\ast(F)$, constructed in Definition~\ref{def:alphaastF}, is also {a finite} event scenario.
\end{lem}
\begin{proof}
{Non-triviality:} For every $\sigma \in \Sigma$, we have  
$
\alpha_\ast(F)(\sigma) = \alpha_{\sigma}(F(\sigma)) \neq \emptyset
$,
since $F(\sigma) \neq \emptyset$.

{Local surjectivity:} For $s\colon \sigma \hookrightarrow \tau$, the map
\[
G(s)|_{\alpha_{\tau}(F(\tau))} \colon \alpha_{\tau}(F(\tau)) \to \alpha_{\sigma}(F(\sigma))
\]
is surjective. This follows from the commutativity of diagram (\ref{eq:GFalpha}) and the fact that $F(s)$ is surjective. 
{Finally, for every $\sigma \in \Sigma$, since $F(\sigma)$ is finite, we have $|\alpha_\ast(F)(\sigma)|=|\alpha_{\sigma}(F(\sigma))|<\infty$. }

%\textbf{Locality:} Given a simplex $\sigma \in \Sigma$ and a cover $\{\sigma_1, \cdots, \sigma_n\}$ of $\sigma$, the inclusions 
%$$
%\alpha_{\sigma_i}(F(\sigma_i)) \hookrightarrow G(\sigma_i)
%$$
%induce a map
%$
%\lim \alpha_\ast(F) \circ \chi_{\sigma_1, \cdots, \sigma_n} \to \lim G \circ \chi_{\sigma_1, \cdots, \sigma_n}
%$
%making the following diagram commute:
%\begin{equation}\label{eq:chialpha}
%\begin{tikzcd}
%\alpha_\ast(F)(\sigma) \arrow[r,hook] \arrow[d] & G(\sigma) 
%\arrow[d,hook] \\
%\lim \alpha_\ast(F)\circ \chi_{\sigma_1,\cdots,\sigma_n} \arrow[r] & 
%\lim G\circ \chi_{\sigma_1,\cdots,\sigma_n}
%\end{tikzcd}
%\end{equation}
%Diagram (\ref{eq:chialpha}) and the locality of $G$ implies that the canonical map
%$
% \alpha_\ast(F)(\sigma)  \to \lim \alpha_\ast(F)\circ \chi_{\sigma_1,\cdots,\sigma_n} 
%$
%is injective.
\end{proof}
\begin{defn}
For a simplicial complex $\Sigma$, we define the \emph{functor of sub-event scenarios over $\Sigma$}
\[
\on{ESub}_{\Sigma} \colon \catPSh_{/\Sigma} \to \catSet
\]
as follows:
\begin{itemize}
    \item For a presheaf $F$ on $\Sigma$, define
    \[
    \on{ESub}_{\Sigma}(F) = \left\{F' \colon \catC_{\Sigma}^{\op} \to \catSet \mid~ F' \leq F \text{ and } F' \text{ is {a finite} event scenario} \right\}.
    \]
    \item For a natural transformation $\alpha\colon F \to G$, define
    \[
   \alpha_\ast= \on{ESub}_{\Sigma}(\alpha) \colon \on{ESub}_{\Sigma}(F) \to \on{ESub}_{\Sigma}(G)
    \]
    by sending $F' \in \on{ESub}_{\Sigma}(F)$ to the functor
    \[
    (\alpha \circ j)_\ast(F')\colon \catC^{\op}_{\Sigma} \to \catSet,
    \]
    where $j\colon F' \to F$ is the natural transformation given by the inclusions $F'(\sigma) \hookrightarrow F(\sigma)$.
\end{itemize}
\end{defn}

%\coc{comment the construction above is well-defined %because
%\medskip

{See Definition \ref{def:alphaastF} and Lemma \ref{lem:alphaastF}.}
%}

%\coc{reorganize the following, comes from rel. Groth.}
%

Let $\pi\colon {\Gamma} \to \hat N \Sigma$ be a simplicial complex map. {As mentioned in Section \ref{subsec:nattt} it induces a functor 
$\pi^*\colon \catPSh_{/\Sigma} \to \catPSh_{/\Gamma}$, by sending $F$ to the composite $F \circ \bar \pi$, 
where $\bar \pi \colon \catC_\Gamma \to \catC_\Sigma$ is the corresponding functor to $\pi$.
Obviously,} if $F'$ is an event scenario such that $F' \leq F$, then $\pi^{\ast}(F')$ is also an event scenario satisfying 
$\pi^{\ast}(F') \leq \pi^{\ast}(F)$. {Thus, we have the following definition:}

%\coc{need to fix the ref:
%\bigskip
%This follows.
%}

\begin{defn}\label{def:EsubbFunc}
For a simplicial complex map $\pi\colon \Gamma \to \hat N \Sigma$ and a functor
$F \colon \catC^\op_{\Sigma} \to \catSet$, we define a map 
\begin{equation}\label{eq:Esubbpi}
\on{ESub}_{\pi}(F)\colon \on{ESub}_{\Sigma}(F) \to \on{ESub}_{\Gamma}(\pi^\ast(F)) 
\end{equation}
by sending $F' \in \on{ESub}_{\Sigma}(F)$ to $\pi^{\ast}(F') \in \on{ESub}_{\Gamma}(\pi^\ast(F))$.

Given presheaves $F,G$ on $\Sigma$, and a natural transformation $\alpha \colon F \to G$, one can check that the following diagram commutes:
$$
 \begin{tikzcd}[column sep=huge,row sep=large]
\on{ESub}_{\Sigma}(F) \arrow[d,"\alpha_\ast"]
\arrow[r,"\on{ESub}_{\pi}(F)"] & \on{ESub}_{\Gamma}(\pi^\ast(F)) \arrow[d,"(\Id_{\overline{\pi}}\star \alpha)_\ast"]   \\
\on{ESub}_{\Sigma}(G) \arrow[r,"\on{ESub}_{\pi}(G)"] & \on{ESub}_{\Gamma}(\pi^\ast(G))
\end{tikzcd}
$$
Therefore, the maps in Eq.~\eqref{eq:Esubbpi} form a natural transformation from $\on{ESub}_{\Sigma}$ to $\on{ESub}_{\Gamma}\circ \pi^\ast$.

For a simplicial complex map $\pi\colon \Gamma \to \hat N \Sigma$, we define the following morphism in $\catCat \slice \catSet$:
\begin{equation}\label{dia:naturalpistar}
\begin{tikzcd}[column sep=huge, row sep=large]
\catPSh_{/\Sigma} \arrow[rr,"\pi^{\ast}"]
\arrow[dr,"\on{ESub}_{\Sigma}"',""{name=A,right}] && {\catPSh_{/\Gamma}} \arrow[dl,"\on{ESub}_{\Gamma}",""{name=B,left}] 
\arrow[Rightarrow, from=A, to=B, "\on{ESub}_{\pi}"]\\
&  \catSet &  
\end{tikzcd}
\end{equation}
Thus, we obtain a functor 
$$
\on{ESub}_{-}\colon \catsRel^\op \to \catCat \slice \catSet 
$$
that sends a simplicial complex $\Sigma$ to the functor $\on{ESub}_{\Sigma} \colon \catPSh_{/\Sigma} \to \catSet$, and sends a simplicial complex map 
$\pi\colon \Gamma \to \hat N \Sigma$ to the diagram~\eqref{dia:naturalpistar}. The relative Grothendieck construction of this functor is the \emph{functor of sub-event scenarios}, denoted by
\[
\on{ESub} = \int_{\catsRel^{\op}} \on{ESub}_{-} \colon  \int_{\catsRel^{\op}} \catPSh_{/-} \to \catSet.
\]
\end{defn}

\subsubsection{Possibilistic empirical models}
In this section, we show that the functor of sub-event scenarios is equivalent to the functor of possibilistic empirical models. {We shall use the equivalence in the simplicial setting established in Section~\ref{subsec:Possibsimpdist}. To this end, we begin with the following proposition.}

%defined in section \ref{subsec:Empirmod}. This done using a relative version of the Grothendieck construction
%(Section \ref{sec:gro}).
{
\begin{pro}\label{pro:rhonaturual}
Let $F$ be a presheaf on $\Sigma$. By sending $F' \in \Esubb_{\Sigma}(F)$ to $E(F')$, we obtain a natural isomorphism 
$$
\rho_{\Sigma} \colon \Esubb_{\Sigma} \to \subb_{S(\Sigma)}\circ \Xi_{\Sigma}.
$$
See Definition \ref{def:subXXXX} and the functor in (\ref{eq:functorXi}).
\end{pro}
}
\begin{proof}
{
We first show that
\[
E(F')\in \subb_{S(\Sigma)}(\Xi_{\Sigma}(F))
=
\subb_{S(\Sigma)}(f_F)
\]
for every $F'\in \Esubb_{\Sigma}(F)$. Since $F'\leq F$, it follows that $E(F')$ is a simplicial subset of $E(F)$. By Proposition~\ref{pro:realization of event is bundle}, the simplicial map $f_{F'}=f_F|_{E(F')}$ is a bundle scenario. Moreover, the fibres of $f_{F'}$ are precisely the sets $F'(\sigma)$ for $\sigma\in \Sigma$, and hence are finite. 

Next, we show that $\rho_{\Sigma,F}$ is an isomorphism. Suppose that $F',F''\in \Esubb_{\Sigma}(F)$ satisfy $
E(F')=E(F'')$. For every simplex
$
\sigma=\{v_0,\dots,v_n\}$ such that
$v_0\leq \cdots \leq v_n
$,
we have
\[
\{(v_0,\dots,v_n)\}\times F'(\sigma)
=
(f_{F'})_n^{-1}(v_0,\dots,v_n)
=
(f_{F''})_n^{-1}(v_0,\dots,v_n)
=
\{(v_0,\dots,v_n)\}\times F''(\sigma).
\]
Hence, $F'(\sigma)=F''(\sigma)$, which proves injectivity. 

To prove surjectivity, let $
E'\in \subb_{S(\Sigma)}(f_F)$.
For each simplex
$
\sigma=\{v_0,\dots,v_n\}\in \Sigma$ such that 
$v_0\leq \cdots \leq v_n$, 
define
\[
F'(\sigma)
:=
\left\{
a \mid~
\bigl((v_0,\dots,v_n),a\bigr)\in E'_n
\right\}.
\]
Since the restriction
\[
f_F|_{E'}\colon E'\to S(\Sigma)
\]
is surjective and has finite fibres, we obtain that $0<|F'(\sigma)|<\infty$.

Now let
$\sigma=\{v_0,\dots,v_n\} \subset \tau=\{u_0,\dots,u_m\}$, and let
$\theta\colon [n]\to [m]$ be the corresponding ordinal map. Since $E'$ is a simplicial subset of $E(F)$, we have a commutative diagram}
$$
\begin{tikzcd}[column sep=huge,row sep=large]
E'_m \arrow[r,"(f_F)_m"]
\arrow[d,"\theta^\ast"'] &
S(\Sigma)_m \arrow[d,"\theta^\ast"] \\
E'_n \arrow[r,"(f_F)_n"'] &
S(\Sigma)_n,
\end{tikzcd}
$$
{
which induces a map $F'(\tau)\to F'(\sigma)$. The functoriality of $F'$ follows from the simplicial identities satisfied in $E'$. Furthermore, since $f_F|_{E'}$ is a flasque simplicial map, it follows that $F'$ is a flasque presheaf. Therefore,
$
F'\in \Esubb_{\Sigma}(F).
$. By construction, we have
$$
\rho_{\Sigma,F}(F')
=
E(F')
=
E',
$$
which proves surjectivity.

Finally, we prove the naturality. Let $\alpha\colon F \to G$ be a morphism between presheaves on $\Sigma$, We claim that the diagram commutes}
$$
\begin{tikzcd}[column sep=huge,row sep=large]
\Esubb_{\Sigma}(F) \arrow[r,"\rho_{\Sigma,F}"] \arrow[d,"\Esubb_{\Sigma}(\alpha)"'] & \subb_{S(\Sigma)}(f_F) 
\arrow[d,"\subb_{S(\Sigma)}(\Xi_{\Sigma}(\alpha))"] \\
\Esubb_{\Sigma}(G) \arrow[r,"\rho_{\Sigma,G}"'] & \subb_{S(\Sigma)}(f_G)
\end{tikzcd}
$$
{Indeed, let $F' \in \Esubb_{\Sigma}(F)$, and let $j\colon F' \to F$ denote the natural transformation given by the inclusions $F'(\sigma) \hookrightarrow F(\sigma)$. Then 
$$
\begin{aligned}
    (\subb_{S(\Sigma)}(\Xi_{\Sigma}(\alpha))\circ \rho_{\Sigma,F})(F')
    &=
(\subb_{S(\Sigma)}(\Xi_{\Sigma}(\alpha))(E(F'))  \\
&={\Xi_{\Sigma}(\alpha)(E(F'))} \\
&=E\left( (\alpha \circ j)_\ast(F')\right)  \\
&=\rho_{\Sigma,G} ( (\alpha \circ j)_\ast(F')) \\
&=(\rho_{\Sigma,G} \circ \Esubb_{\Sigma}(\alpha))(F'). 
\end{aligned}
$$
}
\end{proof}
{For every simplicial complex $\Sigma$, there is a natural isomorphism 
\[
\eta_\Sigma\colon \on{Emp}_{\BB,\Sigma} \to \on{ESub}_\Sigma,
\]
given by the following composition: 
\begin{equation}\label{eq:etadefinition}
\eta_\Sigma=\rho_{\Sigma}^{-1}\circ (\zeta_{S(\Sigma)}\ast \Id_{\Xi_{\Sigma}}) \circ \Phi_{\BB,\Sigma},
\end{equation}
where $\ast$ is the horizontal composition between natural transformations. See Propositions \ref{pro:etaXf}, \ref{pro:rhonaturual} and the natural 
isomorphism in (\ref{eq:isoemptoxisdis}).

For a presheaf $F$ on $\Sigma$ the isomorphism 
\begin{equation}\label{eq:isomorphismL}
\eta_{\Sigma,F}\colon \on{Emp}_{\BB,\Sigma}(F) \to \on{ESub}_{\Sigma}(F)   
\end{equation}
acts by sending $p \in \on{Emp}_{\BB,\Sigma}(F)$ to the event scenario 
$L(p)\colon \catC^{\op}_{\Sigma} \to \catSet$, where
$$
L(p)(\sigma)=\set{x\in F(\sigma)\mid~ p_{\sigma}(x)=1}.
$$ 
}

\begin{pro}\label{pro:emp=sub}
{There is a natural isomorphism}  
\[
\eta\colon \on{Emp}_{\BB,-} \to \on{ESub}_-,
\]
defined as $\eta_\Sigma$ at a simplicial complex $\Sigma$.
%$\eta$ 
% from $\Emp_{\BB,-}$ to $\Esubb_{-}$.
\end{pro}

\begin{proof}
{As shown above, }for every $\Sigma$ object in $\catsRel^\op$ we have the isomorphism $(\Id_{{\catPSh_{/\Sigma}} }{,}\eta_{\Sigma})$ in $\catCat\slice \catSet$:
$$
\begin{tikzcd}[column sep=huge, row sep=large]
{\catPSh_{/\Sigma}} 
\arrow[rr,equal]
\arrow[ddr,"\Emp_{\BB,\Sigma}"',""{name=A,right}] && {\catPSh_{/\Sigma}} 
\arrow[ddl,"\Esubb_{\Sigma}",""{name=B,left}] \\
 &\arrow[Rightarrow, from=A, to=B, "{\eta_\Sigma}"]& \\
&  \catSet&  
\end{tikzcd}
$$

Given a simplicial complex map $\pi \colon \Sigma' \to \hat{N} \Sigma$, we need to prove that the following diagram commutes in $\catCat\slice \catSet$:    
$$
\begin{tikzcd}[column sep=huge,row sep=large]
\Emp_{\BB,\Sigma} \arrow[rr,"{(\pi^\ast{,}\Emp_{\BB,\pi})}"] \arrow[dd,"(\Id_{\catPSh_{/\Sigma}}{,}\eta_{\Sigma})"'] && \Emp_{\BB,{\Sigma'}} 
\arrow[dd,"(\Id_{\catPSh_{/\Sigma'}}{,}\eta_{\Sigma'})"] \\
&& \\
\Esubb_\Sigma \arrow[rr,"(\pi^\ast{,} \Esubb_{\pi})"'] && \Esubb_{\Sigma'}
\end{tikzcd}
$$
Given a {presheaf} $F\colon \catC^{\op}_{\Sigma} \to \catSet$, and an empirical model $p=\{p_{\sigma}\}_{\sigma \in \Sigma} \in \Emp_{\BB}(F)$. We denote $\Emp_{\BB,{\pi}}(F)(p)= \{p_{\overline{\pi}(\tau})\}_{\tau \in \Sigma'}$ by $q$, then we have: 
$$
\eta_{\Sigma',F}\left(\Emp_{\BB,{\pi}}(F)(p)\right)= \eta_{\Sigma',F}\left(q\right)=L(q),
$$
and $L(q)(\tau)=\set{x\in F(\overline{\pi}(\tau))\mid~ q_{\tau}(x)=1}=\set{x\in F(\overline{\pi}(\tau)) \mid ~ p_{\overline{\pi}(\tau)}(x)=1}$ for every $\tau \in \Sigma'$. On the other hand, we have 
$$
\Esubb_{\pi}(F)\left(\eta_{\Sigma,F}(p)\right)=\pi^{\ast}\left(L(p)\right)=L(p)\circ \overline{\pi}, 
$$
and $L\left(\overline{\pi}(p)(\tau)\right)=\set{x\in F(\overline{\pi}(\tau))\mid ~ p_{\overline{\pi}(\tau)}(x)=1}$ for every $\tau \in \Sigma'$.
\end{proof}

\begin{thm}\label{cor:EmpB=Esub}
The functors $\Emp_{\BB}\colon \int \catPSh_{/-} \to \catSet$ and $\Esubb\colon \int \catPSh_{/-} \to \catSet$ are 
%equivalent.
naturally isomorphic.
\end{thm}
\begin{proof}
Directly by Propositions \ref{pro:emp=sub} and \ref{pro:NatNat}.
\end{proof}

% 

%\newpage

\section{{Sufficient conditions for extremality}}
 
\label{sec:Sufficient conditions for extremality}

In this section, we develop methods for studying the extremal points of simplicial distributions. {We do so using the characterization of supports of distributions established in Section~\ref{sec:Possibilistic}.} We also introduce mild finiteness conditions on the {domain and codomain} of simplicial maps $f\colon E\to X$ that ensure that the set of simplicial distributions has finitely many extremal points.

{A \emph{polytope in standard form} is a bounded subset of $\mathbb{R}^n$ of the form
$$
\{x\in\mathbb{R}^n \mid~ Ax=b,\; x\geq 0\},
$$
where $A$ is a matrix and $b$ is a vector.}
\begin{pro}\label{pro:polytope of sdist}
Let $X$ be a simplicial set with finitely many non-degenerate simplices, and let $f\colon E\to X$ be a simplicial map with finite fibers, i.e., $|f^{-1}(x)|<\infty$ for every $x\in X_n$, $n\geq 0$. Then the set of simplicial distributions $\on{sDist}(f)$ is a polytope {in standard form.}
\end{pro}
\begin{proof}
{By Definition~\ref{def:simp-dist}, the set of simplicial distributions on $f$ is given by the pullback}
$$
\begin{tikzcd}[column sep=huge,row sep=large]
\sDist(f) \arrow[r,hook] \arrow[d,""'] & \sDist(X,E) 
\arrow[d,"f^\ast"] \\
\set{\delta_X} \arrow[r,hook] & \sDist(X,X)
\end{tikzcd}
$$
{See Example \ref{ex:first examples}. By \cite[Proposition~3.6]{kharoof2026vertex}, the sets $\sDist(X,E)$ and $\sDist(X,X)$ are polytopes in standard form. Since $\{\delta_X\}$ is a singleton, it is also a polytope in standard form. Moreover, the map
$f^\ast$
is affine. Therefore, $\sDist(f)$ is a pullback of polytopes in standard form along affine maps. Since pullbacks of polytopes in standard form are again polytopes in standard form, it follows that $\sDist(f)$ is a polytope in standard form.}
\end{proof}
%\begin{defn}\label{def:vertex}
%An element $v$ of a convex set $V$ is called a \emph{vertex} or \emph{extremal point} if, for every $0<t<1$ and elements $v_1,v_2 \in V$ such that 
%$v=t\cdot v_1 + (1-t)\cdot v_2$, we have $v=v_1=v_2$.   
%\end{defn}

%\coc{for non-neg. reals prime filter is the same as a face.}

\begin{defn}\label{def:vertex}
Let $V$ be a convex set. An element $v\in V$ is called a \emph{vertex} or \emph{extremal point} if, for every $0<t<1$ and elements $v_1,v_2 \in V$ such that 
$v=t v_1 + (1-t)v_2$, we have $v=v_1=v_2$.

A convex subset $U\subseteq V$ is called a \emph{face} of $V$ if, whenever a convex combination
$
\sum_{i=1}^n \alpha_i v_i
$
with $v_i\in V$, $\alpha_i>0$, and $\sum_i\alpha_i=1$, lies in $U$, then $v_i\in U$ for every $i$.
\end{defn}
 
{\begin{pro}\label{pro:vertexfilter}
Let $U$ be a face in a convex set $V$. An element $v\in U$ is a vertex in $V$ if and only 
if it is a vertex in $U$.
\end{pro}
\begin{proof}
If $v \in U$ is a vertex in $V$, then obviously it is a vertex in $U$. Now, suppose that $v$ is a 
vertex in $U$. If there is $v_1,v_2 \in V$ and $0<\alpha<1$ such that $v=\alpha v_1+(1-\alpha)v_2$, then 
$v_1,v_2 \in U$ since
 $U$ is a face. So $v_1=v_2=v$ since $v$ is a vertex in $U$.  
\end{proof}
} 
%\coc{introduce deterministic simplicial distributions in the first section} 
 
There are two kinds of vertices of $\on{sDist}(f)$: deterministic simplicial distributions, which are well known to be vertices, and the remaining vertices, which are contextual simplicial distributions. Our focus in this section is on the \emph{contextual vertices} of $\on{sDist}(f)$.

\subsection{Topological {condition}}\label{subsec:topcharac}

For a distribution $P\in D(U)$ {on a set $U$}, we write $\on{supp}(P)$ for its \emph{support}, i.e., the subset $\{x\in U\mid~ P(x)\neq 0\}$.

%\coc{$q_x^e$ vs $q_x(e)$}

\begin{defn}\label{def:preordersDist}
%Let \(f\colon E \to X\) be a pre-bundle scenario.  
We define a preorder \(\preceq\) on \(\sDist_R(f)\) by declaring that \(q \preceq p\) if, for every \(x \in X_n\) and every \(e \in E_n\),
\[
\on{supp}\left(q_x(e)\right) \subseteq \on{supp}\left(p_x(e)\right).
\]
%\[
%q_x^e \neq 0 \;\; \Rightarrow \;\; p_x^e \neq 0.
%\]
{We will use the notation $x \prec y$ to denote that $x \preceq y$ and $x \neq y$.} We denote by \(p_{\preceq}\) the set of all simplicial distributions \(q \in \sDist_R(f)\) such that \(q \preceq p\). 
\end{defn}

\begin{pro}\label{pro:preservation of preorder}
Let $f\colon E\to X$ be a simplicial map. 
\begin{enumerate}
\item The possibilistic collapse map
$$
\kappa_{f}\colon \sDist(f) \to \sDist_{\BB}(f).
$$
induced by the map in \eqref{eq:piX} {satisfies  
$$
q \preceq p \quad \Longleftrightarrow \quad \kappa_f(q) \preceq  \kappa_f(p).
$$}
\item The $\zeta_f$ isomorphism of \eqref{eq:zetaXfff} satisfies
\[
q \preceq p \quad \Longleftrightarrow \quad \zeta_f(q) \subseteq \zeta_f(p).  
\]
\end{enumerate}
See Equation (\ref{eq:zetaxf}).
\end{pro}

\begin{pro}\label{pro:simprvertmin}
A simplicial distribution $p \in \sDist(f)$ is a vertex if and only if $p$ is minimal with respect to the preorder $\preceq$.
\end{pro}
\begin{proof}
{This follows from Proposition~\ref{pro:polytope of sdist} and \cite[Corollary~2.11]{kharoof2026vertex}.}
\end{proof}

{We show that extremal simplicial distributions are determined by their corresponding possibilistic simplicial distributions, a {general} fact {about probability polytopes} previously observed in \cite{abramsky2016possibilities}.}

\begin{cor}\label{cor:pvertexifandsimplicialAAAAA}
%Given a 
%pre-bundle scenario 
%{simplicial map}
%$f$, 
A simplicial distribution $p\in \sDist(f)$ is a vertex if and only if $\kappa_f^{-1}(\kappa_f(p))=\set{p}$. 
\end{cor}
\begin{proof}
{Suppose that $p$ is a vertex. If $p' \in \kappa_f^{-1}(\kappa_f(p)_{\preceq})$, then by part (1) of Proposition \ref{pro:preservation of preorder} we get that $p'\preceq p$. By Proposition \ref{pro:simprvertmin} we obtain $p'=p$. We proved that $\kappa_f^{-1}(\kappa_f(p)_{\preceq})=\set{p}$, so $\kappa_f^{-1}(\kappa_f(p))=\set{p}$. Now, suppose that $\kappa_f^{-1}(\kappa_f(p))=\set{p}$. Let $p' \preceq p$. Note that 
$\kappa_f(\frac{1}{2}p+\frac{1}{2}p')=\kappa_f(p)$. In other words, $\frac{1}{2}p+\frac{1}{2}p'\in \kappa_f^{-1}(\kappa_f(p))$, so 
$\frac{1}{2}p+\frac{1}{2}p'=p$, which implies that $p'=p$.}
\end{proof}

\begin{pro}\label{pro:identification of preimage kappa}
For a possibilistic simplicial distribution $p\in \on{sDist}_{\BB}(f)$, the subset $\kappa_f^{-1}(p_{\preceq})$ is a face of $\on{sDist}(f)$ and is affinely isomorphic to $\sDist({f|_{\zeta_f(p)}})$. 
\end{pro}
\begin{proof}
{Given $q,s\in \sDist(f)$ and $0<\alpha<1$ such that $\alpha q +(1-\alpha)s \in \kappa_f^{-1}(p_{\preceq})$. That means
$\alpha \kappa_f(q) +(1-\alpha)\kappa_f(s)=\kappa_f(\alpha q +(1-\alpha)s)\preceq p$.
Since $\alpha>0$, we get that $\kappa_f(q) \preceq p$, and 
since $1-\alpha>0$, we get that $\kappa_f(s) \preceq p$. So $q, s \in \kappa_f^{-1}(p_{\preceq})$. We proved that $\kappa_f^{-1}(p_{\preceq})$ is a face.

Now, note that for 
$q \in \sDist(f)$, we have $q \in \sDist(f|_{\zeta_f(p)})$ if and only if
for every $x\in X_n$ and $e \in f_n^{-1}(x)$ the condition
$q_{x}(e)\neq 0$ implies that $e \in \zeta_f(p)_n$, which means that $p_x(e)=1$. 
This is equivalent to say that $\kappa_f(q) \preceq p$. In other words, $q \in \kappa_{f}^{-1}(p_{\preceq})$.}
\end{proof}

\begin{pro}\label{pro:vertex iff does not lift}
A simplicial distribution $p\in \on{sDist}(f)$ is a vertex if and only if every possibilistic simplicial distribution 
$q\prec \kappa_f(p)$ does not lie in the image of $\kappa_f$.
\end{pro}

\begin{proof}
{Suppose first that $p$ is a vertex. If there exists $p' \in \sDist(f)$ such that
$\kappa_f(p') \prec \kappa_f(p)$, then by part (1) of Proposition \ref{pro:preservation of preorder}, we get that $p' \prec p$. This contradicts
Proposition~\ref{pro:simprvertmin}.

Conversely, assume that every $q \prec \kappa_f(p)$ is not liftable. Then the face $\kappa_f^{-1}(\kappa_f(p)_{\preceq})$ (see Proposition 
\ref{pro:identification of preimage kappa}) is equal to $\kappa_f^{-1}(\kappa_f(p))$.
Let $p'$ be a vertex in the face $\kappa_f^{-1}(\kappa_f(p))$. By  Proposition \ref{pro:vertexfilter}, $p'$ is  a vertex in $\Emp(f)$, and hence,
by Corollary~\ref{cor:pvertexifandsimplicialAAAAA}, we have
$$
\set{p'} = \kappa_f^{-1}(\kappa_f(p')).
$$
Meanwhile, $\kappa_f(p') = \kappa_f(p)$, so it follows that
$$
\kappa_f^{-1}(\kappa_f(p')) = \kappa_f^{-1}(\kappa_f(p)).
$$
Therefore, $p' = p$, and we conclude that $p$ is a vertex.}
\end{proof}

{
\begin{cor}\label{cor:minisvertexsimpversion}
If $\kappa_f(p)$ is minimal, then $p$ is a vertex.    
\end{cor}
}

Next, we introduce the connectivity notion for simplicial maps that will serve as our main criterion for detecting extremal simplicial distributions.

\begin{defn}\label{def:strong-connectivity}
Let $f\colon E\to X$ be a simplicial map. Consider two simplices
$e_1\in E_{n_1}$, $e_2\in E_{n_2}$, {and a simplex $x\in X_{m}$}. We write $e_1\sim_x e_2$ if
there exist injective ordinal maps
\[
\theta_1\colon {[m]}\to [n_1],
\qquad
\theta_2\colon {[m]}\to [n_2]
\]
such that
\[
\theta_1^\ast(e_1)=\theta_2^\ast(e_2) \in f_m^{-1}(x),
\]
and the following lifting problems admit unique solutions:
\begin{equation}\label{dia:e_1xe_2}
\begin{tikzcd}[column sep=huge, row sep=large]
\Delta{[m]} \arrow[d,hook,"\theta_1"] \arrow[r,"\theta_1^\ast(e_1)"] & E \arrow[d,"f"]
&
\Delta{[m]} \arrow[d,hook,"\theta_2"] \arrow[r,"\theta_2^\ast(e_2)"] & E \arrow[d,"f"] \\
\Delta[n_1] \arrow[ru,dashed] \arrow[r,"f_{n_1}(e_1)"] & X
&
\Delta[n_2]\arrow[ru,dashed] \arrow[r,"f_{n_2}(e_2)"] & X .
\end{tikzcd}
\end{equation}
The unique solutions are required to be the simplicial maps
\[
\Delta[n_1]\xrightarrow{\,e_1\,}E
\qquad\text{and}\qquad
\Delta[n_2]\xrightarrow{\,e_2\,}E.
\]
We say that $e_1$ and $e_2$ are \emph{$f$-strongly connected} if
there exist simplices $e'_1,\dots,e'_k$ of $E$ {and simplices $x_1,\dots,x_{k+1}$ of $X$},  such that
\[
e_1\sim_{x_1} e'_1\sim_{x_2} \cdots \sim_{x_k} e'_k\sim_{x_{k+1}} e_2.
\]
This defines an equivalence relation, denoted by $\sim_f$.
We say that $f$ is \emph{strongly connected} if every pair of
generator simplices of $E$ is $f$-strongly connected.
\end{defn}

\begin{pro}\label{pro:uniqeextension}
Let $p$ be a simplicial distribution on 
%a pre-bundle scenario
a simplicial map
$f\colon E \to X$, and 
$E' := \zeta_{f}(\kappa_f(p))$ denote the bundle scenario of the possibilistic collapse.
%(see Equation~\eqref{eq:zetaxf}). 
Given $e_1\in E'_{n_1}$ and $e_2\in E'_{n_2}$, if
$e_1 \sim_{f|_{E'}} e_2$, then
$$
p_{f_{n_1}(e_1)}(e_1)=p_{f_{n_2}(e_2)}(e_2).
$$
\end{pro}

\begin{proof}
It is enough to prove in the case that $e_1 \sim_x e_2$. Suppose there exist {injective} ordinal maps $\theta_1\colon  [m]\to [n_1]$ and 
$\theta_2\colon  [m]\to [n_2]$ such that
$
\theta_1^\ast(e_1)
=
\theta_2^\ast(e_2)
$, and the following squares admit unique liftings:
$$
\begin{tikzcd}[column sep=huge, row sep=large]
\Delta[m] \arrow[d,hook,"\theta_1"] \arrow[r,"\theta_1^\ast(e_1)"] & E' \arrow[d,"f|_{E'}"]
&
\Delta[m] \arrow[d,hook,"\theta_2"] \arrow[r,"\theta_2^\ast(e_2)"] & E' \arrow[d,"f|_{E'}"] \\
\Delta[n_1] \arrow[r,"f_{n_1}(e_1)"] & X
&
\Delta[n_2] \arrow[r,"f_{n_2}(e_2)"] & X
\end{tikzcd}
$$
Denote $\sigma_i := f_{n_i}(e_i)$. By the uniqueness of lifting in the left square and since
$\zeta_{f}(\kappa_f(p))=E'$, we obtain
$$
p_{\theta^\ast_1({\sigma_1})}(\theta^\ast_1(e_1))
= D(\theta^\ast_1)(p_{\sigma_1})(\theta^\ast_1(e_1))
= \sum_{\substack{e' \in E'_{n_1} \\ \theta_1^\ast(e') =  \theta_1^\ast(e) \\ f_{n_1}(e')=\sigma_1 }} p_{\sigma_1}(e')
= p_{\sigma_1}(e_1).
$$
Similarly, we have $
p_{\theta^\ast_2(\sigma_2)}(\theta^\ast_2(e_2))
= p_{\sigma_2}(e_2)$. Note that 
%$\theta_1^\ast(e_1)
%=
%\theta_2^\ast(e_2)$ 
%implies that 
$$
\theta^\ast_1({\sigma_1})=\theta_1^\ast(f_{n_1}(e_1))
=x=\theta_2^\ast(f_{n_2}(e_2))=
\theta^\ast_2({\sigma_2}).
$$ 
So we get that $p_{\sigma_1}(e_1)=p_{\sigma_2}(e_2)$.
\end{proof}
\begin{thm}\label{thm:strongisvertx}
Let $p$ be a simplicial distribution on a simplicial map
$f\colon E\to X$, and let
\[
g := f|_{\zeta_f(\kappa_f(p))}.
\]
If $g$ is strongly connected, then $p$ is a vertex of $\sDist(f)$.
\end{thm}
\begin{proof}
By Corollary~\ref{cor:gentogen} and Proposition~\ref{pro:uniqeextension}, there exists a unique value $t$ such that
$p_{f_n(e)}(e)=t$ whenever $p_{f_n(e)}(e)\neq 0$.
For a generator simplex $\sigma\in X_n$ and $e \in f_n^{-1}(\sigma)$, the simplex $e$ is also a generator.
Since
$$
\sum_{e \in f_n^{-1}(\sigma)} p_{\sigma}(e)=\sum_{e \in g_n^{-1}(\sigma)} p_{\sigma}(e)=1,
$$
we conclude that $t=\frac{1}{|g_n^{-1}(\sigma)|}$.
Thus, $p$ is uniquely determined by its Boolean 
%shadow 
collapse
under $\kappa_f$. Hence,
by Corollary~\ref{cor:pvertexifandsimplicialAAAAA}, $p$ is a vertex of $\sDist(f)$.
\end{proof}

{We now give an alternative explanation of the theorem above in terms of minimality.

\begin{pro}\label{pro:fs.cimpliessub=E}
If $f \colon E\to X$ is strongly connected, then $\subb_X(f)=\set{E}$.    
\end{pro}
}
\begin{proof}
{
Let $E'\in \subb_X(f)$. We first show that if
$\theta^\ast(e)\in E'_m$ for some ordinal map $\theta\colon [m]\to [n]$ and some simplex
$e\in E_n$, then $e\in E'_n$. Indeed, since $f|_{E'}$ is a bundle scenario, the following lifting problem admits a solution:
}
$$
\begin{tikzcd}[column sep=huge,row sep=large]
\Delta[m] \arrow[r,"\theta^\ast(e)"] \arrow[d,"\theta"'] & E' \arrow[d,"f|_{E'}"] \\
\Delta[n] \arrow[r,"f_n(e)"]  \arrow[ru,"{e'}"] & X 
\end{tikzcd}
$$
{Since $f$ is strongly connected, the lift is unique, and therefore $e'=e$. In particular $e \in E'_n$. As a result, there exist a generator $e_1$ of $E$ that belongs to $E'$. Now let $e_2$ be another generator of $E$ such that $e_1 \sim_x e_2$. Then, we have Diagrams (\ref{dia:e_1xe_2}) where 
$$
\theta_2^\ast(e_2)=
\theta_1^\ast(e_1) \in E'_{m}.
$$ 
By the argument above, it follows that $e_2$ is in $E'$. Since $f$ is strongly connected, every generator of $E$ is equivalent to $e_1$ under the relaion $\sim_f$. Hence every generator of $E$ belongs to $E'$. Thus $E'=E$. }
\end{proof}
{
\begin{cor}\label{cor:s.cimpliesmin}
Let $f\colon E\to X$ be a simplicial map, and let $E'\in \Sub_X(f)$.
If the restricted map $g=f|_{E'}$
is strongly connected, then $E'$ is minimal in $\Sub_X(f)$.
\end{cor}

As a consequence, if $p \in \sDist(f)$
such that $f|_{\zeta_f(\kappa_f(p))}$
is strongly connected, then 
$\zeta_f(\kappa_f(p))$ is minimal in $\Sub_X(f)$. By part~(2) of Proposition~\ref{pro:preservation of preorder}, it follows that
$\kappa_f(p)$
is minimal in $\sDist_{\BB}(f)$. Therefore, Corollary~\ref{cor:minisvertexsimpversion} implies that $p$ is a vertex.
}
%{===} 

\subsection{Categorical {condition}}
 
Next, we {provide a categorical criterion for detecting} extremal empirical models in terms of their associated event scenarios.

%\coc{notation $q_\sigma^a$} 

\begin{defn}\label{def:preorder}
%Let \(F \in {\on{PSh}(\Sigma)}\).
% be a pre-event scenario. 
We define a preorder \(\preceq\) on {\(\Emp_R(F)\)} by declaring that \(q \preceq p\) if, for every \(\sigma \in \Sigma\) and every \(a \in F(\sigma)\),
\[
\on{supp}(q_{\sigma}(a)) \subset \on{supp}(p_{\sigma}(a)).
\]
We denote by \(p_{\preceq}\) the set of all empirical models \(q \in {\Emp_R(F)}\) such that \(q \preceq p\).
\end{defn}

\begin{pro}\label{pro:preservation of preorder on empirical}
Let $F$ be a preseheaf on a simplicial complex $\Sigma$. 
\begin{enumerate}
\item The possibilistic collapse map
$$
\kappa_{f}\colon \Emp(F) \to \Emp_{\BB}(F).
$$
induced by the map in \eqref{eq:piX}  {satisfies  
$$
q \preceq p \quad \Longleftrightarrow \quad \kappa_F(q) \preceq  \kappa_F(p).
$$}
\item The isomorphism $\eta_F$ of \eqref{eq:isomorphismL} satisfies
\[
p \preceq q \quad \Longleftrightarrow \quad \eta_F(p) \leq \eta_F(q).  
\]
\item The isomorphism $\Phi_F$ of Proposition \ref{pro:isomorEventBundle} preserves the preorder:
\[
p \preceq q \quad \Longleftrightarrow \quad \Phi_F(p) \preceq \Phi_F(q).  
\]
\end{enumerate}
\end{pro}

\begin{pro}\label{pro:emprvertmin}
An empirical model  
$p \in \Emp(F)$ is a vertex if and only if $p$ is minimal with respect to the preorder $\preceq$.
\end{pro} 
\begin{proof}
{By Proposition~\ref{pro:isomorEventBundle}, $p$ is a vertex of $\Emp(F)$ if and only if $\Phi_F(p)$ is a vertex of $\sDist(f_F)$. 
By Proposition~\ref{pro:simprvertmin}, this holds if and only if $\Phi_F(p)$ is minimal in $\sDist(f_F)$ with respect to the preorder $\preceq$. 
Finally, by part~(3) of Proposition~\ref{pro:preservation of preorder on empirical}, this is equivalent to $p$ being minimal in $\Emp(F)$ with respect to the preorder $\preceq$.}
\end{proof}

{Now, in order to prove that extremal empirical models are determined by their corresponding possibilistic empirical models, we will use the
analogous results established in Section~\ref{subsec:topcharac}. To this end, we consider the following commutative diagram:}
\begin{equation}\label{dia:callpsingfromemptosDist}
\begin{tikzcd}
\Emp(F) \arrow[rr,"\Phi_{\RR_{\geq 0},F}","\cong"'] \arrow[dd,"\kappa_{F}"'] && \sDist(f_F) 
\arrow[dd,"\kappa_{f_F}"] \\
&& \\
\Emp_{\BB}(F) \arrow[rr,"\Phi_{\BB,F}","\cong"'] && \sDist_{\BB}(f_F) 
\end{tikzcd}
\end{equation}
\begin{cor}\label{cor:pvertexifand} 
An empirical model 
$p\in \Emp(F)$ is a vertex if and only if $\kappa_F^{-1}(\kappa_F(p))=\set{p}$.    
\end{cor}
\begin{proof}
{Denote $\Phi_F=\Phi_{\RR_{\geq 0},F}$. Diagram (\ref{dia:callpsingfromemptosDist}) implies that 
$$
\Phi_F\left(\kappa_F^{-1}(\kappa_F(p))\right)=\kappa_{f_F}^{-1}\left(\kappa_{f_F}(\Phi_{\BB,F}(p))\right)
$$
So we get the result by Proposition \ref{pro:isomorEventBundle} and Corollary \ref{cor:pvertexifandsimplicialAAAAA}.}
\end{proof}

\begin{pro}\label{pro:preqpfilter}
For a possibilistic empirical model 
$p \in \on{Emp}_{\BB}(F)$, the subset $\kappa_{F}^{-1}(p_{\preceq})$ is a 
face that is affinely isomorphic to ${\Emp}(\eta_F(p))$.
\end{pro}
\begin{proof}
{Denote $f=f_F$. By part (3) of Proposition \ref{pro:preservation of preorder on empirical} and Diagram (\ref{dia:callpsingfromemptosDist}), we have 
$$
\kappa_{F}^{-1}(p_{\preceq})=\Phi_{\RR_{\geq 0},F}^{-1}\left(\kappa_{f}^{-1}(\Phi_{\BB,F}(p)_{\preceq})\right).
$$
Therefore, by Proposition \ref{pro:identification of preimage kappa}, 
$\kappa_{F}^{-1}(p_{\preceq})$ is a face of $\Emp(F)$ and it is affinely isomorphic to 
$$
\Phi_{\RR_{\geq 0},F}^{-1}\left(\sDist(f|_{\zeta_{f}\left(\Phi_{\BB,F}(p)\right)})\right),
$$
which is equal to 
$$
\begin{aligned}
\Phi_{\RR_{\geq 0},F}^{-1}\left(\sDist(f|_{\left(\rho_F\circ \rho^{-1}_F\right)\left(\zeta_{f}\left(\Phi_{\BB,F}(p)\right)\right)})\right)
&=\Phi_{\RR_{\geq 0},F}^{-1}\left(\sDist(f|_{\rho_F \left(\eta_F(p)\right)})\right) \\
&=\Phi_{\RR_{\geq 0},F}^{-1}\left(\sDist(f_{\eta_F(p)})\right) \\
&=\Emp(\eta_F(p)),
\end{aligned}
$$
see Equation (\ref{eq:etadefinition}).}
\end{proof}

\begin{pro}\label{pro:vertex iff does not lift Empversion}
An empirical model $p\in \on{Emp}(F)$ is a vertex if and only if
every {possibilistic empirical model} $q \prec \kappa_F(p)$ does not lie in the image of $\kappa_F$.
\end{pro}

\begin{proof}
{Denote by $\Phi_F=\Phi_{\RR_{\geq 0},F}$. By Proposition~\ref{pro:isomorEventBundle}, $p$ is a vertex of $\Emp(F)$ if and only if $\Phi_F(p)$ is a vertex of $\sDist(f_F)$. By Proposition~\ref{pro:vertex iff does not lift}, this is equivalent to
\[
\kappa_{f_F}^{-1}\bigl(\kappa_{f_F}(\Phi_F(p))_{\preceq}\bigr)
=
\kappa_{f_F}^{-1}\bigl(\kappa_{f_F}(\Phi_F(p))\bigr).
\]

By the commutative diagram~(\ref{dia:callpsingfromemptosDist}), we have
\[
\kappa_F^{-1}\bigl(\kappa_F(p)_{\preceq}\bigr)
=
\Phi_F^{-1}\Bigl(
\kappa_{f_F}^{-1}\bigl(\kappa_{f_F}(\Phi_F(p))_{\preceq}\bigr)
\Bigr),
\]
and
\[
\kappa_F^{-1}\bigl(\kappa_F(p)\bigr)
=
\Phi_F^{-1}\Bigl(
\kappa_{f_F}^{-1}\bigl(\kappa_{f_F}(\Phi_F(p))\bigr)
\Bigr).
\]
Since $\Phi_F$ is an isomorphism, we conclude that $p$ is a vertex if and only if
\[
\kappa_F^{-1}\bigl(\kappa_F(p)_{\preceq}\bigr)
=
\kappa_F^{-1}\bigl(\kappa_F(p)\bigr).
\]
Equivalently, every possibilistic empirical model $q \prec \kappa_F(p)$ does not lie in the image of $\kappa_F$.
}
\end{proof}

{
\begin{cor}\label{cor:minisvertex}
If $\kappa_F(p)$ is minimal, then $p$ is a vertex.    
\end{cor}
}

\begin{lem}\label{lem:Fa=Ga}
Let 
$\catC$ be a category, 
$F,G\colon \catC^{\op} \to \catSet$ be {non-trivial flasque functors such that $G \leq F$ (see Definition~\ref{def:Funcrelations}),} 
%{functors that are} non-trivial and locally surjective, 
%functors (Definition \ref{def:nontrivlocalsurj}) such that $G \leq F$ (see Definition~\ref{def:Funcrelations}), 
and let
%let 
$s\colon a \to b$ be a morphism of $\catC$. 
%Then
\begin{enumerate}
    \item If $F(b)=G(b)$, then $F(a)=G(a)$.
    \item If $F(s)$ is an isomorphism and $F(a)=G(a)$, then $F(b)=G(b)$.
\end{enumerate}
\end{lem}
\begin{proof}
We have the following commutative diagram:
\begin{equation}\label{eq:GtoF}
\begin{tikzcd}[column sep=huge,row sep=large]
G(b) \arrow[r,hook] \arrow[d,two heads,"G(s)"] & F(b) 
\arrow[d,two heads,"F(s)"] \\
G(a) \arrow[r,hook] & F(a)
\end{tikzcd}
\end{equation}
For $(1)$, if $G(b)=F(b)$, then the surjectivity of $F(s)$ implies $G(a)=F(a)$. For $(2)$, suppose 
$G(a)=F(a)$. Since $G(s)$ is surjective, for $x\in F(b)$ there exists $y \in G(b)$ such that $G(s)(y)=F(s)(x)$. 
By the commutativity of Diagram (\ref{eq:GtoF}), we have $G(s)(y)=F(s)(y)$, hence $F(s)(x)=F(s)(y)$. The injectivity of $F(s)$ 
then implies $x=y$.
\end{proof}

\begin{defn}\label{def:F-strong-connectivity}
Let $\catC$ be a finite poset, and let
$F\colon \catC^{\op}\to \catSet$ be a functor. Two objects
$a,b\in \catC$ are said to be \emph{$F$-strongly connected} if there
exists a zigzag in $\catC$
\begin{equation}\label{eq:Zigzag}
a=c_0 \longleftrightarrow c_1 \longleftrightarrow \cdots
\longleftrightarrow c_m=b
\end{equation}
such that, for each $i$, the corresponding morphism between
$c_i$ and $c_{i+1}$ is sent by $F$ to an isomorphism.
We say that $F$ is \emph{strongly connected} if every pair
of maximal objects of $\catC$ is $F$-strongly connected.
\end{defn}

\begin{pro}\label{pro:zigzagisom}
Let $\catC$ be a finite poset, and let
$F,G\colon \catC^{\op}\to \catSet$ be 
non-trivial flasque
%locally surjective
functors. Suppose that $F$ is strongly connected. If
$G\leq F$ and there exists a maximal object $c$ of $\catC$ such that
$G(c)=F(c)$, then $G=F$.
\end{pro}

\begin{proof}
Suppose that $G \leq F$ and let $c$ be a maximal object of $\catC$ such that $G(c)=F(c)$.
We prove that $G(x)=F(x)$ for every object 
$x\in\catC$. For any maximal object $b$ there exists a zigzag as in \eqref{eq:Zigzag}, such that $F(s)$ is an isomorphism for every morphism $s$ in the zigzag.
By Lemma~\ref{lem:Fa=Ga} we obtain the equality $G=F$ along this zigzag. In particular, we conclude that
$G(b)=F(b)$. Now let $x$ be an arbitrary object of $\catC$.
Since $\catC$ is a finite poset, there exists a maximal object $b$ and a morphism
$x \to b$.
Applying part $(1)$ of Lemma~\ref{lem:Fa=Ga} to this morphism and using the equality $G(b)=F(b)$, we obtain
$
G(x)=F(x)
$.
%Therefore $G=F$.
\end{proof}

\begin{thm}\label{thm:1vert}
Let $F$ be presheaf on $\Sigma$, and $p\in \on{Emp}(F)$. Suppose
that the functor $\eta_F(\kappa_F(p))$ is strongly connected. If for every
$G\in {\Esubb}(F)$ satisfying
$
G\leq \eta_F(\kappa_F(p)),
$
there exists a maximal simplex $\sigma\in \Sigma$ such that
\[
G(\sigma)=\eta_F(\kappa_F(p))(\sigma),
\]
then
$p$ is a vertex of $\on{Emp}(F)$ 
\end{thm}
\begin{proof} 
Assume that the stated condition holds.
By Proposition~\ref{pro:zigzagisom}, $\eta_{F}(\kappa_F(p))$ is minimal.
Therefore, by {part (2) of Proposition~\ref{pro:preservation of preorder on empirical}}, $\kappa_F(p)$ is minimal.
It then follows from Corollary~\ref{cor:minisvertex} that $p$ is a vertex.
\end{proof}

%{===}

\begin{pro}\label{pro:strongconcimplieszigzag}
Let $F$ be a presheaf on $\Sigma$, and set $f:=f_F$. Let
$p\in \on{Emp}(F)$, and define
\[
E' := \zeta_f\bigl(\kappa_f(\Phi_{F}(p))\bigr),
\qquad
g := f|_{E'},
\]
{where $\Phi_F=\Phi_{\RR_{\geq 0},F}$ (see Proposition~\ref{pro:isomorEventBundle}).} If $g$ is strongly connected, then 
\begin{enumerate}
    \item $\eta_F(\kappa_F(p))$ is strongly connected.
    \item {$\eta_F(\kappa_F(p))$ is minimal in $\Esubb_{\Sigma}(F)$.}
\end{enumerate}
    
\end{pro}

\begin{proof}
By Corollary~\ref{cor:gentogen} and Proposition~\ref{pro:uniqeextension}, there exists a value $t$ such that for every simplex $x \in S(\Sigma)_n$ and every $e \in E(F)_n$, if $\Phi_F(p)_{x}(e) \neq 0$, then $\Phi_F(p)_{x}(e) = t$. 
Since 
$$
\sum_{e \in E(F)_{n}} \Phi_F(p)_{x}(e) = 1,
$$
it follows that there exists a natural number $k$ such that
$$
\left| \{ e \in E(F)_{n} \mid \Phi_F(p)_{x}({e}) \neq 0 \} \right| = k
$$
for every simplex $x$ in $S(\Sigma)_n$. Therefore, we have 
$$
\left| \{ a \in F(\sigma) \mid p_{\sigma}(a) \neq 0 \} \right| = k
$$
for every $\sigma \in \Sigma$ {(see Equation (\ref{eq:PhiFFF}))}. Hence, the set $\eta_F(\kappa_F(p))(\sigma)$ has cardinality $k$. Since $\eta_F(\kappa_F(p))$ is locally surjective, it follows that for every inclusion $s \colon \tau 
\hookrightarrow \sigma$, the map $\eta_F(\kappa_F (p))(s)$ is an isomorphism.

The assumption that every pair of generator simplices of $\zeta_f(\kappa_f(\Phi_F(p)))$ is strongly connected implies that $\Sigma$ is connected. Therefore, for every pair of maximal simplices $\sigma_1, \sigma_2 \in \Sigma$, there exists a zigzag in $\catC_{\Sigma}$ as in 
Diagram~\eqref{eq:Zigzag}, and {for every $s$ in the zigzag} $\eta_F(\kappa_F(p))(s)$ is an isomorphism. 

{For part~(2), Corollary~\ref{cor:s.cimpliesmin} implies that
$
E'=\zeta_f\bigl(\kappa_f(\Phi_F(p))\bigr)
$
is minimal in $\subb_{S(\Sigma)}(f)$.
By part~(2) of Proposition~\ref{pro:preservation of preorder}, $\kappa_f(\Phi_F(p))$
is minimal in $\sDist_{\BB}(f)$.
By Diagram (\ref{dia:callpsingfromemptosDist}), we have  
$$
\kappa_f(\Phi_F(p))=\Phi_{\BB,F}(\kappa_F(p)).
$$
So by part (3) of Proposition \ref{pro:preservation of preorder on empirical}, it follows that $\kappa_F(p)$ is minimal in $\Emp_{\BB}(F)$. 
Finally, by part (2) of Proposition \ref{pro:preservation of preorder on empirical},
$
\eta_F(\kappa_F(p))
$
is minimal in $\Esubb_{\Sigma}(F)$.}
\end{proof}

Proposition~\ref{pro:strongconcimplieszigzag} implies that the sufficient condition of Theorem~\ref{thm:strongisvertx} implies the sufficient condition of Theorem~\ref{thm:1vert}. Thus, the former provides a stronger, but more restrictive, criterion for vertex characterization. {In section \ref{subsect:compcharac},
we present an example of a vertex $p \in \Emp(F)$ that satisfies the condition of Theorem~\ref{thm:1vert}, while $\Phi_F(p)$ does not satisfy the condition of Theorem~\ref{thm:strongisvertx}.} 

%\newpage

\section{Examples}

   \label{sec:Examples}

In this section, we present examples of simplicial distributions {and empirical models}, illustrating how Theorems~\ref{thm:strongisvertx} and~\ref{thm:1vert} can be used to detect extremal {objects}. The examples {of simplicial distributions} are defined on simplicial set maps given by projection onto the first factor:
\[
f_{X,m}\colon X\times \Delta_{\ZZ_m}\to X,
\]
where $\Delta_{\ZZ_m}$ is the simplicial set introduced in Example~\ref{ex:first examples}. {The examples of empirical models, on the other hand, arise from the presheaves of events}
\[
\eE_T\colon \catC_{\Sigma}^{\op} \to \catSet,
\]
{where $T=(\Sigma,\ZZ_m)$, as described in Example~\ref{ex:eventpreshef}.}

{In this section, when $p$ is either a simplicial distribution or an empirical model, we use the notation $p_{\sigma}^a$ instead of both $p_{\sigma}(\sigma,a)$ and $p_{\sigma}(a)$.}

\subsection{Low dimensional examples}

We begin with the circle scenario studied in \cite{kharoof2024extremal}, for which the extremal simplicial distributions are completely determined.

\begin{figure}[h!] 
  \centering
\includegraphics[width=0.24\linewidth]{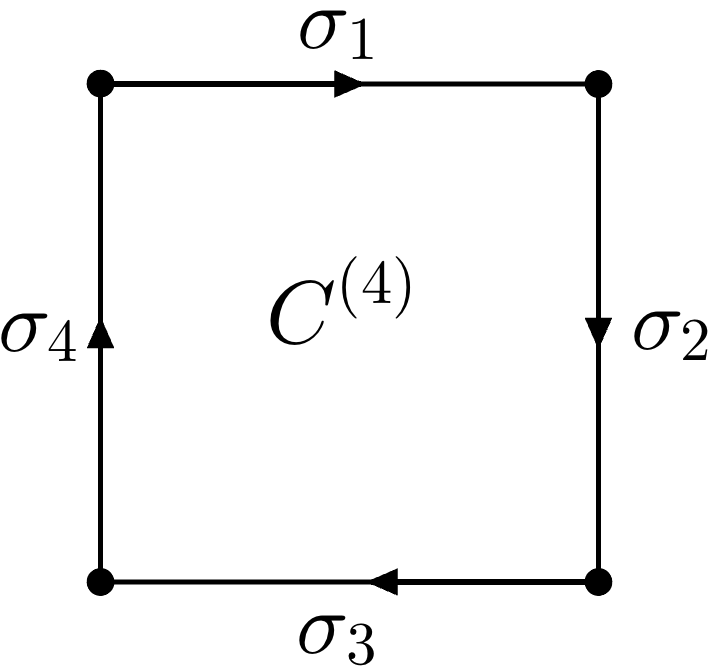}
\caption{The circle with four edges}
\label{fig:Circle}
\end{figure}
 
\begin{ex}\label{ex:circle}
The \emph{(simplicial) circle} $C$ 
is the simplicial set specified by 
a sequence of pairwise distinct $1$-simplices $\sigma_1,\cdots,\sigma_n\in C{_1}$ satisfying
$$
d_{0}(\sigma_1)=d_{1}(\sigma_2)\, , \, 
d_{0}(\sigma_2)=d_{1}(\sigma_3)\, ,\cdots,\, d_{0}(\sigma_{n-1})=d_{1}(\sigma_n)\, , \,
d_{0}(\sigma_n)=d_{1}(\sigma_1).
$$ 
We sometimes write $C=C^{(n)}$ to indicate that the circle has $n$ edges. See Figure \ref{fig:Circle}.

The \emph{$k$-cyclic bundle scenario} over the circle $C^{(n)}$ is the simplicial map
$$
f = f^{n,k} \colon C^{(nk)} \to C^{(n)},
$$
defined by setting
$
f_{\sigma_j} = \sigma_{[j]},
$
where $[j]$ denotes the residue class of $j$ modulo $n$. 
Note that every pair of edges in $C^{(nk)}$ is $f$-strongly connected.

Given the projection map $f_{C^{(n)},m}\colon C^{(n)}\times \Delta_{\ZZ_m} \to C^{(n)}$, 
a simplicial distribution 
\[
p \colon C^{(n)} \to D(C^{(n)}\times \Delta_{\ZZ_m})
\]
is called a \emph{$k$-order cycle distribution} on $f_{C^{(n)},m}$, where $1 \le k \le m$, if there exists a finite sequence
\[
\bigl(
a^{(1)}_1,\ldots,a^{(1)}_n;\;
a^{(2)}_1,\ldots,a^{(2)}_n;\;
\ldots;\;
a^{(k)}_1,\ldots,a^{(k)}_n
\bigr)
\]
of elements in $\ZZ_m$ such that  
$a^{(j)}_i \ne a^{(s)}_i$ for every $1 \le i \le n$ and $j \ne s$,  
and the distribution is defined by
\[
p^{(a,b)}_{\sigma_i} =
\begin{cases}
\displaystyle \frac{1}{k} 
& \text{if } (a,b) = \bigl(a^{(j)}_i,\, a^{(j)}_{i+1}\bigr)
\text{ for some } 1 \le j \le k,\\[6pt]
0 & \text{otherwise},
\end{cases}
\qquad
\]
for $1 \le i \le n-1$, and
\[
p^{(a,b)}_{\sigma_n} =
\begin{cases}
\displaystyle \frac{1}{k} 
& \text{if } (a,b) = \bigl(a^{(j)}_n,\, a^{(j+1)}_1\bigr)
\text{ for some } 1 \le j \le k-1,\\[6pt]
\displaystyle \frac{1}{k} 
& \text{if } (a,b) = \bigl(a^{(k)}_n,\, a^{(1)}_1\bigr),\\[6pt]
0 & \text{otherwise}.
\end{cases}
\]
The restricted scenario
$$
f_{C^{(n)},m}\big|_{\zeta_{f}(\kappa_f(p))}
$$
is isomorphic to the $k$-cyclic scenario $f^{n,k}$. Consequently, by Theorem~\ref{thm:strongisvertx}, every $k$-order cycle distribution is a vertex of $\sDist(f_{C^{(n)},m})$. This reproduces the main result of~\cite[Corollary 4.7]{kharoof2024extremal}.
\end{ex}

Next, we consider a two-dimensional example.

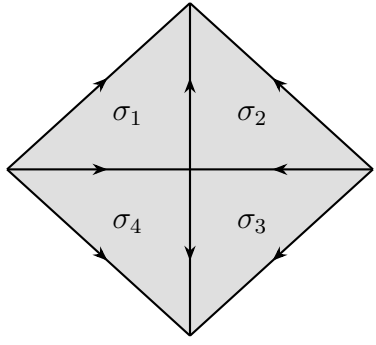
\begin{figure}[h]
\centering
\begin{tikzpicture}[
    scale=1.1,
    oriented/.style={
        line width=0.8pt,
        postaction={decorate},
        decoration={
            markings,
            mark=at position 0.55 with {\arrow{Stealth[length=2mm]}}
        }
    },
    celllabel/.style={font=\large, inner sep=1pt}
]

% vertices
\coordinate (c) at (0,0);
\coordinate (t) at (0,2);
\coordinate (l) at (-2.2,0);
\coordinate (r) at (2.2,0);
\coordinate (b) at (0,-2);

% filled 2-simplices
\fill[gray!25] (l)--(t)--(c)--cycle;
\fill[gray!25] (c)--(t)--(r)--cycle;
\fill[gray!25] (l)--(c)--(b)--cycle;
\fill[gray!25] (c)--(r)--(b)--cycle;

% oriented edges
\draw[oriented] (l) -- (t);
\draw[oriented] (r) -- (t);

\draw[oriented] (c) -- (t);
\draw[oriented] (c) -- (b);

\draw[oriented] (l) -- (c);
\draw[oriented] (r) -- (c);

\draw[oriented] (l) -- (b);
\draw[oriented] (r) -- (b);

% cell labels
\node[celllabel] at (-0.75,0.65) {$\sigma_1$};
\node[celllabel] at ( 0.75,0.65) {$\sigma_2$};
\node[celllabel] at (-0.75,-0.65) {$\sigma_4$};
\node[celllabel] at ( 0.75,-0.65) {$\sigma_3$};

\end{tikzpicture}
\caption{A disk triangulated into four triangles.}
\label{fig:diamond-cells}
\end{figure}

\begin{ex}\label{ex:disk}
Let $X$ be the simplicial set generated by four $2$-simplices
$\sigma_1,\sigma_2,\sigma_3,\sigma_4$, glued along faces as follows:
$$
d_0(\sigma_1)=d_0(\sigma_2), \; d_2(\sigma_2)=d_2(\sigma_3), \;
d_0(\sigma_3)=d_0(\sigma_4), \;
d_2(\sigma_4)=d_2(\sigma_1).
$$
See Figure~\ref{fig:diamond-cells}. 
We define a simplicial distribution $p$ on the 
%pre-bundle scenario 
projection map
$$
f=f_{X,4} \colon X \times \Delta_{\ZZ_4} \to X
$$ 
by setting
$$
p_{\sigma_i}^{(a,b,c)}=
\begin{cases}
\frac{1}{4} & \text{if} \;\; (a,b,c) \in A    \\
0 & \text{otherwise,}
\end{cases}
\qquad
p_{\sigma_3}^{(a,b,c)}=
\begin{cases}
\frac{1}{4} & \text{if} \;\; (a,b,c) \in B    \\
0 & \text{otherwise,}
\end{cases}
$$
where $i\in \set{1,2,4}$ and 
$$
A=\set{(0,0,0),(1,0,1),(2,0,2),(3,0,3)},
\qquad B=\set{(0,0,3),(1,0,0),(2,0,1),(3,0,2)}.
$$
From the 
%geometric shape 
geometry
of $\zeta_{f}\!\left(\kappa_f(p)\right)$ (see Figure~\ref{fig:four-disks}),
one sees that $p$ satisfies the condition of Theorem~\ref{thm:strongisvertx}; namely, {$f|_{\zeta_f(\kappa_f(p))}$} is strongly connected.
Therefore, $p$ is a vertex of $\sDist(f)$.
In fact, the geometric realization of $\zeta_{f}\!\left(\kappa_f(p)\right)$ is a disk composed of $16$ triangles. 
\end{ex}

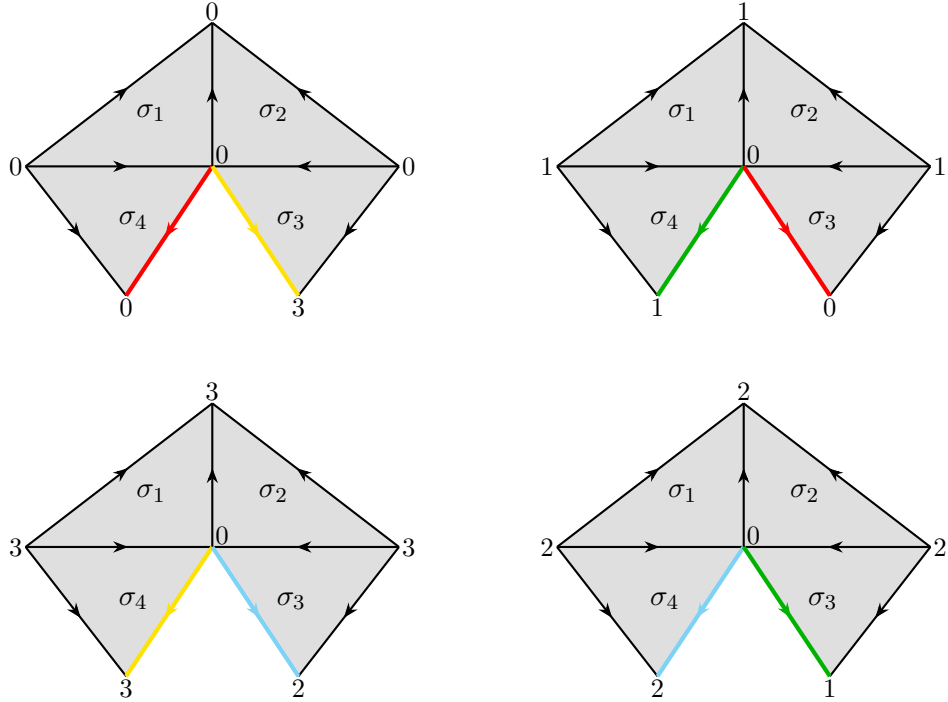
\begin{figure}[h]
\centering
\begin{tikzpicture}[
    scale=0.95,
    oriented/.style={
        line width=0.8pt,
        postaction={decorate},
        decoration={
            markings,
            mark=at position 0.55 with {\arrow{Stealth[length=2mm]}}
        }
    },
    colored/.style={
        line width=1.6pt,
        postaction={decorate},
        decoration={
            markings,
            mark=at position 0.55 with {\arrow{Stealth[length=2mm]}}
        }
    },
    vlabel/.style={font=\small, inner sep=1pt}
]

\newcommand{\DiskPanel}[9]{%
\begin{scope}[shift={(#1,#2)}]

% vertices
\coordinate (c)  at (0,0);
\coordinate (t)  at (0,2);
\coordinate (l)  at (-2.6,0);
\coordinate (r)  at (2.6,0);
\coordinate (bl) at (-1.2,-1.8);
\coordinate (br) at (1.2,-1.8);

% filled 2-simplices
\fill[gray!25] (l)--(t)--(c)--cycle;
\fill[gray!25] (c)--(t)--(r)--cycle;
\fill[gray!25] (l)--(c)--(bl)--cycle;
\fill[gray!25] (c)--(r)--(br)--cycle;

% black oriented edges
\draw[oriented] (l) -- (t);
\draw[oriented] (r) -- (t);
\draw[oriented] (c) -- (t);

\draw[oriented] (l) -- (c);
\draw[oriented] (r) -- (c);

\draw[oriented] (l) -- (bl);
\draw[oriented] (r) -- (br);

% colored oriented edges
\draw[colored, #8] (c) -- (bl);
\draw[colored, #9] (c) -- (br);

% simplex labels
\node at (-0.85,0.75) {$\sigma_1$};
\node at ( 0.85,0.75) {$\sigma_2$};
\node at (-1.10,-0.75) {$\sigma_4$};
\node at ( 1.10,-0.75) {$\sigma_3$};

% vertex labels
\node[vlabel, above]       at (t)  {$#3$};
\node[vlabel, left]        at (l)  {$#4$};
\node[vlabel, right]       at (r)  {$#5$};
\node[vlabel, above right] at (c)  {$0$};
\node[vlabel, below]       at (bl) {$#6$};
\node[vlabel, below]       at (br) {$#7$};

\end{scope}
}

% top row
\DiskPanel{0}{0}{0}{0}{0}{0}{3}{red}{yellow!85!orange}
\DiskPanel{7.4}{0}{1}{1}{1}{1}{0}{green!70!black}{red}

% bottom row
\DiskPanel{0}{-5.3}{3}{3}{3}{3}{2}{yellow!85!orange}{cyan!45}
\DiskPanel{7.4}{-5.3}{2}{2}{2}{2}{1}{cyan!45}{green!70!black}

\end{tikzpicture}
\caption{{The space $\zeta_{f}\bigl(\kappa_f(p)\bigr)$ corresponding to the simplicial distribution in Example~\ref{ex:disk}. The same colored edges are identified.}}
\label{fig:four-disks}
\end{figure}

\subsection{Combinatorial sphere}
\label{sec:Combinatorial sphere}

In this section, we consider {empirical models} on the boundary of the standard $n$-simplex. Showing that these scenarios admit contextual simplicial distributions was pivotal in the proof of Vorobev's celebrated theorem~\cite{vorob1962consistent}, which characterizes acyclic simplicial complexes.

\begin{defn}\label{def:standardsimpl}
The \emph{standard $n$-simplex} $\Delta^n$ is the simplicial complex whose vertices are the elements of the set 
$
\{0,1,\dots,n\},
$
and whose simplices are all nonempty subsets of $\{0,1,\dots,n\}$. 
That is,
\[
\Delta^n = \{\, \sigma \subseteq \{0,1,\dots,n\} \mid~ \sigma \neq \emptyset \,\}.
\]
The \emph{boundary of the standard $n$-simplex}, denoted $\partial \Delta^n$, is the subcomplex consisting of all proper faces of $\Delta^n$, that is,
\[
\partial \Delta^n = \Delta^n-\set{\set{0,1,\dots,n}}.
\]    
\end{defn}

The following example is the smallest simplicial complex admitting a
contextual simplicial distribution. It belongs to the class of
distributions known as Popescu--Rohrlich (PR) boxes \cite{pr94}.

%\coc{Prefer to omit arrows on top of letters.}

\begin{ex}\label{ex:3PRbox}
{Consider} 
$$
\partial \Delta^2 = \{\{x\}, \{y\}, \{z\}, \{x,y\}, \{x,z\}, \{y,z\}\}.
$$
We define $p \in \on{Emp}(\eE_{(\partial \Delta^2,\ZZ_2}))$ (see Example~\ref{ex:eventpreshef}) 
by {specifying its values on the maximal simplices of $\partial \Delta^2$:}
\[
p_{\{x,y\}}^{a} = p_{\{x,z\}}^{a} =
\begin{cases}
\frac{1}{2}, & \text{if } a \in \{(0,0),(1,1)\},\\[2mm]
0, & \text{otherwise,}
\end{cases}
\qquad
p_{\{y,z\}}^{a} =
\begin{cases}
\frac{1}{2}, & \text{if } a \in \{(1,0),(0,1)\},\\[2mm]
0, & \text{otherwise.}
\end{cases}
\]
The empirical model $p$ is contextual. 
%This model can be viewed as an analogue of the well-known \emph{PR box} model (see~\cite{pr94,abramsky2011sheaf}),  
%with the underlying measurement scenario given by triangle instead of square.  
%
\end{ex}

We now give a higher-dimensional generalization of this example.

\begin{pro}\label{pro:vertboundrysimpled}    
{Let} {$0,1,\dots,n$} be the vertices of $\partial \Delta^{n}$. We define 
$p \in \on{Emp}(\eE_{(\partial \Delta^{n}, \ZZ_2)})$ as follows: 
$$
p_{\set{0,\dots,j-1,j+1,\dots,n}}^{a}=
\begin{cases}
\frac{1}{n} & \text{if} \;\; a \in \set{(0,\dots,0),(1,1,0,\dots,0),(1,0,1,0,\dots,0),\dots,(1,0,\dots,0,1)},    \\
0 & \text{otherwise,}
\end{cases}
$$
for every $1\leq j \leq n$, and
$$
p_{\set{1,\dots,n}}^{a}=
\begin{cases}
\frac{1}{n} & \text{if} \;\; a \in \set{(1,0,\dots,0),(0,1,0,\dots,0),\dots,(0,\dots,0,1)},   \\
0 & \text{otherwise.} 
\end{cases}
$$
The empirical model $p$ is a vertex. 
\end{pro}

\begin{proof}
%The argument parallels the proof of Proposition~\ref{pro:vetdetal3}. 
{Denote $T=(\partial \Delta^{n}, \ZZ_2)$} and let $F \colon \catC_{\partial \Delta^n}^\op \to \catSet$ be the event scenario $\eta_{\eE_T}(\kappa_{\eE_T}(p))$. 
The first hypothesis of Theorem~\ref{thm:1vert} is clear; namely,
$F$ is strongly connected. For the second hypothesis, let $G \in \Esubb(\eE_T)$ such that $G \leq F$. 
Since $G$ is non-trivial, we have $G(\set{1,\dots,n}) \neq \emptyset$. Assume
$$
(1,0,\dots,0) \in G(\set{1,\dots,n}),
$$ 
then $(0,\dots,0) \in G(\set{2,\dots,n})$. Since $G$ is locally surjective, it follows that 
$$
(0,\dots,0) \in G(\set{0,2,\dots,n}).
$$
Proceeding inductively, we deduce successively that 
$(0,\dots,0)\in G(\set{0,3\dots,{n}})$, so 
\begin{equation}\label{eq:000G1}
(0,\dots,0)\in G(\set{0,1,3\dots,{n}}),
\end{equation}
then we get that 
$(0,\dots,0)\in G(\set{1,3\dots,n})$, as a result
$$
(0,1,0,\dots,0)\in G(\set{1,2,3\dots,{n}}).
$$
Equation (\ref{eq:000G1}) implies also that $(0,\dots,0)\in G(\set{0,1,4\dots,{n}})$, so
\begin{equation}\label{eq:000G2}
(0,\dots,0)\in G(\set{0,1,2,4,\dots,{n}}),
\end{equation}
then we get that 
$(0,\dots,0)\in G(\set{1,2,4,\dots,{n}})$, as a result
$$
(0,0,1,0,\dots,0)\in G(\set{1,2,3,\dots,{n}}).
$$
Equation (\ref{eq:000G2}) implies also that $(0,\dots,0)\in G(\set{0,1,2,5\dots,{n}})$, so
$$
(0,\dots,0)\in G(\set{0,1,2,3,5,\dots,{n}}),
$$
then we obtain that 
$(0,\dots,0)\in G(\set{1,2,3,5,\dots,{n}})$, as a result
$$
(0,0,0,1,0,\dots,0)\in G(\set{1,2,3,4,\dots,{n}}).
$$
Continuing this process, we conclude that $G(\set{1,\dots,{n}})=F(\set{1,\dots,{n}})$. 
Theorem \ref{thm:1vert} implies that $p$ is a vertex.
\end{proof}

Next, we illustrate the result in the case $n=3$.

 \begin{ex}\label{ex:vertdeta3}
Let us write $\set{0,1,2,3}$ for the vertices of $\partial \Delta^3$. We define 
$p \in \on{Emp}(\eE_{(\partial \Delta^3,\ZZ_2)})$ as follows:
$$
p_{\set{0,1,2}}^{a}=p_{\set{0,1,3}}^{a}=p_{\set{0,2,3}}^{a}=
\begin{cases}
\frac{1}{3} & \text{if} \;\; a \in \set{(0,0,0),(1,1,0),(1,0,1)}    \\
0 & \text{otherwise,}
\end{cases}
$$
$$
p_{\set{1,2,3}}^{a}=
\begin{cases}
\frac{1}{3} & \text{if} \;\; a \in \set{(1,0,0),(0,1,0),(0,0,1)}  \\
0 & \text{otherwise.}  
\end{cases}
$$
We prove that the empirical model $p$ is a vertex. 
{Let $T=(\partial \Delta^{3}, \ZZ_2)$ and} let $F \colon \catC_{\partial \Delta^3}^\op \to \catSet$ be the event scenario $\eta_{\eE_T}(\kappa_{\eE_T}(p))$. We have 
\[
\begin{aligned}
F(\set{0,1,2})
&=
F(\set{0,1,3})
=
F(\set{0,2,3})
=
\set{(0,0,0),(1,1,0),(1,0,1)}, \\
F(\set{1,2,3})
&=
\set{(1,0,0),(0,1,0),(0,0,1)}, \\
F(\set{1,2})
&=
F(\set{1,3})
=
F(\set{2,3})
=
\set{(0,0),(0,1),(1,0)}, \\
F(\set{0,1})
&=
F(\set{0,2})
=
F(\set{0,3})
=
\set{(0,0),(1,0),(1,1)}, \\
F(\set{0})
&=
F(\set{1})
=
F(\set{2})
=
F(\set{3})
=
\set{0,1}.
\end{aligned}
\]
The event scenario $F$ is clearly strongly connected. For the second hypothesis of Theorem~\ref{thm:1vert}, let
$G\in \Esubb(\eE_T)$ be such that $G\leq F$. Suppose that
$G(\set{1,2,3})\neq \emptyset$. Without loss of generality, assume that
\[
(1,0,0)\in G(\set{1,2,3}).
\]
Then
\[
(0,0)\in G(\set{2,3}),
\]
which implies, {because of the local surjectivity of $G$}, that
\[
(0,0,0)\in G(\set{0,2,3}).
\]
Similarly, we obtain
\[
(0,0,0)\in G(\set{0,1,2})
\qquad \text{and} \qquad
(0,0,0)\in G(\set{0,1,3}).
\]
Therefore,
\[
(0,0)\in G(\set{1,2})
\qquad \text{and} \qquad
(0,0)\in G(\set{1,3}).
\]
It follows, {again because of the local surjectivity of $G$,} that both $(0,0,1)$ and $(0,1,0)$ belong to
$G(\set{1,2,3})$. Hence
\[
G(\set{1,2,3})=F(\set{1,2,3}).
\]
By Theorem~\ref{thm:1vert}, we conclude that $p$ is a vertex.
\end{ex}
%{===}

\subsection{Octohedral sphere}

The canonical examples of scenarios arising in quantum foundations, known as \emph{Bell scenarios} {\cite{brunner2014bell}}, are defined on spheres with specific triangulations. {A Bell scenario consists of $n$ parties, each of whom performs a measurement with $m$ outcomes. We now introduce the underlying simplicial complex, or simplicial set, for this scenario.}

\begin{defn}\label{def:B(n,m)} 
Let $n,m\geq 1$, and let $A_1,\dots,A_n$ be pairwise disjoint sets
such that $|A_i|=m$ for every $1\leq i\leq n$. We define $B(n,m)$ to be
the simplicial complex with vertex set
\[
\bigsqcup_{i=1}^n A_i
\]
whose simplices are the subsets containing at most one vertex from each
$A_i$. Equivalently,
\[
B(n,m)
=
\Bigl\{
\set{x_1,\dots,x_k} \mid~
1\leq i_1<\cdots<i_k\leq n
\text{ and }
x_j\in A_{i_j}\text{ for }1\leq j\leq k
\Bigr\}.
\]
We write ${s}B(n,m)$ for the simplicial set obtained as the singular realization of this simplicial complex, {where the chosen ordering of the vertices satisfies
$$
x<y
\qquad
\text{for every } x\in A_i,\; y\in A_j,\; \text{whenever } i<j;
$$
see Definition \ref{def:singular realization}.}  In the case {$n=3$ and} $m=2$, this is the usual octahedral
sphere.
\end{defn}

The canonical example of a PR box arises from the following {Bell} scenario, which can also be viewed as a $4$-circle (Example \ref{ex:circle}).

\begin{ex}\label{ex:4PRbox}
The simplicial complex $B(2,2)$ 
%(Definition \ref{def:B(n,m)}) 
has maximal simplices given by
$$
\set{x,y}, \;\; \set{x,y'}, \;\; \set{x',y}, \;\; \set{x',y'}.
$$
Consider $T=(B(2,2),\ZZ_2)$.
% (see Example \ref{ex:eventpreshef}).
We define $p \in \on{Emp}(\eE_{T})$ by setting
\begin{align*}
p_{\set{x,y'}}^{a}
&=p_{\set{x',y}}^{a}=p_{\set{x',y'}}^{a}=
\begin{cases}
\frac{1}{2} & \text{if} \;\; a \in \set{(0,0),(1,1)}    \\
0 & \text{otherwise,}
\end{cases}\\
p_{\set{x,y}}^{a}
&=
\begin{cases}
\frac{1}{2} & \text{if} \;\; a \in \set{(1,0),(0,1)}  \\
0 & \text{otherwise.}  
\end{cases}
\end{align*}
The empirical model $p$ is a contextual vertex of $\on{Emp}(\eE_{T})$. Moreover, it can be viewed as a $2$-order cycle distribution on a $4$-circle (Example~\ref{ex:circle}).
\end{ex}

\begin{pro}\label{pro:vertex(n,2,2)}
Consider the simplicial complex $\Sigma=B(n,2)$ %(Definition~\ref{def:B(n,m)}) be the simplicial complex whose
with maximal simplices
% are the sets consisting of one element from each of the following pairs:
$$
\{x_1,x'_1\}, \{x_2,x'_2\},\dots, \{x_n,x'_n\}.
$$
Let $T=(B(n,2),\ZZ_2)$ and define a distribution
$
p\in \on{Emp}(\eE_T)
$
by
\begin{equation}\label{eq:basisvertexAiceBobbbbbbbbbb}
p_{\{x_1,\dots,x_n\}}^{(a_1,\dots,a_n)}
=
\begin{cases}
\displaystyle \frac{1}{2^{\,n-1}} & \text{if } \sum_{i=1}^n a_i = 1, \\[6pt]
0 & \text{otherwise,}
\end{cases}
\qquad
p_A^{(a_1,\dots,a_n)}
=
\begin{cases}
\displaystyle \frac{1}{2^{\,n-1}} & \text{if } \sum_{i=1}^n a_i = 0, \\[6pt]
0 & \text{otherwise,}
\end{cases}
\end{equation}
for every maximal simplex $A\neq \{x_1,\dots,x_n\}$.
The empirical model $p$ is a vertex. 
\end{pro}

\begin{proof}
Let $F \colon \catC^{\op}_{\Sigma} \to \catSet$ be the event scenario $\eta_{\eE_T}(\kappa_{\eE_{T}}(p))$. The first hypothesis of Theorem \ref{thm:1vert} is clear. For the second hypothesis, let $G \in \Esubb(\eE_T)$ such that $G \leq F$. Since $G$ is non-trivial, there exists $(i_1,i_2,\dots,i_n) \in G(\set{x_1,x_2,\dots,x_n})$. 
Take 
$(i_1,i_2,\dots,i_n)\neq (j_1,j_2,\dots,j_n) \in F(\set{x_1,x_2,\dots,x_n})$. Since 
$$
i_1+i_2+\dots+i_n=1   \;\; \text{and} \;\;  j_1+j_2+\dots+j_n=1, 
$$
we have $$
(i_1-j_1)+(i_2-j_2)+\dots+(i_n-j_n)=0.
$$
Thus, the number of indices $k$ for which $i_k \neq j_k$ is even.
Let $k<s$ be two such indices with $i_k \neq j_k$ and $i_s \neq j_s$. Because $(i_1,i_2,\dots,i_n) \in G(\set{x_1,x_2,\dots,x_n})$, we have  
$$
(i_1,\dots,i_{k-1},i_{k+1},\dots,i_n) \in G(\set{x_1,\dots,x_{k-1},x_{k+1},\dots,x_n}).
$$ 
Since $G$ is locally surjective, $(i_1,\dots,i_{k-1},i_{k},i_{k+1},\dots,i_n) \notin G(\set{x_1,\dots,x_{k-1},x'_k,x_{k+1},\dots,x_n})$, and $i_k \neq j_k$, it follows that  
$$
(i_1,\dots,i_{k-1},j_{k},i_{k+1},\dots,i_n) \in G(\set{x_1,\dots,x_{k-1},x'_k,x_{k+1},\dots,x_n}).
$$
Consequently,
$$
(i_1,\dots,i_{k-1},j_{k},i_{k+1},\dots,i_{s-1},i_{s+1},\dots,i_n) \in G(\set{x_1,\dots,x_{k-1},x'_k,x_{k+1},\dots,x_{s-1},x_{s+1},\dots,x_n}).
$$
By the same argument for the index $s$, we get
$$
(i_1,\dots,i_{k-1},j_{k},i_{k+1},\dots,i_{s-1},i_s,i_{s+1},\dots,i_n) \in G(\set{x_1,\dots,x_{k-1},x'_k,x_{k+1},\dots,x_{s-1},x'_s,x_{s+1},\dots,x_n})
$$
and similarly,
$$
(i_1,\dots,i_{k-1},j_{k},i_{k+1},\dots,i_{s-1},i_s,i_{s+1},\dots,i_n) \in G(\set{x_1,\dots,x_{k-1},x_k,x_{k+1},\dots,x_{s-1},x'_s,x_{s+1},\dots,x_n})
$$
and finally, 
$$
(i_1,\dots,i_{k-1},j_{k},i_{k+1},\dots,i_{s-1},j_s,i_{s+1},\dots,i_n) \in G(\set{x_1,\dots,x_{k-1},x_k,x_{k+1},\dots,x_{s-1},x_s,x_{s+1},\dots,x_n})
$$
Since the number of indices $t$ with $i_t \neq j_t$ is even, repeating this process for all such indices yields
$$
(j_1,j_2,\dots,j_n) \in G(\set{x_1,x_2,\dots,x_n}).
$$
We proved that $G(\set{x_1,x_2,\dots,x_n})=F(\set{x_1,x_2,\dots,x_n})$. Theorem \ref{thm:1vert} implies that $p$ is a vertex.
\end{proof}
%
%{We now illustrate Proposition \ref{pro:vertex(n,2,2)} with the case of three parties.}

We illustrate the result in the case $n=3$. The distribution we will consider appears in 
\cite[Equation (29)]{barrett2005nonlocal} under the class of \emph{three-way nonlocal vertices}.

\begin{figure}[h]
\centering
\begin{tikzpicture}[
    scale=1.05,
    oriented/.style={
        line width=0.9pt,
        postaction={decorate},
        decoration={
            markings,
            mark=at position 0.58 with {\arrow{Stealth[length=2mm]}}
        }
    },
    edgecol/.style={
        line width=1.6pt,
        postaction={decorate},
        decoration={
            markings,
            mark=at position 0.58 with {\arrow{Stealth[length=2mm]}}
        }
    }
]

% vertices
\coordinate (A) at (-3.0,0);
\coordinate (P) at (-1.35,0);
\coordinate (Q) at (0,0);
\coordinate (S) at (1.35,0);
\coordinate (R) at (3.0,0);
\coordinate (T) at (0,2.45);
\coordinate (B) at (0,-2.45);

% gray regions
\fill[gray!25] (A)--(P)--(T)--cycle;
\fill[gray!25] (P)--(Q)--(T)--cycle;
\fill[gray!25] (Q)--(S)--(T)--cycle;
\fill[gray!25] (S)--(R)--(T)--cycle;

\fill[gray!25] (A)--(P)--(B)--cycle;
\fill[gray!25] (P)--(Q)--(B)--cycle;
\fill[gray!25] (Q)--(S)--(B)--cycle;
\fill[gray!25] (S)--(R)--(B)--cycle;

% colored outer edges
\draw[edgecol,red] (A) -- (T);
\draw[edgecol,red] (R) -- (T);

\draw[edgecol,olive] (A) -- (B);
\draw[edgecol,olive] (R) -- (B);

% black horizontal edges
\draw[oriented] (A) -- (P);
\draw[oriented] (Q) -- (P);
\draw[oriented] (Q) -- (S);
\draw[oriented] (R) -- (S);

% black upper internal edges
\draw[oriented] (P) -- (T);
\draw[oriented] (Q) -- (T);
\draw[oriented] (S) -- (T);

% black lower internal edges
\draw[oriented] (P) -- (B);
\draw[oriented] (Q) -- (B);
\draw[oriented] (S) -- (B);

% vertex labels
\node[left]  at (A) {$x$};
\node[below left=2pt and 1pt] at (P) {$y$};
\node[above right=2pt and 1pt] at (Q) {$x'$};
\node[below right=1pt and 1pt] at (S) {$y'$};
\node[right] at (R) {$x$};

\node[above] at (T) {$z$};
\node[below] at (B) {$z'$};

\end{tikzpicture}
\caption{The simplicial complex $B(3,2)$. The same colored edges are identified.}
\label{fig:base1}
\end{figure}

\begin{ex}\label{ex:AliceBobChar}
Consider
%Let 
$\Sigma=B(3,2)$ 
%be the simplicial complex with the following maximal simplices:
with maximal simplices
$$
\set{x,y,z},\; \set{x,y,z'}, \; \set{x,y',z},\; \set{x,y',z'},\;\set{x',y,z},\; \set{x',y,z'}, \; \set{x',y',z},\; \set{x',y',z'}.
$$    
For $T=(B(3,2),\ZZ_2)$, we define 
%an empirical model 
$p \in \on{Emp}(\eE_{T})$ as follows:
\begin{equation}\label{eq:basisvertexAiceBob}
p_{\set{x,y,z}}^{(a,b,c)}=
\begin{cases}
\frac{1}{4} & \text{if} \;\; a+b+c=1    \\
0 & \text{otherwise,}
\end{cases}
\;\;\;\;\;,\;\;\;\; \;
p_{A}^{(a,b,c)}=
\begin{cases}
\frac{1}{4} & \text{if} \;\; a+b+c=0    \\
0 & \text{otherwise,}
\end{cases}
\end{equation}
for every maximal simplex $A \neq \set{x,y,z}$.
We prove that this distribution is a vertex using Theorem \ref{thm:1vert}.
Let us write $F:=\eta_{\eE_T}(\kappa_{\eE_T}(p))$ for the event scenario associated to the possibilistic collapse of $p$.
%$F \colon \catC^{\op}_{\Sigma} \to \catSet$ be the event scenario $\eta_{\eE_T}(\kappa_{\eE_T}(p))$. 
Then $F$ is strongly connected, hence, satisfies the first condition of Theorem \ref{thm:1vert}.
% is clear. 
For the second condition, let $G \in \Esubb(\eE_{T})$ such that $G \leq F$. Since $G$ is non-trivial, assume that $(1,0,0) \in G(\{x,y,z\})$.  
Then $(0,0) \in G(\{y,z\})$. Because $G$ is locally surjective, we obtain
$$
(0,0,0) \in G(\{x',y,z\}).
$$
By the same argument, we get
$$
(0,0,0) \in G(\{x',y',z\}), \;
(0,0,0) \in G(\{x,y',z\}), \; \text{and} \;(0,1,0) \in G(\{x,y,z\})
$$  
Consequently,
$$
(0,1,1) \in G(\{x,y,z'\}), \; 
(0,1,1) \in G(\{x',y,z'\}), \; \text{and} \;
(0,1,1) \in G(\{x',y,z\}),
$$
and hence
$$
(1,1,1) \in G(\{x,y,z\}).
$$
On the other hand, we also have
$$
(0,1,1) \in G(\{x',y',z\}), \;
(0,1,1) \in G(\{x,y',z\}), \; \text{and} \;
(0,0,1) \in G(\{x,y,z\}).
$$
Altogether, this shows that
$$
G(\{x,y,z\}) = F(\{x,y,z\}).
$$
Note that for every maximal simplex $A$ and every $x \in A$, we have
$|F(A)| = |F(A \setminus \{x\})| = 4$, which implies that the map
$$
F(A) \longrightarrow F(A \setminus \{x\})
$$
is an isomorphism.  
Therefore, by Theorem \ref{thm:1vert} $p$ is a 
vertex. We could also conclude that $p$ is a vertex by {by viewing it as a simplicial distribution on $f_{sB(3,2),2}$ and applying the topological criterion}, namely Theorem~\ref{thm:strongisvertx}. See Figure~\ref{fig:base1-bundle} to observe that  $f|_{\zeta_{f}(\kappa_f(p))}$ is strongly connected.
%$\zeta_{f}(\kappa_f(p))$ is strongly connected 
%by $f|_{\zeta_{f}(\kappa_f(p))}$. 
\end{ex}

\begin{figure}[h]
\centering
\begin{tikzpicture}[
    scale=0.85,
    oriented/.style={
        line width=0.85pt,
        postaction={decorate},
        decoration={
            markings,
            mark=at position 0.58 with {\arrow{Stealth[length=2mm]}}
        }
    },
    edgecol/.style={
        line width=1.5pt,
        postaction={decorate},
        decoration={
            markings,
            mark=at position 0.58 with {\arrow{Stealth[length=2mm]}}
        }
    },
    vlabel/.style={font=\small, inner sep=1pt}
]

\newcommand{\DrawPanel}{%

% vertices
\coordinate (A) at (-3.0,0);
\coordinate (P) at (-1.55,0);
\coordinate (Q) at (0,0);
\coordinate (S) at (1.55,0);
\coordinate (R) at (3.0,0);
\coordinate (T) at (0,2.55);
\coordinate (D) at (-1.55,-2.15);
\coordinate (B) at (0,-2.15);

% gray regions
\fill[gray!25] (A)--(P)--(T)--cycle;
\fill[gray!25] (P)--(Q)--(T)--cycle;
\fill[gray!25] (Q)--(S)--(T)--cycle;
\fill[gray!25] (S)--(R)--(T)--cycle;
\fill[gray!25] (A)--(P)--(D)--cycle;
\fill[gray!25] (P)--(Q)--(B)--cycle;
\fill[gray!25] (Q)--(S)--(B)--cycle;
\fill[gray!25] (S)--(R)--(B)--cycle;

% colored edges
\draw[edgecol,\colAT] (A) -- (T);
\draw[edgecol,\colTR] (R) -- (T);
\draw[edgecol,\colRB] (R) -- (B);
\draw[edgecol,\colAD] (A) -- (D);
\draw[edgecol,\colPD] (P) -- (D);
\draw[edgecol,\colPB] (P) -- (B);

% black oriented edges
\draw[oriented] (A) -- (P);
\draw[oriented] (Q) -- (P);
\draw[oriented] (Q) -- (S);
\draw[oriented] (R) -- (S);

\draw[oriented] (P) -- (T);
\draw[oriented] (Q) -- (T);
\draw[oriented] (S) -- (T);

\draw[oriented] (Q) -- (B);
\draw[oriented] (S) -- (B);

% vertex labels
\node[vlabel, above]       at (T) {$\labT$};
\node[vlabel, left]        at (A) {$\labA$};
\node[vlabel, above left]  at (P) {$\labP$};
\node[vlabel, above right, xshift=1pt, yshift=1pt] at (Q) {$\labQ$};
\node[vlabel, above right] at (S) {$\labS$};
\node[vlabel, right]       at (R) {$\labR$};
\node[vlabel, below]       at (D) {$\labD$};
\node[vlabel, below]       at (B) {$\labB$};
}

% ------------------------------------------------------------
% top-left panel
% ------------------------------------------------------------
\begin{scope}[shift={(0,0)}]
\def\labT{1}\def\labA{0}\def\labP{0}\def\labQ{1}
\def\labS{0}\def\labR{1}\def\labD{0}\def\labB{1}
\def\colAT{green!70!black}
\def\colTR{yellow!80!green}
\def\colRB{blue!70}
\def\colAD{pink!70}
\def\colPD{violet!70}
\def\colPB{gray!80!blue}
\DrawPanel
\end{scope}

% ------------------------------------------------------------
% top-right panel
% ------------------------------------------------------------
\begin{scope}[shift={(8.4,0)}]
\def\labT{1}\def\labA{1}\def\labP{1}\def\labQ{0}
\def\labS{1}\def\labR{0}\def\labD{0}\def\labB{1}
\def\colAT{yellow!80!green}
\def\colTR{green!70!black}
\def\colRB{orange!85!yellow}
\def\colAD{olive}
\def\colPD{cyan!45}
\def\colPB{orange!30!pink}
\DrawPanel
\end{scope}

% ------------------------------------------------------------
% bottom-left panel
% ------------------------------------------------------------
\begin{scope}[shift={(0,-5.7)}]
\def\labT{0}\def\labA{1}\def\labP{0}\def\labQ{0}
\def\labS{0}\def\labR{0}\def\labD{1}\def\labB{0}
\def\colAT{red}
\def\colTR{brown!75!red}
\def\colRB{pink!65}
\def\colAD{blue!70}
\def\colPD{gray!80!blue}
\def\colPB{violet!70}
\DrawPanel
\end{scope}

% ------------------------------------------------------------
% bottom-right panel
% ------------------------------------------------------------
\begin{scope}[shift={(8.4,-5.7)}]
\def\labT{0}\def\labA{0}\def\labP{1}\def\labQ{1}
\def\labS{1}\def\labR{1}\def\labD{1}\def\labB{0}
\def\colAT{brown!75!red}
\def\colTR{red}
\def\colRB{olive}
\def\colAD{orange!85!yellow}
\def\colPD{orange!30!pink}
\def\colPB{cyan!45}
\DrawPanel
\end{scope}

\end{tikzpicture}
\caption{{The space $\zeta_f(\kappa_f(p))$ corresponding to $p$ of} Example \ref{ex:AliceBobChar}.}
\label{fig:base1-bundle}
\end{figure}
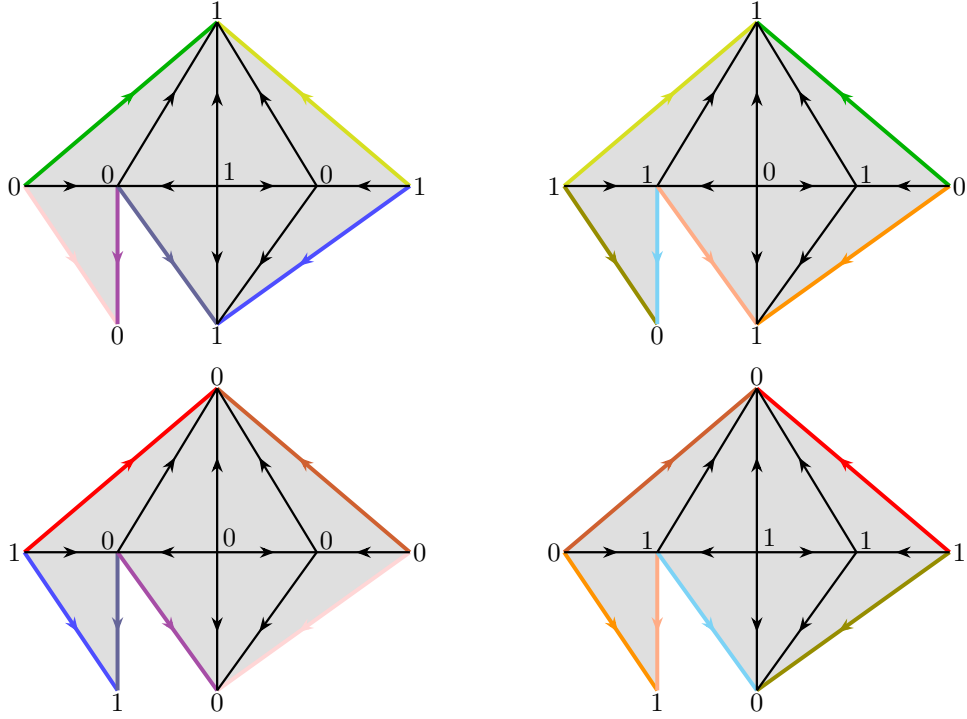

 {
We conclude this section by presenting the other two three-way nonlocal vertices from \cite{barrett2005nonlocal} as simplicial distributions on $f_{sB(3,2),2}$ and showing that they satisfy the condition of Theorem~\ref{thm:strongisvertx}.

Suppose that 
$$
(x,y,z),\; (x,y,z'), \; (x,y',z),\; (x,y',z'),\;(x',y,z),\; (x',y,z'), \; (x',y',z),\; (x',y',z')    
$$
are the generators of the simplicial set $sB(3,2)$.

Let $\Omega_1=\set{(x,y,z'),(x,y',z)}$. We define  
$p \in \sDist(f_{sB(3,2),2})$ by 
\begin{equation}\label{eq:vertex27}
p_{\sigma}^{(a,b,c)}=
\begin{cases}
\frac{1}{4} & \text{if} \;\;  a+b+c=1    \\
0 & \text{otherwise,}
\end{cases}
\;\;\;\;\;,\;\;\;\; \;
p_{\tau}^{(a,b,c)}=
\begin{cases}
\frac{1}{4} & \text{if} \;\; a+b+c=0    \\
0 & \text{otherwise,}
\end{cases}
\end{equation}
for generators $\sigma\in \Omega_1$ and  $\tau \notin \Omega_1$.
This simplicial distribution represents the vertex given in \cite[Equation~(27)]{barrett2005nonlocal}. Figure \ref{fig:eq27} illustrates the geometric realization of its support.

Next, let $\Omega_2=\set{(x,y',z'),(x',y',z),(x',y,z'),(x',y',z')}$ and define a simplicial distribution
%an empirical model 
$q \in \sDist(f_{sB(3,2),2})$ by 
\begin{equation}\label{eq:vertex28}
q_{\sigma}^{(a,b,c)}=
\begin{cases}
\frac{1}{4} & \text{if} \;\;  a+b+c=1    \\
0 & \text{otherwise,}
\end{cases}
\;\;\;\;\;,\;\;\;\; \;
q_{\tau}^{(a,b,c)}=
\begin{cases}
\frac{1}{4} & \text{if} \;\; a+b+c=0    \\
0 & \text{otherwise,}
\end{cases}
\end{equation}
for generators $\sigma\in \Omega_2$ and  $\tau \notin \Omega_2$.
This represents the vertex given in \cite[Equation~(28)]{barrett2005nonlocal} and Figure \ref{fig:eq28} illustrates the geometrical realization of its support.
}

 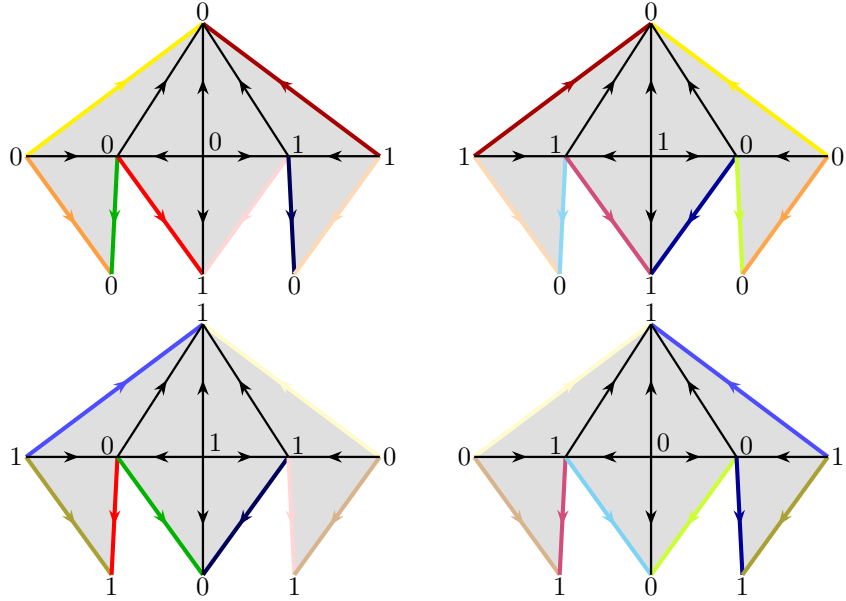
\begin{figure}[h]
\centering
\begin{tikzpicture}[
    scale=0.78,
    oriented/.style={
        line width=0.85pt,
        postaction={decorate},
        decoration={
            markings,
            mark=at position 0.58 with {\arrow{Stealth[length=2mm]}}
        }
    },
    edgecol/.style={
        line width=1.55pt,
        postaction={decorate},
        decoration={
            markings,
            mark=at position 0.58 with {\arrow{Stealth[length=2mm]}}
        }
    },
    vlabel/.style={font=\small, inner sep=1pt}
]

\newcommand{\DrawFanPanel}{%

% vertices
\coordinate (A) at (-3.0,0);
\coordinate (P) at (-1.45,0);
\coordinate (Q) at (0,0);
\coordinate (S) at (1.45,0);
\coordinate (R) at (3.0,0);
\coordinate (T) at (0,2.25);
\coordinate (D) at (-1.55,-2.0);
\coordinate (B) at (0,-2.0);
\coordinate (E) at (1.55,-2.0);

% gray regions
\fill[gray!25] (A)--(P)--(T)--cycle;
\fill[gray!25] (P)--(Q)--(T)--cycle;
\fill[gray!25] (Q)--(S)--(T)--cycle;
\fill[gray!25] (S)--(R)--(T)--cycle;

\fill[gray!25] (A)--(P)--(D)--cycle;
\fill[gray!25] (P)--(Q)--(B)--cycle;
\fill[gray!25] (Q)--(S)--(B)--cycle;
\fill[gray!25] (S)--(R)--(E)--cycle;

% colored edges
\draw[edgecol,\colAT] (A) -- (T);
\draw[edgecol,\colTR] (R) -- (T);

\draw[edgecol,\colAD] (A) -- (D);
\draw[edgecol,\colPD] (P) -- (D);
\draw[edgecol,\colPB] (P) -- (B);

\draw[edgecol,\colSB] (S) -- (B);
\draw[edgecol,\colSE] (S) -- (E);
\draw[edgecol,\colRE] (R) -- (E);

% black oriented edges
\draw[oriented] (A) -- (P);
\draw[oriented] (Q) -- (P);
\draw[oriented] (Q) -- (S);
\draw[oriented] (R) -- (S);

\draw[oriented] (P) -- (T);
\draw[oriented] (Q) -- (T);
\draw[oriented] (S) -- (T);

\draw[oriented] (Q) -- (B);

% vertex labels
\node[vlabel, above] at (T) {$\labT$};
\node[vlabel, left] at (A) {$\labA$};
\node[vlabel, above left] at (P) {$\labP$};
\node[vlabel, above right, xshift=1pt, yshift=1pt] at (Q) {$\labQ$};
\node[vlabel, above right] at (S) {$\labS$};
\node[vlabel, right] at (R) {$\labR$};
\node[vlabel, below] at (D) {$\labD$};
\node[vlabel, below] at (B) {$\labB$};
\node[vlabel, below] at (E) {$\labE$};
}

% ------------------------------------------------------------
% top-left panel
% ------------------------------------------------------------
\begin{scope}[shift={(0,0)}]
\def\labT{0}\def\labA{0}\def\labP{0}\def\labQ{0}
\def\labS{1}\def\labR{1}\def\labD{0}\def\labB{1}\def\labE{0}
\def\colAT{yellow}
\def\colTR{red!65!black}
\def\colAD{orange!75}
\def\colPD{green!70!black}
\def\colPB{red}
\def\colSB{pink!65}
\def\colSE{blue!35!black}
\def\colRE{orange!30}
\DrawFanPanel
\end{scope}

% ------------------------------------------------------------
% top-right panel
% ------------------------------------------------------------
\begin{scope}[shift={(7.6,0)}]
\def\labT{0}\def\labA{1}\def\labP{1}\def\labQ{1}
\def\labS{0}\def\labR{0}\def\labD{0}\def\labB{1}\def\labE{0}
\def\colAT{red!65!black}
\def\colTR{yellow}
\def\colAD{orange!30}
\def\colPD{cyan!40}
\def\colPB{purple!70}
\def\colSB{blue!60!black}
\def\colSE{lime!80}
\def\colRE{orange!70}
\DrawFanPanel
\end{scope}

% ------------------------------------------------------------
% bottom-left panel
% ------------------------------------------------------------
\begin{scope}[shift={(0,-5.1)}]
\def\labT{1}\def\labA{1}\def\labP{0}\def\labQ{1}
\def\labS{1}\def\labR{0}\def\labD{1}\def\labB{0}\def\labE{1}
\def\colAT{blue!70}
\def\colTR{yellow!25}
\def\colAD{olive!80}
\def\colPD{red}
\def\colPB{green!70!black}
\def\colSB{blue!35!black}
\def\colSE{pink!65}
\def\colRE{brown!60}
\DrawFanPanel
\end{scope}

% ------------------------------------------------------------
% bottom-right panel
% ------------------------------------------------------------
\begin{scope}[shift={(7.6,-5.1)}]
\def\labT{1}\def\labA{0}\def\labP{1}\def\labQ{0}
\def\labS{0}\def\labR{1}\def\labD{1}\def\labB{0}\def\labE{1}
\def\colAT{yellow!25}
\def\colTR{blue!70}
\def\colAD{brown!60}
\def\colPD{purple!70}
\def\colPB{cyan!45}
\def\colSB{lime!80}
\def\colSE{blue!60!black}
\def\colRE{olive!80}
\DrawFanPanel
\end{scope}

\end{tikzpicture}
\caption{{The space $\zeta_{f}\bigl(\kappa_f(p)\bigr)$ corresponding to the vertex defined in Equation (\ref{eq:vertex27}).}}
\label{fig:eq27}
\end{figure}

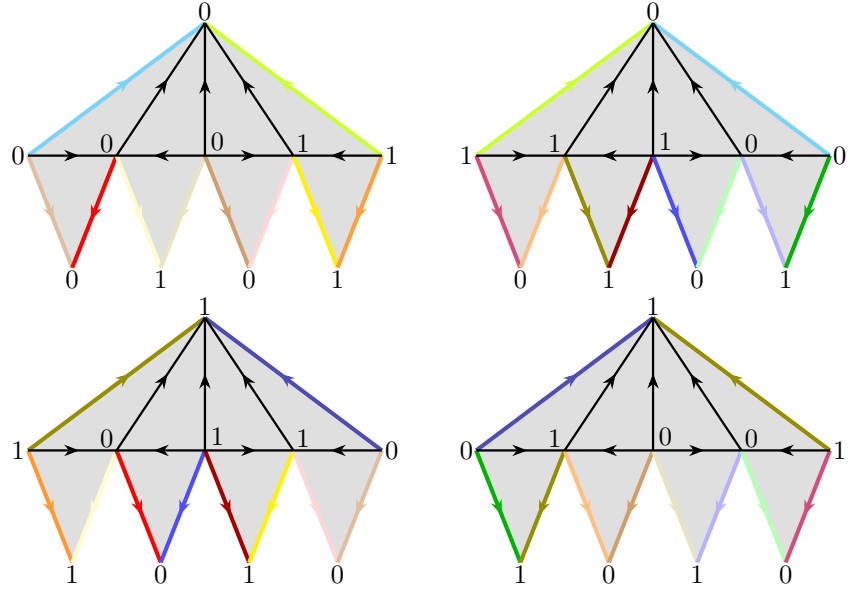
\begin{figure}[h]
\centering
\begin{tikzpicture}[
    scale=0.78,
    oriented/.style={
        line width=0.85pt,
        postaction={decorate},
        decoration={
            markings,
            mark=at position 0.58 with {\arrow{Stealth[length=2mm]}}
        }
    },
    edgecol/.style={
        line width=1.55pt,
        postaction={decorate},
        decoration={
            markings,
            mark=at position 0.58 with {\arrow{Stealth[length=2mm]}}
        }
    },
    vlabel/.style={font=\small, inner sep=1pt}
]

\newcommand{\DrawFourDownPanel}{%

% vertices
\coordinate (A) at (-3.0,0);
\coordinate (P) at (-1.5,0);
\coordinate (Q) at (0,0);
\coordinate (S) at (1.5,0);
\coordinate (R) at (3.0,0);
\coordinate (T) at (0,2.25);

\coordinate (D) at (-2.25,-1.9);
\coordinate (E) at (-0.75,-1.9);
\coordinate (F) at (0.75,-1.9);
\coordinate (G) at (2.25,-1.9);

% gray regions
\fill[gray!25] (A)--(P)--(T)--cycle;
\fill[gray!25] (P)--(Q)--(T)--cycle;
\fill[gray!25] (Q)--(S)--(T)--cycle;
\fill[gray!25] (S)--(R)--(T)--cycle;

\fill[gray!25] (A)--(P)--(D)--cycle;
\fill[gray!25] (P)--(Q)--(E)--cycle;
\fill[gray!25] (Q)--(S)--(F)--cycle;
\fill[gray!25] (S)--(R)--(G)--cycle;

% colored edges
\draw[edgecol,\colAT] (A) -- (T);
\draw[edgecol,\colTR] (R) -- (T);

\draw[edgecol,\colAD] (A) -- (D);
\draw[edgecol,\colPD] (P) -- (D);

\draw[edgecol,\colPE] (P) -- (E);
\draw[edgecol,\colQE] (Q) -- (E);

\draw[edgecol,\colQF] (Q) -- (F);
\draw[edgecol,\colSF] (S) -- (F);

\draw[edgecol,\colSG] (S) -- (G);
\draw[edgecol,\colRG] (R) -- (G);

% black oriented edges
\draw[oriented] (A) -- (P);
\draw[oriented] (Q) -- (P);
\draw[oriented] (Q) -- (S);
\draw[oriented] (R) -- (S);

\draw[oriented] (P) -- (T);
\draw[oriented] (Q) -- (T);
\draw[oriented] (S) -- (T);

% vertex labels
\node[vlabel, above] at (T) {$\labT$};
\node[vlabel, left] at (A) {$\labA$};
\node[vlabel, above left] at (P) {$\labP$};
\node[vlabel, above right, xshift=1pt, yshift=1pt] at (Q) {$\labQ$};
\node[vlabel, above right] at (S) {$\labS$};
\node[vlabel, right] at (R) {$\labR$};

\node[vlabel, below] at (D) {$\labD$};
\node[vlabel, below] at (E) {$\labE$};
\node[vlabel, below] at (F) {$\labF$};
\node[vlabel, below] at (G) {$\labG$};
}

% ------------------------------------------------------------
% top-left panel
% ------------------------------------------------------------
\begin{scope}[shift={(0,0)}]
\def\labT{0}\def\labA{0}\def\labP{0}\def\labQ{0}
\def\labS{1}\def\labR{1}
\def\labD{0}\def\labE{1}\def\labF{0}\def\labG{1}

\def\colAT{cyan!45}
\def\colTR{lime!80}
\def\colAD{brown!50}
\def\colPD{red}
\def\colPE{yellow!25}
\def\colQE{olive!20}
\def\colQF{brown!75}
\def\colSF{pink!65}
\def\colSG{yellow}
\def\colRG{orange!75}

\DrawFourDownPanel
\end{scope}

% ------------------------------------------------------------
% top-right panel
% ------------------------------------------------------------
\begin{scope}[shift={(7.6,0)}]
\def\labT{0}\def\labA{1}\def\labP{1}\def\labQ{1}
\def\labS{0}\def\labR{0}
\def\labD{0}\def\labE{1}\def\labF{0}\def\labG{1}

\def\colAT{lime!80}
\def\colTR{cyan!45}
\def\colAD{purple!70}
\def\colPD{orange!50}
\def\colPE{olive}
\def\colQE{red!60!black}
\def\colQF{blue!70}
\def\colSF{green!30}
\def\colSG{blue!30}
\def\colRG{green!70!black}

\DrawFourDownPanel
\end{scope}

% ------------------------------------------------------------
% bottom-left panel
% ------------------------------------------------------------
\begin{scope}[shift={(0,-5.0)}]
\def\labT{1}\def\labA{1}\def\labP{0}\def\labQ{1}
\def\labS{1}\def\labR{0}
\def\labD{1}\def\labE{0}\def\labF{1}\def\labG{0}

\def\colAT{olive}
\def\colTR{gray!60!blue}
\def\colAD{orange!80}
\def\colPD{yellow!25}
\def\colPE{red}
\def\colQE{blue!70}
\def\colQF{red!65!black}
\def\colSF{yellow}
\def\colSG{pink!65}
\def\colRG{brown!50}

\DrawFourDownPanel
\end{scope}

% ------------------------------------------------------------
% bottom-right panel
% ------------------------------------------------------------
\begin{scope}[shift={(7.6,-5.0)}]
\def\labT{1}\def\labA{0}\def\labP{1}\def\labQ{0}
\def\labS{0}\def\labR{1}
\def\labD{1}\def\labE{0}\def\labF{1}\def\labG{0}

\def\colAT{gray!60!blue}
\def\colTR{olive}
\def\colAD{green!70!black}
\def\colPD{olive}
\def\colPE{orange!50}
\def\colQE{brown!75}
\def\colQF{olive!20}
\def\colSF{blue!30}
\def\colSG{green!30}
\def\colRG{purple!70}

\DrawFourDownPanel
\end{scope}

\end{tikzpicture}
\caption{{The space $\zeta_{f}\bigl(\kappa_f(q)\bigr)$ corresponding to the vertex defined in Equation (\ref{eq:vertex28}).}}
\label{fig:eq28}
\end{figure}

\subsection{Comparing {extremality conditions}}\label{subsect:compcharac}

We conclude the examples section with an extremal distribution that is detected by the categorical {criterion} but not by the topological {criterion}. In addition, we show that the categorical condition is not necessary for extremality.

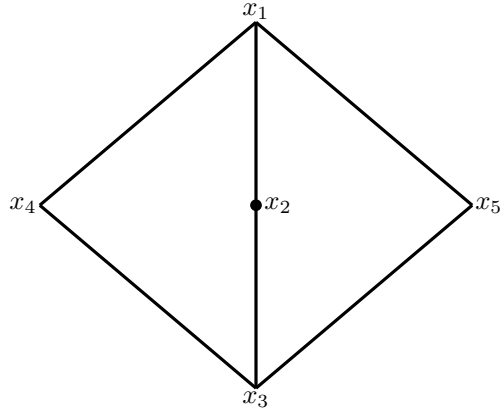
\begin{figure}[h]
\centering
\begin{tikzpicture}[
    scale=1.1,
    vertexlabel/.style={font=\small, inner sep=1pt}
]

% vertices
\coordinate (x1) at (0,2.2);
\coordinate (x3) at (0,-2.2);
\coordinate (x4) at (-2.6,0);
\coordinate (x5) at (2.6,0);
% vertex at x_2
\fill (0,0) circle (2pt);

% edges
\draw[line width=1.2pt] (x4) -- (x1);
\draw[line width=1.2pt] (x1) -- (x5);
\draw[line width=1.2pt] (x5) -- (x3);
\draw[line width=1.2pt] (x3) -- (x4);
\draw[line width=1.2pt] (x1) -- (x3);

% labels
\node[vertexlabel, above] at (x1) {$x_1$};
\node[vertexlabel, below] at (x3) {$x_3$};
\node[vertexlabel, left]  at (x4) {$x_4$};
\node[vertexlabel, right] at (x5) {$x_5$};

% middle label on vertical edge
\node[vertexlabel, right, xshift=2pt] at (0,0) {$x_2$};

\end{tikzpicture}
\caption{The graph $\Sigma$ of Example \ref{ex:graphSigma}.}
\label{fig:graph1}
\end{figure}
%\begin{figure}[h!] 
%  \centering
%  \includegraphics[width=0.34\linewidth]{graph1.png}
%\caption{The graph $\Sigma$ of Example \ref{ex:graphSigma}}
%\label{fig:graph1}
%\end{figure}

\begin{ex}\label{ex:graphSigma}
Let $\Sigma$ be the graph with edges
$$
\{x_1,x_2\}, \{x_2,x_3\}, \{x_1,x_4\}, \{x_1,x_5\}, \{x_3,x_4\}, \{x_3,x_5\}.
$$
See Figure \ref{fig:graph1}.
{Consider the presheaf $F=\eE_{(\Sigma,\ZZ_3)}$,
and define $p\in \on{Emp}(F)$ by specifying its values on the maximal simplices of $\Sigma$:
\[
p_{\{x_1,x_2\}}^{a}=
\begin{cases}
\frac14 & \text{if } a\in \{(0,0),(1,0)\},\\
\frac12 & \text{if } a=(2,2),\\
0 & \text{otherwise,}
\end{cases}
\qquad
p_{\{x_2,x_3\}}^{a}=
\begin{cases}
\frac14 & \text{if } a\in \{(2,0),(2,1)\},\\
\frac12 & \text{if } a=(0,2),\\
0 & \text{otherwise,}
\end{cases}
\]
and
\[
p_{\{x_3,x_5\}}^{a}=
\begin{cases}
\frac14 & \text{if } a\in \{(0,1),(1,0)\},\\
\frac12 & \text{if } a=(2,2),\\
0 & \text{otherwise,}
\end{cases}
\]
while
\[
p_{\{x_1,x_4\}}^{a}
=
p_{\{x_3,x_4\}}^{a}
=
p_{\{x_1,x_5\}}^{a}
=
\begin{cases}
\frac14 & \text{if } a\in \{(0,0),(1,1)\},\\
\frac12 & \text{if } a=(2,2),\\
0 & \text{otherwise.}
\end{cases}
\]
}

%\begin{align*}
%p_{\set{x_1,x_2}}^{(0,0)}=p_{\set{x_1,x_2}}^{(1,0)} &= \frac{1}{4},
%\qquad
%p_{\set{x_1,x_2}}^{(2,2)}=\frac{1}{2}, \\
%p_{\set{x_2,x_3}}^{(2,0)}=p_{\set{x_2,x_3}}^{(2,1)}&=\frac{1}{4},
%\qquad
%p_{\set{x_2,x_3}}^{(0,2)}=\frac{1}{2}, \\ 
%p_{\set{x_1,x_4}}^{(0,0)}=p_{\set{x_1,x_4}}^{(1,1)}&=\frac{1}{4},\qquad
%p_{\set{x_1,x_4}}^{(2,2)}= \frac{1}{2}, \\
%p_{\set{x_3,x_4}}^{(0,0)}=p_{\set{x_3,x_4}}^{(1,1)}
%&=\frac{1}{4},\qquad
%p_{\set{x_3,x_4}}^{(2,2)}=\frac{1}{2}, \\
%p_{\set{x_1,x_5}}^{(0,0)}=p_{\set{x_1,x_5}}^{(1,1)}
%&=\frac{1}{4},\qquad p_{\set{x_1,x_5}}^{(2,2)}
%=\frac{1}{2}, \\
%p_{\set{x_3,x_5}}^{(0,1)}=p_{\set{x_3,x_5}}^{(1,0)} &= \frac{1}{4}, \qquad 
%p_{\set{x_3,x_5}}^{(2,2)}
%=
%\frac{1}{2}.
%\end{align*}
%
We show that $p$ is a vertex satisfying the conditions of Theorem~\ref{thm:1vert}.  
The following diagram contains all maximal simplices of $\Sigma$:
\begin{equation}\label{dia:thediagofSigma}
\begin{tikzcd}[column sep=small,row sep=small]
&& \set{x_1} \arrow[ld,hook]  \arrow[rd,hook] \arrow[d,hook] &&  \\
& \set{x_1,x_5} & \set{x_1,x_2} & \set{x_1,x_4} &   \\
\set{x_5} \arrow[ru,hook] \arrow[rd,hook]  &&&& \set{x_4}  \arrow[lu,hook] \arrow[ld,hook]  \\
& \set{x_3,x_5}   & \set{x_2,x_3} & \set{x_3,x_4} & \\
&&  \set{x_3} \arrow[lu,hook]  \arrow[ru,hook] \arrow[u,hook]  &&
\end{tikzcd}
\end{equation}
Applying the functor $\eta_F(\kappa_F(p))$ to Diagram (\ref{dia:thediagofSigma}) yields a diagram whose arrows are all isomorphisms.

%\ak{Aziz: Fix since the theorem is not if and only if} 
{
%Moreover, $p$ is the unique 
%empirical model mapped by $\kappa_F$ to $\kappa_F(p)$. Hence, by Corollary~\ref{cor:pvertexifand}, $\Phi(p)$ is a vertex.
Now let $G\colon {\catC_{\Sigma}^\op} \to \catSet$  be an event scenario satisfying $G \leq \eta_F(\kappa_F(p))$.
Since $
\emptyset \neq G({x_2}) \subseteq  \eta_F(\kappa_F(p))(\set{x_2})=\set{0,2}
$, we may assume that $0 \in G(\set{x_2})$; the argument is identical if we instead start with $2\in G(\{x_2\})$. 
By the local surjectivity of $G$, 
$$
(0,2) \in G(\set{x_2,x_3}),
$$ 
and hence $2 \in G(\set{x_3})$. Applying local surjectivity repeatedly, we obtain
$$
(2,2) \in G(\set{x_3,x_4}) \Rightarrow 2 \in G(\set{x_4}) \Rightarrow (2,2) \in G(\set{x_1,x_4}) \Rightarrow 2 \in G(\set{x_1}) \Rightarrow  
(2,2) \in G(\set{x_1,x_2}).
$$
Therefore, $2 \in G(\set{x_2})$. It follows that at least one of
$(2,0),(2,1)$ belongs to $G(\{x_2,x_3\})$. Without loss of generality, assume that
$$
(2,0)\in G(\set{x_2,x_3}).
$$ 
Then $0\in G(\set{x_3})$. Applying the same argument once more, we obtain
$$
(0,1)\in G(\{x_3,x_5\})
\Rightarrow
1\in G(\{x_5\})
\Rightarrow
(1,1)\in G(\{x_1,x_5\})
\Rightarrow
1\in G(\{x_1\}),
$$
and therefore
$$
(1,1)\in G(\{x_1,x_4\})
\Rightarrow
1\in G(\{x_4\})
\Rightarrow
(1,1)\in G(\{x_3,x_4\})
\Rightarrow
1\in G(\{x_3\}).
$$
So we conclude that 
$$
(2,1) \in G(\set{x_2,x_3}).
$$
Consequently,
$$
G(\{x_2,x_3\})
=
\eta_F(\kappa_F(p))(\{x_2,x_3\}).
$$
Since $\{x_2,x_3\}$ is a maximal simplex of $\Sigma$, the hypotheses of Theorem~\ref{thm:1vert} are satisfied. Therefore, $p$ is a vertex.
}

For the singular realization, fix the order
$x_1<x_2<x_3<x_4<x_5$, and let $f:=f_F$. Set
\[
E' := \zeta_f\bigl(\kappa_f(\Phi_F(p))\bigr),
\qquad
g := f|_{E'}.
\]
Not all simplices in $E'$ are $g$-strongly connected. Indeed,
otherwise Proposition~\ref{pro:uniqeextension} would imply that all
probability values are equal, which is not the case.
\end{ex}
%
%\
Theorem~\ref{thm:1vert} detects more extremal distributions than
Theorem~\ref{thm:strongisvertx}, but it still does not detect all vertices.
The following example illustrates this.

\begin{ex}
Consider the {presheaf} $F=\eE_{(\Sigma,\ZZ_3)}$, where $\Sigma$ is
the bipartite graph with vertex set
\[
\set{x_1,x_2,y_1,y_2,y_3,y_4}
\]
and edge set
\[
\set{x_1,y_1},\ \set{x_1,y_2},\ \set{x_1,y_3},\ \set{x_1,y_4},\ 
\set{x_2,y_1},\ \set{x_2,y_2},\ \set{x_2,y_3},\ \set{x_2,y_4}.
\]
We define the distribution
$p\in \on{Emp}(\eE_{(\Sigma,\ZZ_3)})$ by setting
\begin{align*}
p_{\set{x_1,y_1}}^{a}
=
p_{\set{x_1,y_2}}^{a}
=
p_{\set{x_1,y_3}}^{a}
=
p_{\set{x_1,y_4}}^{a}
&=
\begin{cases}
\frac{1}{2} & \text{if } a=(0,0), \\
\frac{1}{4} & \text{if } a\in \set{(1,2),(2,1)}, \\
0 & \text{otherwise,}
\end{cases}
\\[1em]
p_{\set{x_2,y_1}}^{a}
&=
\begin{cases}
\frac{1}{4} & \text{if } a\in \set{(0,0),(0,1),(1,2),(2,0)}, \\
0 & \text{otherwise,}
\end{cases}
\\[1em]
p_{\set{x_2,y_2}}^{a}
&=
\begin{cases}
\frac{1}{4} & \text{if } a\in \set{(0,0),(0,2),(1,0),(2,1)}, \\
0 & \text{otherwise,}
\end{cases}
\\[1em]
p_{\set{x_2,y_3}}^{a}
&=
\begin{cases}
\frac{1}{4} & \text{if } a\in \set{(0,0),(0,2),(1,1),(2,0)}, \\
0 & \text{otherwise,}
\end{cases}
\\[1em]
p_{\set{x_2,y_4}}^{a}
&=
\begin{cases}
\frac{1}{4} & \text{if } a\in \set{(0,1),(0,2),(1,0),(2,0)}, \\
0 & \text{otherwise.}
\end{cases}
\end{align*}
This is the presheaf version of the vertex given in
\cite[Example~4.4]{kharoof2026vertex}. However, the hypotheses of
Theorem~\ref{thm:1vert} are not satisfied, since
\[
\left|\eta_F(\kappa_F(p))(\set{x_1,y_1})\right|=3,
\qquad
\left|\eta_F(\kappa_F(p))(\set{x_2,y_1})\right|=4.
\]
\end{ex}

%{===}

\bibliography{bib.bib}
\bibliographystyle{ieeetr}

\appendix

%\section{The Grothendieck and the relative Grothendieck constructions}
\section{Relative Grothendieck construction}
\label{sec:gro}

This section recalls the usual Grothendieck construction~\cite{mac2013categories} and presents a relative version, introduced in~\cite{kharoof2025simplicial}, that extends to $2$-categories~\cite[Section~XII.3]{mac2013categories}.

\begin{defn}\label{def:Grothconst}
Let $\catC$ be a category, and let $\catCat$ be the category of locally small categories.
Given a functor $F\colon \catC\to \catCat$, its \emph{Grothendieck construction} is the category $\int_{\catC} F$ defined as follows:
\begin{itemize}
\item The objects are pairs $(c,x)$ with $c\in\catC$ and $x\in F(c)$.
\item A morphism $(c,x)\to(d,y)$ consists of a morphism $h\colon c\to d$ in $\catC$ together with a morphism $\gamma\colon F(h)(x)\to y$ in $F(d)$.
\end{itemize}
\end{defn}
The Grothendieck construction is functorial; that is, it defines a functor from {the functor category}
$[\catC,\catCat]$ to $\catCat$.
In addition, there is a canonical projection
$$
\int_{\catC} F\to \catC
$$
which forgets the second component.
There is also a contravariant version of the Grothendieck construction for a functor $F\colon \catC^{\op}\to\catCat$. The objects are again pairs $(c,x)$, while a morphism $(c,x)\to(d,y)$ now consists of a morphism $h\colon d\to c$ in $\catC$ and a map $\gamma\colon F(h)(x)\to y$ in $F(d)$.

%\begin{pro}\label{pro:equileadsisomo}
%Given functors $F,G\colon \catC^\op \to \catCat$ with a natural isomorphism $F\Rightarrow G$. Then 
%there is an equivalence between the associated Grothendieck constructions
%$$
%\int_{\catC^{\op}} F \xrightarrow{\simeq} \int_{\catC^{\op}} G.
%$$
%\end{pro} 

%An important special case of the (contravariant) Grothendieck construction is when the functor $F$ takes values in sets instead of categories. The Grothendieck construction of such a functor $F\colon\catC\to \catSet$ is called the \emph{category of elements}. We will write $\catC_F$ for this category. Its objects consists of pairs $(c,x)$ where $c$ is an object of $\catC$ and $x\in F(c)$. A morphism $(c,x)\to (d,y)$ in this category is given by a morphism ${h}\colon c\to d$ of $\catC$ such that $F(h)(x)=y$.  

%We next introduce a version of the Grothendieck construction adapted to $2$-categorical settings.
A \emph{$2$-category} consists of:
\begin{itemize}
\item objects,
\item $1$-morphisms between objects,
\item $2$-morphisms between $1$-morphisms.
\end{itemize}
Our main example is $\catCat$, whose $1$-morphisms are functors and whose $2$-morphisms are natural transformations.
\begin{defn}\label{def:Sli2cat} 
Let $\mathcal{K}$ be a $2$-category and let $c$ be an object of $\mathcal{K}$.
The \emph{thick slice category} $\mathcal{K}\slice c$ is defined as follows:
\begin{itemize}
\item Objects are morphisms $f\colon a\to c$ in $\mathcal{K}$.
\item A morphism from $f\colon a\to c$ to $g\colon b\to c$ is a pair $(h,\eta)$, where $h\colon a\to b$ is a $1$-morphism in $\mathcal{K}$ and $\eta\colon f\Rightarrow g\circ h$ is a $2$-morphism; see Diagram~\eqref{dia:mapinSlice}.
\end{itemize}
\begin{equation}\label{dia:mapinSlice}
    \begin{tikzcd}[column sep=huge,row sep=large]
a
\arrow[rr,"h"]
\arrow[dr,"f"',""{name=A,right}] && b
\arrow[dl,"g",""{name=B,left}] 
 \arrow[Rightarrow, from=A, to=B, "\eta"]\\
&  c &  
\end{tikzcd}
\end{equation}
We define the functor $\Pi\colon\mathcal{K}{\slice} c \to \mathcal{K}$ by sending $f\colon a\to c$ to $a$ and $(h,\eta)$ to $h$. 
\end{defn}
\begin{defn}\label{def:GenGroth}
Let $\catC$ {be a ctegory, let $\catE$ be a small category}, and let 
$
F\colon \catC \to \catCat{\slice}\catE
$
be a functor. Write $\bar{F}$ for the composite $\Pi\circ F\colon\catC\to\catCat$.
The \emph{relative Grothendieck construction} is the functor
\[
\int_{\catC}F\colon \int_{\catC} \bar{F} \to \catE,
\] 
defined as follows:
\begin{itemize}
\item For an object $(c,x)$ of $\int_{\catC}\bar{F}$, where $x$ is an object of $\bar{F}(c)$, set
\[
(\int_{\catC} F)(c,x)= F(c)(x).
\]
\item For a morphism $(h,\gamma)\colon(c,x)\to(d,y)$ in $\int_{\catC}\bar{F}$, where {$h \colon c \to d$ in $\catC$ and}
$\gamma\colon \bar{F}(h)(x)\to y$ lies in $\bar{F}(d)$, suppose that $F(h)=(\bar{F}(h),\eta)$. Then we define
\[
(\int_{\catC} F)(h,\gamma)= F(d)(\gamma) \circ \eta_x.
\]
\end{itemize}
\end{defn}

The composite above has the form
\[
F(c)(x) \xrightarrow{\eta_x} F(d)(\bar{F}(h)(x))
\xrightarrow{F(d)(\gamma)} F(d)(y),
\]
where $\eta$ is a natural transformation $F(c)\to F(d) \circ \bar{F}(h)$, and $F(d)$ is a functor $\bar{F}(d)\to \catE$.

The relative Grothendieck construction is functorial; that is, it defines a functor
$$
[\catC,\catCat \slice \catE] \longrightarrow \catCat \slice \catE.
$$
As a consequence, we obtain the following result.

\begin{pro}\label{pro:NatNat}
A natural isomorphism between functors
$$
F\colon \catC \longrightarrow \catCat \slice \catE
\quad \text{and} \quad 
G\colon \catC \longrightarrow \catCat \slice \catE
$$
induces a natural isomorphism between the corresponding relative Grothendieck constructions
$$
\int_{\catC} F \quad \text{and} \quad \int_{\catC} G .
$$
\end{pro}
% $\epsilon, \mu,\xi, \kappa$.

\end{document}